\documentclass[12pt,a4paper]{article}
\usepackage{amsmath}
\usepackage{amsthm}
\usepackage{amsfonts}
\usepackage{amssymb}
\usepackage{stmaryrd}
\usepackage{latexsym}

\allowdisplaybreaks

\newtheorem{theorem}{Theorem}[section]
\newtheorem{lemma}{Lemma}[section]
\newtheorem{corollary}{Corollary}[section]

\theoremstyle{definition}
\newtheorem{definition}{Definition}[section]
\newtheorem{remark}{Remark}[section]
\newtheorem{example}{Example}[section] 

\begin{document}



\title{\large{\textbf{On the structure of  left and right F-, SM- and E-quasigroups}}}
\author{\normalsize {V.A. Shcherbacov}}
 \maketitle

\textbf{Abstract.} {\footnotesize It is proved that any  left F-quasigroup is isomorphic to the direct product
of a left F-quasigroup with a unique idempotent element  and isotope of a special form of a left distributive
quasigroup. The similar theorems are proved for right F-quasigroups, left and right SM- and E-quasigroups.

Information on simple quasigroups from  these  quasigroup classes is given, for example, finite simple
F-quasigroup is a simple group or a simple medial quasigroup.

It is proved that any left F-quasigroup is isotopic  to the direct product of a group and a left S-loop. Some
properties of loop isotopes of F-quasigroups (including M-loops) are pointed out. A left special loop  is an
isotope of a left F-quasigroup  if and only if  this loop  is isomorphic the direct product of a group and a
left S-loop (this is an answer to  Belousov \lq\lq 1a\rq\rq\, problem).

Any left FESM-quasigroup is isotopic to the direct product of an abelian group and a left S-loop (this is an
answer to Kinyon-Phillips 2.8(2) problem). New proofs of some known results on the structure of commutative
Moufang loops are presented. }

\bigskip


\bigskip


\tableofcontents

\section{Introduction}

  D.C.~Murdoch  introduced   F-quasigroups in  \cite{MURD_39}. At  this time A.K.~Sushkevich studied
 quasigroups with the weak associative properties \cite{SUSHKEV, SUSHKEV_BOOK}. Their name F-quasigroups obtained
 in an article  of V.D.~Belousov \cite{PROP_BIN_OPERATS}. Later Belousov and his pupils I.A.~Golovko
 and  I.A.~Florja,   M.I.~Ursul, T. Kepka, M. Kinyon,
 J.D.~Phillips, L.V.~Sabinin, L.V.~Sbitneva, L.L.~Sabinina and many other mathematicians   studied F-quasigroups
 and left F-quasigroups
 \cite{VD, 1a,vdb_FLOR, GOLOVKO_I, GOLOVKO_II, FLOR_URS, FLOR_71,  FLOR_1973, kepka79, SAB_SAB, SAB_SBIT,Cheban_AM,  kepka_05}. In
\cite{kepka_05, kepka_07, kepka_08} it is proved that any F-quasigroup is linear over a Moufang loop.  The
structure of F-quasigroups also is described in \cite{kepka_05, kepka_07,kepka_08}.

Left and right SM-quasigroups (semimedial quasigroups) are defined by T. Kepka. In  \cite{kepka75_I}  Kepka has
called these quasigroups LWA-quasigroups and  RWA-quasigroups, respectively. SM-quasigroups are connected with
trimedial quasigroups. These quasigroup classes are studied in \cite{kepka75_I, kepka76, kepka78, 12,
BEGLARYAN_85, SHCHUKIN_86, Kin_PHIL_02, Kin_PHIL_04}. M.~Kinyon and J.D.~Phillips have defined and studied left
and right E-quasigroups  \cite{Kin_PHIL_04}.

Main idea of this paper is to use quasigroup endomorphisms by the   study  of structure of quasigroups with some
generalized distributive identities.
   This idea has  been  used by the study of many loop and quasigroup classes, for example by the study
of commutative Moufang loops, commutative diassociative loops, CC-loops (LK-loops),  F-quasigroups,
SM-quasigroups, trimedial quasigroups and so on \cite{RHB, HOP, VD, BRUCK_56, BRUCK_60, BEGLARYAN1, BASARAB_91,
Kin_KUN_05, Kin_KUN_06,  KIN_VOITEH_07, PHILL_00}. Especially clear this idea is expressed in K.K. Shchukin's
book \cite{12}.

 Using language of identities of quasigroups with three operations in signature, i.e. of quasigroups of the form
$(Q, \cdot, \slash, \backslash)$,  we can say that we study some quasigroups from the following quasigroup
classes: (i) $(xy)\backslash (xy) = (x \backslash x) \cdot (y \backslash y)$; (ii) $(xy)\slash (xy) = (x \slash
x) \cdot (y \slash y)$; (iii)  $(xy)\cdot (xy) = (x  x) \cdot (y  y)$.

This paper is connected with  the following problems.

\smallskip

\noindent  \textbf{Problem 1.}  (Belousov  Problem 1a, \cite{VD}, \cite{kepka_05, TWENTY_BEL}) Find necessary
and sufficient conditions that a left special loop is isotopic to a left F--quasigroup.

Problem 1a has been  solved partially by I.A. Florea and  M.I. Ursul \cite{FLOR_71, FLOR_URS}. They proved that
a left F--quasigroup with IP-property is isotopic to an A-loop.

\noindent \textbf{Problem 2.} (Problem 2.8 from \cite{Kin_PHIL_04}) (1) Characterize the loop isotopes of
quasigroups satisfying ($E_l$).

(2) Characterize the loop isotopes of quasigroups satisfying ($E_l$), ($S_l$), and ($F_l$).

\noindent \textbf{Problem 3.} It is easy to see that in loops $1 \cdot ab = 1a \cdot 1b$. Describe quasigroups
with the property $f(ab) = f(a)f(b)$ for all $a, b \in  Q$, where $f(a)$ is left local identity element of $a$
(\cite{SCERB_03}, p. 12).

Results of this paper  were presented at the conference LOOPS'07 (August 19--24, 2007, Prague). In order to make
the reading of this paper more or less  easy we give some necessary preliminary results and quit detailed
proofs.

\textbf{Acknowledgment.} The author thanks MRDA-CRDF (ETGP, grant no. 1133), Consiliul Suprem pentru \c Stiin\c
t\u a  \c si Dezvoltare Tehnologic\u a al Republicii Moldova (grant 08.820.08.08 RF)  and  organizers of the
conference LOOPS'07 for financial support.  The author thanks Prof.~V.I.~Arnautov for his helpful comments.

\subsection{Quasigroups}

Let  $(Q,\cdot)$ be a groupoid (be a magma in alternative terminology).  As usual, the map $L_{a}: Q\rightarrow
Q, L_{a}x = a\cdot x$ for all $x\in Q$, is a left translation of the groupoid  $(Q,\cdot)$ relative to  a fixed
element $a\in Q$, the map $R_{a}: Q\rightarrow Q $, $R_{a} x = x\cdot a$,  is a right translation.

\begin{definition}
A groupoid $(G,\cdot)$ is said to be a \textit{division groupoid} if the  mappings $L_x$  and  $R_x$  are
surjective for every $x\in G$.
\end{definition}

In a division groupoid $(G,\cdot)$ any from equations $a\cdot x = b$ and $y\cdot a = b$ has at least one
solution for any fixed $a, b \in Q$, but we cannot guarantee that these solutions are unique solutions.

\begin{definition}
A groupoid $(G,\cdot)$ is said to be a \textit{cancellation groupoid} if
\[
a\cdot b = a \cdot c   \Longrightarrow b = c, \qquad b \cdot a = c \cdot a  \Longrightarrow b = c\] for all $a,
b, c \in G$.
\end{definition}

 If any from equations $a\cdot x = b$ and $y\cdot a = b$ has a solution
in a cancellation groupoid $(G,\cdot)$ for some  fixed $a, b \in Q$, then this solution is unique. In other
words, in a cancellation groupoid  the mappings $L_x$ and $R_x$ are injective for every $x\in G$.

\begin{definition}
A groupoid $(Q,\cdot )$ is called a {\it quasigroup} if, for all $a,b \in Q$, there exist  unique solutions $x,
y  \in Q$ to the equations $x\cdot a = b$ and $a\cdot  y = b,$ i.e. in this case any right and any left
translation of the groupoid $(Q,\cdot )$ is a bijection  of the set $Q$.
\end{definition}

\begin{remark}\label{remark_one}
Any division cancellation groupoid is a quasigroup and vice versa.
\end{remark}

A  sub-object $(H,\cdot)$  of a  quasigroup $(Q, \cdot)$ is closed relative to the operation $\cdot$, i.e., if
$a, b \in H$, then $a\cdot b \in H$.

We denote by $S_Q$ the group of all bijections (permutations in finite case) of a set $Q$.

\begin{definition}
A groupoid  $(Q, A)$ is an \textit{isotope of a groupoid  $(Q, B)$} if there exist permutations  $\mu_1, \mu_2,
\mu_3$ of the set $Q$ such that $A(x_1, x_2) = \mu_3^{-1} B(\mu_1 x_1,\mu_2 x_2)$ for all $x_1, x_2 \in Q.$ We
also can say that a groupoid $(Q, A)$ is an \textit{isotopic image of a groupoid  $(Q, B)$}. The triple $(\mu_1,
\mu_2, \mu_3)$ is called an \textit{isotopy (isotopism)}.

We shall write this fact also in the form $(Q,A) = (Q,B)T$, where $T = (\mu_1, \mu_2,\mu_3)$  \cite{VD, 1a,
HOP}.

If only the fact will be important that binary groupoids $(Q,\circ)$ and $(Q,\cdot)$ are isotopic, then we shall
use the record $(Q,\cdot)\sim (Q,\circ)$.
\end{definition}
\index{Groupoid!isotopy (isotopism)} \index{Isotopy (isotopism)}

\begin{definition}
Isotopy of the form $(\mu_1, \mu_2,\varepsilon)$ is called a \textit{principal isotopy}.
\end{definition}

\begin{remark} \label{REMARK_ON_PRINC_ISOT}
Up to isomorphism any isotopy  is  a \textit{principal isotopy}. Indeed, $T = (\mu_1,$ $ \mu_2, $ $\mu_3) =
(\mu_1\mu^{-1}_3, \mu_2 \mu^{-1}_3,\varepsilon)(\mu_3, \mu_3,\mu_3).$
\end{remark}

There exist the following definition of a quasigroup.

\begin{definition} \label{def1}  A binary  groupoid $(Q, A)$  such that in the equality $A(x_1,$ $
x_2) = x_3$ knowledge of any $2$ elements of $x_1, x_2,x_3$ uniquely specifies the remaining one is called a
\textit{binary quasigroup}  \cite{DORNTE, 2, MUFANG}. \end{definition}

From Definition \ref{def1} it follows that  with any quasigroup $(Q, A)$ it possible to associate else
$(3!-1)=5$ quasigroups, so-called parastrophes of quasigroup $(Q, A)$:
 $A(x_1, x_2) = x_3 \Leftrightarrow A^{(12)}(x_2, x_1) = x_3
\Leftrightarrow {A}^{(13)}(x_3, x_2) = x_1 \Leftrightarrow {A}^{(23)}(x_1, x_3) = x_2 \Leftrightarrow
{A}^{(123)}(x_2, x_3) = x_1 \Leftrightarrow {A}^{(132)}(x_3, x_1) = x_2.$ We shall denote:

the operation of $(12)$-parastrophe of a quasigroup $(Q, \cdot)$  by $\ast$;

the operation of $(13)$-parastrophe of a quasigroup $(Q, \cdot)$  by $/$;

the operation of $(23)$-parastrophe of a quasigroup $(Q, \cdot)$  by $\backslash$;

 the operation of $(123)$-parastrophe of a quasigroup
$(Q, \cdot)$  by $\slash\slash$;

the operation of $(132)$-parastrophe of the  quasigroup $(Q, \cdot)$  by $\backslash\backslash$.

We have defined left and right translations of a groupoid and, therefore, of  a quasigroup. But for  quasigroups
it is possible to define   the third kind of translations. If $(Q,\cdot)$ is a quasigroup, then the map $P_a :
Q\longrightarrow Q$, $x\cdot P_{a}x = a$  for all $x \in Q$, is called a \textit{middle translation}
\cite{BELAS, SCERB_07}.

In the following table   connections between different kinds of  translations in different parastrophes of a
quasigroup   $(Q,\cdot)$ are given. This table in fact  there is in  \cite{BELAS}. See also \cite{DUPLAK,
SCERB_91}.

\hfill Table 1
\[
\begin{array}{|c||c| c| c| c| c| c|}
\hline
Kinds  & \varepsilon = \cdot & (12) = \ast & (13)= / & (23) = \backslash  & (123)= // & (132) = \backslash\backslash \\
\hline\hline
 {R}  &  {R} &  {L} &  {R}^{-1} &  {P} &  {P}^{-1} &  {L}^{-1} \\
\hline
 {L}  &  {L} &  {R} &  {P}^{-1} &  {L}^{-1} &  {R}^{-1} &  {P} \\
\hline
 {P}  &  {P} &  {P}^{-1} &  {L}^{-1} &  {R} &  {L} &  {R}^{-1} \\
\hline
 {R}^{-1} &  {R}^{-1} &  {L}^{-1} &  {R} &  {P}^{-1} &   {P} &  {L} \\
\hline
 {L}^{-1}  &  {L}^{-1} &  {R}^{-1} &  {P} &  {L} &  {R} &  {P}^{-1} \\
\hline
 {P}^{-1}  &  {P}^{-1} &  {P} &  {L} &  {R}^{-1} &  {L}^{-1} &  {R} \\
\hline
\end{array}
\]

In Table 1, for example, ${R}^{(23)} = {R}^{\backslash } =  {P}^{(\cdot)}.$

If $T = (\alpha_1, \alpha_2, \alpha_3)$ is an isotopy, $\sigma$ is a parastrophy, then we define $T^{\sigma} =
(\alpha_{\sigma^{-1} 1}, \alpha_{\sigma^{-1} 2}, \alpha_{\sigma^{-1} 3})$.

\begin{lemma} \label{L2.1} In a quasigroup $(Q,A)$:   ${(AT)}^\sigma  = {A}^\sigma T^\sigma,$ $(T_1T_2)^{\sigma} = T_1^{\sigma}
T_2^{\sigma}$ \cite{VD, 2}.
\end{lemma}

\begin{definition} An element $f(b)$ of a quasigroup $(Q,\cdot)$ is called  left local identity element of an
 element $b\in Q$, if $f(b)\cdot b = b$, in other words, $f(b) = b / b$.

 An element $e(b)$ of a quasigroup $(Q,\cdot)$ is called   right  local
identity element of an
 element $b\in Q$, if $b\cdot e(b) = b$, in other words, $e(b) = b \backslash b$.

An element $s(b)$ of a quasigroup $(Q,\cdot)$ is called   middle  local identity element of an element $b\in Q$,
if $b\cdot b = s(b)$ \cite{SC_89_1, SCERB_91}. \index{identity element!middle  local}

An element $e$ is a left (right) identity element for quasigroup $(Q,\cdot)$ means that $e = f(x)$ for all $x\in
Q$ (respectively, $e= e(x)$ for all $x\in Q$). \index{identity element!left} A quasigroup with the  left (right)
identity element will be called  a \textit{left (right) loop}.

The fact that an element $e$ is an \textit{identity element} of a quasigroup $(Q,\cdot)$ means that $e(x) = f(x)
= e$ for all $x\in Q$, i.e. all left and right local identity elements in the quasigroup $(Q,\cdot)$ coincide
\cite{VD}.
\end{definition}

Connections between different kinds of local identity elements  in different parastrophes of a quasigroup
$(Q,\cdot)$ are given in the following table  \cite{SC_89_1, SCERB_91}.

\hfill Table 2 \label{TABLE_two}
\[
\begin{array}{|c||c| c| c| c| c| c|}
\hline
  & \varepsilon = \cdot & (12) = \ast & (13)= / & (23) = \backslash  & (123)= // & (132) = \backslash\backslash \\
\hline\hline
 {f}  &  {f} &  {e} &  {s} &  {f} &  {s} &  {e} \\
\hline
 {e}  &  {e} &  {f} &  {e} &  {s} &  {f} &  {s} \\
\hline
 {s}  &  {s} &  {s} &  {f} &  {e} &  {e} &  {f} \\
\hline
\end{array}
\]

In Table 2, for example, ${s}^{(123)} =  {e}^{(\cdot)}.$

\begin{remark}
We notice that in \cite{BEGLARYAN_85, 12} the mapping $s$ is denoted by $\beta$.
\end{remark}

\begin{definition} A quasigroup $(Q,\cdot)$ with an identity  element $e \in Q$  is called a {\it loop}.
\end{definition}

Quasigroup isotopy of the form $(R^{-1}_a, L^{-1}_b, \varepsilon)$ is called an LP-isotopy. Any LP-isotopic
image of a quasigroup is a loop \cite{VD, 1a}.

\begin{lemma} \label{LP_AND_PRINCIP_ISOT}
Let $(Q,+)$ be a loop and $(Q, \cdot)$ be a quasigroup. If $(Q,+) =(Q, \cdot)(\alpha, \beta, \varepsilon)$, then
$(\alpha, \beta, \varepsilon) = (R^{-1}_a, L^{-1}_b, \varepsilon)$ for some translations of $(Q, \cdot)$
(\cite{1a}, Lemma 1.1).
\end{lemma}

\begin{lemma} \label{LP_AND_subquas}
If $(Q,\cdot)$ is a  quasigroup,   $(H,\cdot)$  is its subquasigroup, $a, b \in H$, then $(H,\cdot)T$ is a
subloop of the loop  $(Q,\cdot)T$, where $T$ is an isotopy of the form $(R^{-1}_a, L^{-1}_b, \varepsilon)$.
\end{lemma}
\begin{proof}
We have that $R_a|_H, L_b|_H$ are translations of $(H,\cdot)$, since $a, b\in H$.
\end{proof}

We define the following mappings of a quasigroup $(Q, \cdot)$:  $f : x \mapsto f(x)$, $f(x)\cdot x = x$ for all
$x \in Q$; $e : x \mapsto e(x)$, $x\cdot e(x) = x$ for all $x \in Q$; $s : x \mapsto s(x)$, $s(x)= x\cdot x$ for
all $x \in Q$.

\medskip

\begin{definition} \label{quasigr_as_algebra}   An algebra $(Q, \cdot, \backslash, /)$
  is called a quasigroup, if on the set $Q$ there exist
operations "$\backslash$" and "$/$" such that in  $(Q, \cdot, \backslash, /)$  identities
\begin{equation}
x\cdot(x \backslash y) = y, \label{(1)}
\end{equation}
\begin{equation}
(y / x)\cdot x = y, \label{(2)}
\end{equation}
\begin{equation}
x\backslash (x \cdot y) = y \label{(3)},
\end{equation}
\begin{equation}
(y \cdot x)/ x = y \label{(4)}
\end{equation}
 are fulfilled \cite{EVANS, BIRKHOFF, BURRIS, VD, 1a, HOP, SCERB_07}. \label{def3}
\end{definition}

\begin{lemma} \label{SUB_GROUPOIDS}
1.  Any sub-object of a  quasigroup $(Q, \cdot)$ is a cancellation groupoid.

2.  Any sub-object of a  quasigroup $(Q, \cdot, \backslash, /)$ is a subquasigroup.

3. Any subquasigroup  of a  quasigroup $(Q, \cdot)$ is a a subquasigroup in $(Q, \cdot, \backslash, /)$ and vice
versa, any subquasigroup  of a  quasigroup $(Q, \cdot, \backslash, /)$ is a a subquasigroup in $(Q, \cdot)$.
\end{lemma}
\begin{proof}
1. If $a,b,c \in H$, then from $a\cdot b = a \cdot c$  follows $b=c$, since $(H, \cdot) \subseteq  (Q, \cdot)$.
Similarly from $b\cdot a = c\cdot a$  follows $b=c$.

2 and 3.  See, for example, \cite{VD, HOP, MALTSEV, BURRIS}. \end{proof}

Left, middle and right nucleus  of a loop $(Q, \cdot)$ are  defined in the following way:
\begin{equation*}
\begin{split}
& N_l = \{ a \in Q \,|\, a \cdot xy = ax \cdot y, x, y \in Q\},  N_m =\{ a \in Q \,|\, xa \cdot y = x \cdot a y,
x, y \in Q\},\\ & N_r =\{ a \in Q \,|\, xy \cdot a = x \cdot  ya, x, y \in Q\}.
\end{split}
\end{equation*}
Nucleus of a loop is defined in the following way $N= N_l\cap N_m \cap N_r$ \cite{RHB, VD}.  R.H.~Bruck defined
a center of a loop $(Q,\cdot)$ as $C (Q,\cdot) = N \cap Z$, where $Z = \{a \in Q \, \mid \, a\cdot x =x\cdot a
\,\, \forall x \in Q\}$. Information on quasigroup nuclei there is in \cite{SCERB_03}.

\subsection{Autotopisms}

\begin{definition} An autotopism (sometimes we shall call autotopism and as autotopy) is an isotopism of a
quasigroup $(Q,\cdot)$ into itself, i.e. a triple $(\alpha, \beta, \gamma)$ of permutations of the set $Q$ is an
autotopy if the equality $x\cdot y = \gamma^{-1}(\alpha x \cdot \beta y)$ is fulfilled for all $x, y \in Q$.
\end{definition} \index{autotopy(autotopism)}

\begin{definition}
 The third component of any autotopism  is called a \textit{quasiautomorphism}.
\end{definition}
\index{quasiautomorphism}

By $Top \, (Q, \cdot)$ we shall denote the group of all autotopies of a quasigroup $(Q, \cdot)$.

\begin{theorem} If  quasigroups $(Q, \cdot)$ and $(Q, \circ)$ are isotopic
with isotopy $T$, i.e. $(Q,\cdot) = (Q, \circ)T$, then $ Top\, (Q, \cdot) = T^{-1} Top \, (Q, \circ) T$
\cite{VD, 1a, 2}.
\end{theorem}

\begin{lemma}  \cite{VD, 1a} If $(Q,\cdot)$ is a loop, then any its autotopy has the form
$(R^{-1}_a, L^{-1}_b, \varepsilon)(\gamma, \gamma,\gamma)$.
\end{lemma}
\begin{proof} Let $T=(\alpha, \beta, \gamma)$ be an autotopy of a loop $(Q,\cdot)$, i.e. $\alpha x \cdot
\beta y = \gamma(x \cdot y)$. If we put $x=1$, then we obtain $\alpha 1 \cdot \beta y = \gamma y$, $\gamma =
L_{\alpha 1} \beta$, $\beta = L^{-1}_{\alpha 1} \gamma$. If we put $y=1$, then, by analogy, we obtain, $\alpha =
R^{-1}_{\beta 1} \gamma$.

Then $T=(R^{-1}_{\beta 1} \gamma, L^{-1}_{\alpha 1} \gamma, \gamma)= (R^{-1}_k, L^{-1}_d, \varepsilon)(\gamma,
\gamma, \gamma)$, where $\beta 1 =k$, $\alpha 1 =$ $ d$.
\end{proof}

We can obtain more detail information on  autotopies of a group, and, since autotopy groups of isotopic
quasigroups are isomorphic, on autotopies of quasigroups that are some group isotopes.

\begin{theorem} Any autotopy of a group $(Q,+)$ has the form
$$(L_a \, \delta, R_b \, \delta, L_a R_b \, \delta),$$ where
$L_a$ is a left translation of the group $(Q,+)$, $R_b$ is a right translation of this group, $\delta$ is an
automorphism of $(Q,+)$ \cite{1a}.
\end{theorem}

\begin{corollary} \label{RAVNYE_KOMPON_AVTOT}
1. If $L_a \delta = L_a R_b \delta,$ then $R_b = \varepsilon$.  2. If $R_b \delta =  L_a R_b \delta,$ then $L_a
= \varepsilon$. 3. If $L_a \delta = R_b \delta,$ then $a \in C(Q,+)$.
\end{corollary}
\begin{proof}
3. We have $ a + \delta x +  a + \delta y   = a + a  + \delta x + \delta y$,  $\delta x +  a  = a  + \delta x$
for all $x\in Q$.
\end{proof}

\begin{corollary} \label{QUASIAUT_FORM}
Any group quasiautomorphism has the form $L_d \, \varphi$, where $\varphi \in Aut(Q,+)$ \cite{vs2}.
\end{corollary}
\begin{proof}
We have $L_aR_b \, \delta x = a + \delta x + b = a + b - b + \delta x + b = L_{a+b} I_{b} \, \delta x = L_d \,
\varphi,$ where $d = a + b$, $\varphi = I_{b} \delta$, $I_{b}\, x = -b + x + b$.
\end{proof}

\begin{lemma} \label{ON AUTOMORPHISM_OF_IDEM_QUAS}
1. If $x\cdot  y = \alpha x \ast  y$,  where $(Q, \ast)$ is an idempotent quasigroup, $\alpha$ is a permutation
of the set Q, then $Aut(Q, \cdot) = C_{Aut(Q, \ast)} (\alpha) = \{\tau  \in Aut(Q, \ast) \, | \, \tau \alpha =
\alpha \tau \} $, in particular,  $Aut(Q, \cdot) \subseteq Aut(Q, \ast)$.

2. If $x\cdot  y = x \ast \beta y$,  where $(Q, \ast)$ is an idempotent quasigroup, $\beta$ is a permutation of
the set Q, then $Aut(Q, \cdot) = C_{Aut(Q, \ast)} (\beta) = \{\tau  \in Aut(Q, \ast) \, | \, \tau \beta = \beta
\tau \} $, in particular,  $Aut(Q, \cdot) \subseteq Aut(Q, \ast)$ (\cite{MARS}, Corollary 12).
\end{lemma}
\begin{proof}
1. We give a sketch of the proof. If $\varphi \in Aut(Q, \cdot)$, then  $\varphi (x\cdot y) =  \varphi(\alpha x
\ast
 y) = \varphi x\cdot \varphi y = \alpha \varphi x\ast  \varphi y.$ If $y = \alpha x$, then $\varphi \alpha
x = \alpha \varphi  x \ast  \varphi \alpha x,$ $\varphi \alpha =  \alpha \varphi $, $\varphi(\alpha x \ast  y) =
\varphi \alpha x\ast  \varphi  y $.

2. The proof of Case 2 is similar to the proof of Case 1.
\end{proof}

\subsection{Quasigroup classes}

\begin{definition} \label{DEF_11}
A quasigroup $(Q, \cdot)$ is

\noindent \textit{medial}, if  $xy\cdot uv = xu\cdot yv$ for all $x,y, u, v \in Q$;

\noindent  \textit{left distributive},  if  $x\cdot uv = xu\cdot xv$ for all $x, u, v \in Q$;

\noindent \textit{right  distributive},  if  $x u\cdot v = xv\cdot uv$ for all $x, u, v \in Q$;

\noindent \textit{distributive}, if it is left and right distributive;

\noindent \textit{idempotent},  if $x\cdot x = x$ for all $x \in Q$;

\noindent \textit{unipotent},  if there exists an element $a\in Q$  such that \index{Quasigroup!unipotent}
$x\cdot x = a$ for all $x\in Q;$

\noindent \textit{left semi-symmetric},  if  $x\cdot xy = y$ for all $x,y \in Q$;

\noindent \textit{TS-quasigroup},  if  \index{TS-quasigroup} $x\cdot xy = y, xy=yx$ for all $x,y \in Q$;

\noindent \textit{left F-quasigroup},  if  $x\cdot y z = x y \cdot e(x) z$ for all $x,y, z \in Q$;

\noindent \textit{right F-quasigroup},  if  $x y\cdot z = xf(z) \cdot y z$ for all $x,y, z \in Q$;

 \noindent
\textit{left semimedial  or middle F-quasigroup},  if  $s(x)\cdot y z =  x x \cdot y z = x y \cdot x z$ for all
$x,y, z \in Q$;

\noindent \textit{right semimedial},  if  $zy \cdot s(x) = zx \cdot yx$ for all $x,y, z \in Q$;

\noindent \textit{F-quasigroup}, if it is left and right F-quasigroup;

\noindent \textit{left E-quasigroup},  if  $x \cdot  yz = f(x)y \cdot xz $ for all $x,y, z \in Q$;

\noindent \textit{right  E-quasigroup},  if  $zy \cdot x = zx \cdot y e(x) $ for all $x,y, z \in Q$;

\noindent \textit{E-quasigroup}, if it is left and right E-quasigroup;

\noindent \textit{LIP-quasigroup}, if there exists a permutation $\lambda$ of the set $Q$ such that $\lambda x
\cdot (x\cdot y) = y$ for all $x,y \in Q$;

\noindent \textit{RIP-quasigroup}, if there exists a permutation $\rho$ of the set $Q$ such that $ (x\cdot y)
\cdot \rho y  = x $ for all $x,y \in Q$.

\noindent \textit{IP-quasigroup}, if if it is LIP- and RIP-quasigroup.
\end{definition}

A quasigroup $(Q, \cdot)$ of the form $x\cdot y = \varphi x + \beta y+c$, where $(Q, +)$ is a group,  $\varphi
\in Aut(Q,+)$, $\beta$ is a permutation of the set $Q$, is called a \textit{left linear quasigroup}; a
quasigroup $(Q, \cdot)$ of the form $x\cdot y = \alpha x + \psi y+c$, where $(Q, +)$ is a group,  $\psi  \in
Aut(Q,+)$, $\alpha$ is a permutation of the set $Q$,  is called a \textit{right linear quasigroup}
\cite{TABAR_92, SOH_07}.

\index{Loop!Bol}\index{Loop!Moufang}
\begin{definition} \label{LOOP_DEFS}
A loop  $(Q, \cdot)$ is

\noindent  \textit{Bol loop (left Bol loop)}, if  $x(y \cdot xz) = (x\cdot yx)z $ for all $x, y, z \in Q$;

\noindent  \textit{Moufang  loop}, if  $x(yz \cdot x) = xy \cdot zx$ for all $x, y, z \in Q$;

\noindent  \textit{commutative Moufang  loop (CML)}, if  $x x \cdot y z = x y \cdot x z$ for all $x, y, z \in
Q$;

\noindent \textit{left M-loop},  if  $ x\cdot (y \cdot z) = (x\cdot (y \cdot I\varphi x))\cdot (\varphi x \cdot
z)$ for all $x, y, z\in Q$,  where $\varphi$ is a mapping of the set $Q$, $x \cdot Ix = 1$ for all $x \in Q$;

\noindent \textit{right M-loop},  if  $ (y\cdot z) \cdot x = (y\cdot \psi x) \cdot ((I^{-1}\psi x \cdot z) \cdot
x)$ for all $x, y, z\in Q$, where $\psi$ is a mapping of the set $Q$.

\noindent \textit{M-loop}, if if it is left M- and right M-loop.

\noindent \textit{left  special},  if $S_{a,b}=L^{-1}_bL^{-1}_aL_{ab}$ is an automorphism of $(Q, \cdot)$ for
any pair $a,b\in Q$ \cite{Soikis_70}.

\noindent \textit{right   special},  if $T_{a,b}=R^{-1}_bR^{-1}_aR_{ba}$ is an automorphism of $(Q, \cdot)$ for
any pair $a,b\in Q$ \cite{Soikis_70}.
\end{definition}

In \cite{VD} the left special loop is called \textit{special}. In \cite{kepka76, 12, Kin_PHIL_04} left
semimedial quasigroups are studied.
 A quasigroup is trimedial if and only if it is  satisfies left and right E-quasigroup  equality \cite{Kin_PHIL_04}.
Information on properties of trimedial quasigroups there is in   \cite{Kin_PHIL_02}.

 Every semimedial quasigroup is isotopic to a commutative Moufang loop \cite{kepka76}.
 In the trimedial case the isotopy has a more restrictive form \cite{kepka76}.

In a quasigroup $(Q, \cdot, \backslash, \slash)$ the  equalities $x\cdot y z = x y \cdot e(x) z$,  $x y\cdot z =
xf(z) \cdot y z$, $x \cdot  yz = f(x)y \cdot xz $ and  $zy \cdot x = zx \cdot y e(x) $ take the form $x\cdot y z
= x y \cdot (x \backslash x) z$,  $x y\cdot z = x(z\slash z) \cdot y z$, $x \cdot  yz = (x\slash x) y \cdot xz $
and $zy \cdot x = zx \cdot y (x\backslash x)$,  respectively,  and they are identities in $(Q, \cdot,
\backslash, \slash)$.

Therefore  any subquasigroup of a left F-quasigroup $(Q, \cdot, \backslash, \slash)$ is a left F-quasigroup, any
homomorphic image of a left F-quasigroup $(Q, \cdot, \backslash, \slash)$ is a left F-quasigroup \cite{BURRIS,
MALTSEV}. It is clear that the same situation is  for right F-quasigroups, left and right E- and SM-quasigroups.

\begin{lemma} \label{medial_F_SM_E}
Any medial quasigroup $(Q, \cdot)$ is  both a left and right F-, SM- and E-quasigroup.
\end{lemma}
\begin{proof}
 Equality $x \cdot uv = xu\cdot e(x) v$ follows from medial identity $xy \cdot uv = xu\cdot yv$ by $y = e(x)$.
Respectively by $u = e(x)$ we have  $xy \cdot e(x) v = x\cdot yv$. I. e.  $(Q,\cdot)$ is a left F-quasigroup in
these cases. And so on.
\end{proof}






\begin{lemma} \label{distrib_F_SM_E}
1. Any left distributive  quasigroup $(Q, \cdot)$ is  a left  F-, SM- and E-quasigroup.

2. Any right distributive  quasigroup $(Q, \cdot)$ is a right  F-, SM- and E-quasigroup.
\end{lemma}
\begin{proof}
1. It is easy to see that $(Q, \cdot)$ is idempotent quasigroup. Therefore $x\cdot x = x\slash x = x \backslash
x = x$. Then $x\cdot yz = xy \cdot xz = ((x\slash x)\cdot y) \cdot xz =
 ((x\backslash x)\cdot y) \cdot xz = xy \cdot (x\slash x)z = xy \cdot (x\backslash x)z = x x\cdot yz$.

2. The proof of this case is similar to the proof of Case 1.
\end{proof}

\begin{lemma} \label{LEMMA_ON_GEN_F}
A quasigroup $(Q, \cdot)$ in which:
\begin{enumerate}
    \item  the equality $x\cdot y z = x y \cdot \delta(x) z$ is true for all $x, y, z
\in Q$, where  $\delta$ is a map of the set $Q$, is a left F-quasigroup \cite{1a};
    \item  the equality $x y\cdot z = x\delta(z) \cdot y z$ is true for all $x, y, z
\in Q$, where  $\delta$ is a map of the set $Q$, is a right  F-quasigroup;
    \item the equality $\delta(x)\cdot y z =   x y \cdot x z$ is true for all $x, y, z \in Q$, where  $\delta$ is a map of
the set $Q$, is a left semimedial quasigroup;
    \item the equality $zy \cdot \delta(x) = zx \cdot yx$ is true for all $x, y, z \in Q$, where  $\delta$ is a map of the
set $Q$, is a right  semimedial quasigroup;
    \item the equality $x \cdot  yz = \delta(x) y \cdot xz $ is true for all $x, y, z \in Q$, where  $\delta$ is a map of
the set $Q$, is a left E-quasigroup;
    \item the equality $zy \cdot x = zx \cdot y \delta(x) $ is true for all $x, y, z \in Q$, where  $\delta$ is a map of
the set $Q$, is a right  E-quasigroup.
\end{enumerate}
\end{lemma}
\begin{proof}
\begin{enumerate}
    \item If we  take $y = e(x)$, then we have  $x\cdot e(x) z = x \cdot  \delta(x) z$, $e(x) =  \delta(x)$.

\end{enumerate}
Cases 2-6 are proved similarly.
\end{proof}

\begin{theorem} \label{Toyoda_Theorem}
Toyoda Theorem  \cite{VD, 1a, 4, TOYODA, 6, SHCH_STR_05}. Any medial quasigroup $(Q,\cdot)$ can be presented in
the form: $ x\cdot y = \varphi x + \psi y + a$, \label{MEDIAL_eqno(2)} where $(Q, +)$ is an abelian group,
$\varphi, \psi $ are automorphisms of $ (Q, +)$ such that $\varphi \psi = \psi \varphi,$  $a$ is some fixed
element of  the set $Q$ and vice versa.
\end{theorem}

\begin{theorem} \label{Belousov_Theorem}
 Belousov  Theorem  \cite{DISTRIB_BEL_60, VD, 1a}. Any  distributive  quasigroup $(Q,\circ)$ can be
presented in the form: $ x\circ y = \varphi x + \psi y$, where $(Q, +)$ is a commutative Moufang loop, $\varphi,
\psi \in Aut (Q, +)$, $\varphi, \psi \in Aut (Q, \cdot)$,   $\varphi \psi = \psi \varphi$.
\end{theorem}

A left (right) F--quasigroup is isotopic to a left (right) M--loop \cite{GOLOVKO_II, 1a}. A left (right)
F--quasigroup is isotopic to a left (right) special loop \cite{PROP_BIN_OPERATS, vdb_FLOR, VD, GOLOVKO_I}. An
F--quasigroup is isotopic to a Moufang loop \cite{kepka_05}.

If a loop $(Q, \circ)$ is isotopic to a left distributive quasigroup $(Q, \cdot)$ with isotopy  the form $x\circ
y = R^{-1}_a x\cdot L^{-1}_a  y$, then $(Q, \circ )$  will be  called a \textit{left  $S$-loop}. Loop  $(Q,
\circ)$ and  quasigroup $(Q, \cdot)$ are said to be \textit{related}.

If a loop $(Q, \circ)$ is isotopic to a right distributive quasigroup $(Q, \cdot)$ with isotopy  the form
$x\circ y=R^{-1}_ax\cdot L^{-1}_ay$, then $(Q, \circ )$  will be  called a \textit{right  $S$-loop}.

\begin{definition}
An automorphism  $\psi$ of a loop $(Q,\circ)$ is called {\it complete}, if  there exists a permutation $\varphi$
of the set $Q$ such that $\varphi x \circ \psi x =x $ for all  $x\in Q$. Permutation  $ \varphi $ is called a
complement of automorphism  $ \psi$. \label{Automorphism_full}
\end{definition}

The following theorem is proved in \cite{BO}.

\begin{theorem}\label{AT1}
A loop   $(Q,\circ)$   is a left $S$-loop,  if and only if there exists a complete  automorphism  $\psi$ of the
loop $(Q,\circ)$ such that at least one from the following conditions is fulfilled:

a) $\varphi (x\circ \varphi^{-1}y)\circ (\psi x\circ z)= x\circ (y\circ z)$;

b) $ L_{x,y}^{\circ} \psi = \psi L_{x,y}^{\circ }$ and  $\varphi x\circ(\psi x \circ y) = x\circ y$ for all $x,
y \in Q$, $x, y \in Q$, $ L_{x,y}^{\circ} \in LI(Q,\circ )$.

Thus  $(Q,\cdot)$, where  $x \cdot y = \varphi x \circ \psi y,$ is a left distributive quasigroup which
corresponds to  the loop $(Q,\circ)$.
\end{theorem}

\begin{remark}
In \cite{BO,ONOI_D} a left S-loop is called an  \textit{S-loop}.
\end{remark}

A left distributive quasigroup $(Q, \cdot)$  with identity $x\cdot xy = y$ is isotopic to a left Bol loop
\cite{vdb_FLOR, VD, 1a}. Last results of G. Nagy \cite{NAGY_G_08} let us to hope on progress in researches of
left distributive quasigroups.  Some properties of   distributive and left distribute quasigroups are described
in \cite{GALKIN_78, GALKIN_79, GALKIN_D, STEIN_A}.

\begin{theorem} \label{BELOUSOV_LEFT M_LOOP}
Any loop which is isotopic to a left F-quasigroup is a left M-loop (\cite{1a}, Theorem 3.17, p. 109).
\end{theorem}

\begin{theorem} \label{ALBERT_THEOREM}
Generalized Albert Theorem. Any  loop isotopic to a group is a group  \cite{1a, A1, A2, VD, 8, HOP, SCERB_03}.
\end{theorem}

\subsection{Congruences and  homomorphisms}

Results of this subsection are standard, well known  \cite{VD, RHB, HOP, MALTSEV, BURRIS} and slightly adapted
for our aims.

A binary relation $\varphi $ on a set $Q $ is  a subset of  the cartesian product $Q\times Q $ \cite {BIRKHOFF,
II, MALTSEV}. \index{Binary relation}

\index{Binary relations!product of} If $ \varphi $ and $ \psi $ are  binary relations on $Q $, then  their
product is defined in the following way: $ (a, b) \in \varphi \circ \psi $, if there is an element $c\in Q $
such that $ (a, c) \in \varphi $ and $ (c, b) \in \psi $. The last condition is written also in such form $a\,
\varphi \, c \, \psi \, b$.

\begin{theorem}  \label{EQUIVALENCE_TH}
Let $S$ be a nonempty set and let $\sim$ be a relation between elements of $S$ that satisfies
 the following properties:

 1. (Reflexive) $a\sim a$ for all $a\in S$.

2. (Symmetric) If $a\sim b$, then $b\sim a$.

3. (Transitive) If $a\sim b$ and $b\sim c$, then $a\sim c$.

Then $\sim$ yields a natural partition of $S$, where $\bar a = \{x\in S \,|\, x\sim a\}$ is the cell containing
$a$ for all $a\in S$. Conversely, each partition of $S$ gives rise to a natural relation $\sim$ satisfying the
reflexive, symmetric, and transitive properties if $a\sim b$ is defined to mean that $a\in \bar b$ \cite{HER}.
\index{relation!reflexive} \index{relation!symmetric} \index{relation!transitive}
\end{theorem}

\begin{definition} A relation $\sim$ on a set $S$ satisfying the reflexive, symmetric, and transitive
properties is an equivalence relation on $S$. Each cell $\bar a$ in the natural partition given by an
equivalence  relation is an equivalence class. \index{relation!equivalence}
\end{definition}

\begin{definition} \label{DEF_CONGR_QUAS}
An equivalence $\theta$ is a congruence of a groupoid  $(Q, \cdot)$, if the following implications are true for
all  $x, y, z \in Q$: $x \theta y  \Longrightarrow (z\cdot x)\theta (z\cdot y), x\theta y \Longrightarrow
(x\cdot z)\theta (y\cdot z)$ \cite{PC}.
\end{definition}

In other words equivalence $\theta $ is a \textit{congruence} of $(Q, \cdot)$ if and only if $\theta$ is a
subalgebra of $(Q\times Q, (\cdot,\cdot))$.\index{congruence} Therefore we can formulate Definition
\ref{DEF_CONGR_QUAS} in the following form.
\begin{definition}
An equivalence $\theta$ is a congruence of a groupoid  $(Q, \cdot)$, if the following implication is true for
all  $x, y, w, z \in Q$: $x \theta y  \wedge w \theta z  \Longrightarrow (x \cdot w)\, \theta \, (y\cdot z)$
\cite{PC}.
\end{definition}

\begin{definition}\label{NORM_CONGR_QUAS} A congruence $\theta$ of a quasigroup $(Q,\cdot)$ is \textit{normal},
if the following implications are true for all  $x, y, z \in Q$: $(z\cdot x) \theta (z\cdot y)  \Longrightarrow
x\theta y, (x\cdot z) \theta (y\cdot z) \Longrightarrow x\theta y$ \cite{VD, 1a}.
\end{definition}

\begin{definition} \label{DEF_CONGR_QUAS_AS_ALGEBRA}
An equivalence $\theta$ is a congruence of a quasigroup  $(Q, \cdot, \slash, \backslash)$, if the following
implications are true for all  $x, y, z \in Q$: \begin{equation}
\begin{split}
& x \theta y  \Longrightarrow (z\cdot x)\theta (z\cdot y), x\theta y \Longrightarrow (x\cdot z)\theta (y\cdot
z), \\ &  x \theta y  \Longrightarrow (z\slash x)\theta (z\slash y), x\theta y \Longrightarrow (x\slash z)\theta
(y\slash z), \\ & x \theta y  \Longrightarrow (z\backslash x)\theta (z\backslash y), x\theta y \Longrightarrow
(x\backslash z)\theta (y\backslash z).
\end{split}
\end{equation}
\end{definition}

One from the most important properties of  e-quasigroup  $(Q, \cdot, \backslash, /)$  is the following property.
\begin{lemma} \label{CONGR_PRIM_QUAS}
Any congruence of a quasigroup $(Q, \cdot, \backslash, /)$  is a normal congruence of quasigroup $(Q, \cdot)$;
any normal congruence of a quasigroup $(Q, \cdot)$ is a congruence of  quasigroup $(Q, \cdot, \backslash, /)$
\cite{BIRKHOFF, MALTSEV,  VD, 1a}.
\end{lemma}

\begin{definition}
 If $\theta$ is a binary relation on a  set $(Q, \cdot)$, $\alpha$ is a permutation of the set $Q$ and from $x\theta y$
it follows  $\alpha x \theta \alpha y$  and  $\alpha^{-1} x \theta \alpha^{-1} y$ for all $(x,y) \in \theta$,
then we shall say  that the permutation $\alpha$ is an \textit{admissible} permutation relative to the binary
relation $\theta$ \cite{VD}.
\end{definition}

\begin{lemma} \label{NORM_CONGR_ADMISSBLE_REL_TRANS}
Any normal quasigroup congruence is admissible relative to  any left, right and middle  quasigroup translation.
\end{lemma}
\begin{proof}
The fact that any normal quasigroup congruence is admissible relative to  any left and right  quasigroup
translation follows from Definitions \ref{DEF_CONGR_QUAS} and \ref{NORM_CONGR_QUAS}.

Let $\theta$ be a normal congruence of a quasigroup $(Q,\cdot)$. Prove the following implication
\begin{equation} \label{IMPLICATION_1}
a\theta b \rightarrow P_c a \, \theta \, P_c b.
\end{equation}
 If $P_c a = k$, then $a\cdot k = c$, $k = a\backslash c$, $k =
R^{\backslash}_c a$. Similarly  if $P_c b = m$, then $b\cdot m = c$, $m = b\backslash c$, $m = R^{\backslash}_c
b$. Since $\theta$ is a congruence of quasigroup $(Q, \cdot, \backslash, /)$ (Lemma \ref{CONGR_PRIM_QUAS}), then
implication (\ref{IMPLICATION_1}) is true.

Implication \begin{equation}\label{IMPLICATION_2}
 a\theta b \rightarrow P^{-1}_c a \, \theta \, P^{-1}_c b
\end{equation} is proved in the similar way.
 If $P^{-1}_c a = k$, then $k \cdot a = c$, $k = c\slash a$, $k =
L^{\slash}_c a$. Similarly  if $P^{-1}_c b = m$, then $m\cdot b = c$, $m = c\slash b$, $m = L^{\slash}_c b$.
Since $\theta$ is a congruence of quasigroup $(Q, \cdot, \backslash, /)$ (Lemma \ref{CONGR_PRIM_QUAS}), then
implication (\ref{IMPLICATION_2}) is true.
\end{proof}
\begin{corollary} \label{PARASTROPHY_INVAR_NORM_CONG}
If $\theta$ is a normal quasigroup congruence of a quasigroup $Q$, then $\theta$ is a normal congruence of any
parastrophe  of $Q$  \cite{BELAS}.
\end{corollary}
\begin{proof}
The proof follows from Lemma \ref{NORM_CONGR_ADMISSBLE_REL_TRANS} and Table 1.
\end{proof}

In Lemma \ref{NORMAL_LEMMA} we shall use the following fact about quasigroup translations and normal quasigroup
congruences.
\begin{lemma}
If $a\theta b$, $c\theta d$, then $R^{-1}_a c \, \theta \, R^{-1}_b d$.
\end{lemma}
\begin{proof}
If $ac\, \theta \, bd$ and $c\theta d$, then $a\theta b$. Indeed, if $c\theta d$, then $ac\theta ad$. If $ac\,
\theta \, bd$ and $ac\theta ad$, then $bd\theta ad$, and, finally, $a\theta b$. In other words, if $R_c a \,
\theta \, R_d b$ and $c\theta d$, then $a\theta b$.

Since $a\theta b$ and $\theta$ is a normal quasigroup congruence, we have  $c \, \theta \, d \Longleftrightarrow
R_a R^{-1}_a c \, \theta \, R_b R^{-1}_b d \Longleftrightarrow  R^{-1}_a c \, \theta \, R^{-1}_b d. $
\end{proof}

We give a sketched  proof of the following well known fact  \cite{7, SM, BURRIS}. We follow \cite{SM}.
\begin{lemma} \label{NORMAL_LEMMA}
Normal quasigroup congruences commute in pairs.
\end{lemma}
\begin{proof}
Let $\theta_1$ and $\theta_2$ be normal congruences of a quasigroup $(Q, \cdot)$. Then $a(\theta_1\circ
\theta_2) b$ means that there exists an element $c\in Q$ such that $a \theta_1 c$ and $c \theta_2 b$.

Further we have
$$
\begin{array} {lc}
a\theta_2 a, & a\theta_2 a \\
c\theta_2 b, & L^{-1}_c c\theta_2 L^{-1}_c  b \\
b\theta_2 b, & L^{-1}_c b\theta_2 L^{-1}_c  b.
\end{array}
$$
Then $$
 R^{-1}_{L^{-1}_c c} a\cdot L^{-1}_c b \, \theta_2 \, R^{-1}_{L^{-1}_c b} a\cdot L^{-1}_c b = a.
$$
From relations
$$
\begin{array} {lc}
a\theta_1 a, & a\theta_1 a \\
a\theta_1 c, & L^{-1}_c a\theta_1 L^{-1}_c  c \\
b\theta_1 b, & L^{-1}_c b\theta_1 L^{-1}_c  b
\end{array}
$$
we obtain
$$
b =  R^{-1}_{L^{-1}_c a} a\cdot L^{-1}_c b \, \theta_1 \,  R^{-1}_{L^{-1}_c c} a\cdot L^{-1}_c b.
$$
Therefore, $a(\theta_2\circ \theta_1) b$.
\end{proof}

See \cite{Fujiwara_66} for additional information on permutability of quasigroup congruences.

\begin{definition}\label{d5.1}  If  $(Q, \cdot)$ and $(H, \circ)$ are binary quasigroups,
 $h$ is  a single valued mapping of $Q$ into $H$ such that $h (x_1 \cdot x_2) = h x_1 \circ h x_2$, then $h$ is
called a \textit{homomorphism (a multiplicative homomorphism)} of $(Q, \cdot)$ into $(H, \circ)$ and the set $\{
h x\, | \, x \in Q\}$ is called \textit{homomorphic image} of $ (Q, \cdot)$ under $h$ \cite{HOP}.
\end{definition}

In  case $(Q, \cdot) = (H, \circ)$ a homomorphism is also called an \textit{endomorphism} and an isomorphism is
referred to as an automorphism.

\begin{lemma}1. Any homomorphic image of a quasigroup $(Q, \cdot)$ is a division groupoid \cite{BK, RHB}.

2. Any homomorphic image of a quasigroup $(Q, \cdot, \backslash, \slash)$ is a quasigroup \cite{BURRIS,
MALTSEV}.
 \label{l_5.74}
\end{lemma}
\begin{proof}
1. Let $h(a), h(b) \in h(Q)$. We demonstrate that solution of equation $h(a)\circ x = h(b)$ lies in $h(Q)$.
Consider the equation $a \cdot y = b$. Denote solution of this equation by $c$. Then $h(c)$ is solution to the
equation $h(a)\circ x = h(b)$. Indeed, $h(a) \circ h(c) = h(a\cdot c) = h(b)$. For equation  $x \cdot h(a)  =
h(b)$ the proof is similar.

2. See \cite{VD, 1a, HOP, BURRIS, MALTSEV}.
\end{proof}

Let $h$ be a homomorphism of a quasigroup $(Q, \cdot)$ onto a groupoid  $(H, \circ)$.
 Then $h$ induces a congruence $Ker\,h = \theta$ (the kernel of $h$) in the following way $x \, \theta \, y$
if and only if $h(x) = h(y)$ \cite{1a, HOP}.

If   $\theta$ is a normal congruence of a quasigroup $(Q, \cdot)$, then $\theta$ determines natural homomorphism
$h$ ($h(a) = \theta(a))$ of $(Q, \cdot)$ onto some quasigroup $(Q^{\,\prime}, \circ)$ by the rule:
 $\theta(x) \circ \theta(y) =  \theta(x \cdot y)$, where $\theta(x), \theta(y), \theta(x \cdot y) \in Q\slash
\theta$ \cite{1a, HOP}.

\begin{theorem}  \label{NORM_QUAS_CONGR}
If $h$ is a homomorphism of a quasigroup $(Q,\cdot)$ onto a quasigroup $(H, \circ)$, then $h$ determines a
normal congruence $\theta$ on $(Q,\cdot)$ such that $Q\slash \theta \cong (H,\circ)$, and vice versa, a normal
congruence $\theta$ induces a homomorphism from $(Q, \cdot)$ onto $(H, \circ)\cong Q\slash \theta$. (\cite{1a},
\cite{HOP}, I.7.2 Theorem).
\end{theorem}

A subquasigroup $(H,\cdot)$ of a quasigroup $(Q, \cdot)$ is \textit{normal} ( $(H, \cdot) \trianglelefteqslant
(Q,\cdot)$), if $(H,\cdot)$ is an equivalence class (in other words, a coset class)  of a normal congruence.

\begin{lemma} \label{BINARY_SUB_GROUPOIDS} An equivalence class $\theta(h) = H$ of a  congruence $\theta$
of a  quasigroup $(Q, \cdot)$ is a sub-object  of $(Q, \cdot)$ if and only if $(h \cdot h) \; \theta \; h$.
\end{lemma}
\begin{proof}
We recall by Lemma \ref{SUB_GROUPOIDS} any quasigroup sub-object is a cancellation groupoid. The proof is
similar to the proof of Lemma 1.9 from \cite{1a}. If $a\, \theta \, h$ and $b\, \theta \, h$, then $ab\, \theta
\, h^2$, moreover $h^2 \, \theta \, h$, since $ab\in H$. Converse. Let $h^2 \, \theta \, h$. If $a, b \in H$,
then $a \,\theta \, h$ and $b \, \theta \, h$, $ab \, \theta \, h^2 \theta \, h$. Then $ab\in H$.
\end{proof}

\begin{lemma} \label{BINARY_SUB_QUAS} An equivalence class $\theta(h)$ of a normal congruence $\theta$
of a  quasigroup $(Q, \cdot)$ is a subquasigroup of $(Q, \cdot)$ if and only if $(h \cdot h) \; \theta \; h$
(\cite{VD}, \cite{1a},  Lemma 1.9).
\end{lemma}

\begin{lemma}\label{l_5.4} If  $h$ is an  endomorphism   of a quasigroup $(Q, \cdot)$, then
$(hQ, \cdot)$ is a subquasigroup of  $(Q,\cdot)$.
\end{lemma}
\begin{proof}
We re-write the proof from  (\cite{1a},  p. 33) for  slightly more general case. Prove  that $(hQ, \cdot)$ is a
subquasigroup of quasigroup $(Q, \cdot)$. Let $h(a)$, $h(b) \in h(Q)$. We demonstrate that solution of equation
$h(a)\cdot x = h(b)$ lies in $h(Q)$. Consider the equation $a \cdot y = b$. Denote solution of this equation by
$c$. Then $h(c)$ is solution of equation $h(a)\cdot x = h(b)$. Indeed, $h(a) \cdot h(c) = h(a\cdot c) = h(b)$.

It is easy to see, that this is a unique solution. Indeed, if $h(a) \cdot c_1 = h(b)$, then $h(a) \cdot h(c) =
h(a) \cdot c_1$. Since $h(a), h(c), c_1$ are elements of quasigroup $(Q, \cdot)$, then $h(c) = c_1$.

For equation  $x \cdot h(a)  = h(b)$ the proof is similar.
\end{proof}

\begin{remark}
It is possible to give the following proof of Lemma \ref{l_5.4}. The $(hQ, \cdot)$ is a cancellation groupoid,
since it is a sub-object of the quasigroup $(Q, \cdot)$ (Lemma \ref{SUB_GROUPOIDS}). From the other side $(hQ,
\cdot)$ is a division groupoid, since it is a homomorphic image of $(Q, \cdot)$ (Lemma \ref{l_5.74}). Therefore
by Remark \ref{remark_one} $(hQ, \cdot)$ is a subquasigroup of the quasigroup $(Q, \cdot)$.
\end{remark}

\begin{corollary}\label{corll_5.4} 1. Any subquasigroup $(H,\cdot)$ of a left F-quasigroup $(Q, \cdot)$ is a left
F-quasigroup.

2. Any endomorphic image of a left F-quasigroup $(Q, \cdot)$ is a left F-quasigroup.
\end{corollary}
\begin{proof}
1. If $a\in H$, then the solution of equation $a\cdot x=a$, $x=e(a)$ also is in $H$.

2. From Case 1 and  Lemma \ref{l_5.4} it follows that any endomorphic image of a left F-quasigroup $(Q, \cdot)$
is a left F-quasigroup.
\end{proof}

\begin{remark}\label{REM_7_NORM}
The same situation is for right F-quasigroups, left and right E- and SM-qua\-si\-gro\-ups and all combinations
of these properties.
\end{remark}

\begin{corollary} \label{COROLL_1}
If  $h$ is an  endomorphism   of a quasigroup $(Q, \cdot)$, then $h$ is an endomorphism of the quasigroups $(Q,
\ast)$, $(Q, \slash)$, $(Q, \backslash)$, $(Q,  \slash\slash)$, $(Q, \backslash\backslash)$, i.e. from $h(x
\cdot y) = h (x) \cdot h (y)$ we obtain  that

1.      $h(x \ast y) = h (x) \ast h (y)$. 2. $h(x \slash y) = h (x) \slash h (y)$.

3. $h(x \backslash y) = h (x) \backslash  h(y)$. 4.  $h(x \slash\slash y) = h (x) \slash\slash h (y)$.

5.    $h(x \backslash\backslash y) = h (x) \backslash\backslash  h (y)$.
\end{corollary}
\begin{proof}
From Lemma \ref{l_5.4} we have  that $(hQ, \cdot)$ is a subquasigroup of  $(Q,\cdot)$.

Case 1. If we pass from the quasigroup $(Q, \cdot)$ to quasigroup $(Q, \ast)$, then subquasigroup $(hQ, \ast)$
of the quasigroup $(Q, \ast)$  will correspond to the subquasigroup $(hQ, \cdot)$. Indeed, any subquasigroup of
the quasigroup $(Q, \cdot)$ is closed relative to parastrophe operations $\ast , \slash, \backslash,
\slash\slash, \backslash\backslash$ of the quasigroup  $(Q, \cdot)$.  Further we have $h(x \ast y) = h(y \cdot
x) = h(y) \cdot h(x) = h (x) \ast h (y)$.

Case 2. If we pass from the quasigroup $(Q, \cdot)$ to quasigroup $(Q, \slash)$, then subquasigroup  \, $(hQ,
\slash)$ of the quasigroup $(Q, \slash)$  will correspond to the subquasigroup $(hQ, \cdot)$.

Let $z = x\slash y$, where $x, y \in Q$. Then from definition of the operation $\slash$ it follows that $x =
zy$. Then  $h (x) = h (z) h (y)$, $h(x \slash y) = h (z) = h (x) \slash  h (y)$ (\cite{MALTSEV}, p. 96, Theorem
1).

The remaining cases are  proved in the similar way.
\end{proof}

\begin{lemma}\label{l5.4} If $(Q,\cdot)$ is a finite quasigroup, then any its congruence is
normal,   any its homomorphic image is a  quasigroup  \cite{VD, SC05}.
\end{lemma}

\begin{lemma} \label{COMMUTING_ENDOMORPH}
 Let $(Q, \cdot)$ be a quasigroup. If $f$ is an endomorphism of $(Q,\cdot)$, then $f(e(x)) = e(f(x))$,
$f(s(x)) = s(f(x))$  for all $x\in Q$;

if $e$ is an endomorphism of $(Q, \cdot)$, then $e(f(x)) = f(e(x))$, $e(s(x)) = s(e(x))$ for all $x\in Q$;

if $s$ is an endomorphism of $(Q, \cdot)$, then $s(f(x)) = f(s(x))$, $s(e(x)) = e(s(x))$ for all $x\in Q$ (Lemma
2.4. from \cite{Kin_PHIL_04}).
\end{lemma}
\begin{proof}
We shall use Corollary \ref{COROLL_1}.
 If $f$ is an endomorphism, then $f(e(x)) = f(x \backslash x) = f(x)\backslash f(x) = e(f(x))$, $f(s(x)) = f(x)\cdot f(x) =
s(f(x))$.

If $e$ is an endomorphism, then $e(f(x)) = e(x)\slash e(x) = f(e(x))$, $e(s(x)) = e(x)\cdot e(x) = s(e(x))$.

If $s$ is an endomorphism, then $s(f(x)) = s(x)\slash s(x) = f(s(x))$, $s(e(x)) = s(x)\backslash  s(x) =
s(e(x))$.
\end{proof}

 The group $M(Q,\cdot) = \left< L_a, R_b \, | \, a, b \in Q\right > $, where $(Q,\cdot)$ is  a quasigroup,
 is called multiplication group of quasigroup.

 The group $\mathbb I_h =\{\alpha \in M(Q,\cdot) \, | \, \alpha h = h\}$ is called
inner mapping group of a quasigroup $(Q,\cdot)$ relative to an element $h\in Q$. Group $\mathbb I_h$ is
stabilizer of a fixed element $h$ by action ($\alpha: x\longmapsto \alpha(x)$ for all $\alpha \in M(Q,\cdot)$,
$x\in Q$) of group $M(Q,\cdot)$ on the set $Q$. In loop case usually it is studied the group $\mathbb I_1
(Q,\cdot) = \mathbb I (Q,\cdot)$, where $1$ is the identity element of a loop $(Q,\cdot)$.

\begin{theorem} \label{norm_subquas}
A subquasigroup $H$ of a quasigroup $Q$ is normal if and only if $\mathbb I_k H\subseteq H$ for a fixed element
$k\in H$ \cite{VD}.
\end{theorem}

In \cite{VD}, p. 59 the following  key  lemma is proved.
\begin{lemma} \label{NL1}
 Let $ \theta $ be a normal congruence of a quasigroup $ (Q, \cdot) $. If a quasigroup
$ (Q, \circ) $ is isotopic  to $ (Q, \cdot) $ and the isotopy $ (\alpha, \beta, \gamma) $ is admissible relative
to $ \theta $, then $ \theta $ is a normal congruence also in $(Q, \circ)$.
\end{lemma}

For our aims will be usable  the following theorem.

\begin{theorem} \label{DLYA_CLM}
Let $(Q, +)$ be an IP-loop, $x\cdot y = (\varphi x + \psi y)+c$, where $\varphi, \psi \in Aut(Q, +)$,  $a\in
C(Q,+)$, $\theta$ be a normal congruence of $(Q, +)$. Then $\theta$ is normal congruence of $(Q, \cdot)$ if and
only if $\varphi \mid_{Ker \, \theta}$, $\psi \mid_{Ker \, \theta}$ are automorphisms of $Ker \, \theta$
\cite{pntk, tkpn, SCERB_91, vs2}.
\end{theorem}

We denote by  $nCon (Q, \cdot) $  the  set of all normal congruences of a quasigroup $ (Q, \cdot) $.

\begin{corollary} \label{NC10}
 If $(Q, \cdot)$ is a quasigroup, $ (Q, +) $ is  a loop of the form  $x + y
= R ^ {-1} _a x\cdot L ^ {-1} _b y $ for all $x, y \in Q $, then $nCon (Q, \cdot) \subseteq nCon (Q, +)$.
\end{corollary}
\begin{proof} If  $\theta $ is a normal congruence of a quasigroup $ (Q, \cdot) $, then, since  $
\theta $  is admissible relative to   the isotopy $T = (R ^ {-1} _a, L ^ {-1} _b, \varepsilon) $, $ \theta $ is
also a normal congruence of a loop $ (Q, +) $. \end{proof}

In loop case situation with normality of subloops is well known and more near to the group case \cite{RHB, HOP,
VD, 1a}. As usual a subloop $(H,+)$ of a loop $(Q,+)$ is normal, if $H = \theta (0) = Ker \, \theta$, where
$\theta (0)$ is an equivalence class of a normal congruence  $\theta$ that  contains identity  element of
$(Q,+)$ \cite{VD, HOP}. We shall name  congruence $\theta$ and subloop $(H,+)$ by \textit{corresponding}.

\begin{example}
In the group $S_3$ ($S_3 = \left<a, b\, | \, a^3=b^2 = 1, bab=a^{-1} \right>$, $S_3 \cong Z_3 \leftthreetimes
Z_2$) there exists endomorphism $h$ ($h(a) =1, h(b) = b$)  such that $h(S_3) = \left<b\right> \cong Z_2$, $Ker
\, h = \left< a \right> \cong Z_3 $ and $Z_2 \ntrianglelefteq S_3$.
\end{example}

\begin{example} \label{Cyclic_Group_4_ENDOM}
In the cyclic group $(Z_4, +)$, $Z_4 = \{0, 1, 2, 3 \} $,   there exists endomorphism $h$ ($h(x) = x+x$)  such
that $h(Z_4) = Ker \, h = \{ 0, 2 \}$. The endomorphism $h$ defines normal congruence $\theta$ with the
following coset classes $\theta (0) = \{0, 2 \}$ and $\theta (1) = \{1, 3 \}$. It is clear that $Z_4 \slash
\theta \cong Z_2$.
\end{example}

\begin{definition}
A normal subloop $(H,+)$ of a loop $(Q,+)$ is admissible relative to a permutation $\alpha$ of the set $Q$ if
and only if the corresponding to $(H,+)$ normal congruence $\theta$ is admissible relative to  $\alpha$.
\end{definition}

\begin{definition} \label{SIMPLE_QUAS}
A quasigroup $(Q,\cdot)$ is \emph{simple} if its only normal congruences are the diagonal $\hat{Q} = \{(q,q)\, |
q\in Q\}$ and universal $Q\times Q$. \index{$n$-Ary quasigroup!simple}
\end{definition}

\begin{definition} \label{ALPHA_SIMPLE_QUAS}
We shall name a subloop  $(H,+)$ of a loop $(Q, +)$ $\alpha$-invariant relative to a permutation $\alpha$ of the
set $Q$, if  $\alpha H = H$.

We shall name a loop $(Q,+)$ $\alpha$-simple if only identity subloop and the loop $(Q,+)$ are invariant
relative to the permutation  $\alpha$ of the set $Q$.

We shall name a quasigroup $(Q, \cdot)$ $\alpha$-simple relative to the permutation  $\alpha$ of the set $Q$, if
only the diagonal and universal congruences are admissible relative to $\alpha$.
\end{definition}

\begin{corollary} \label{DLYA_PROSTYH_F_QUAS}
Let   $(Q,\cdot)=(Q,+) (\alpha, \beta, \varepsilon)$, where $(Q,+)$ is a loop, $\alpha, \beta \in S_Q$.
 If  $(Q,+)$ does not contain normal subloops admissible relative to permutations $\alpha, \beta$,
 then quasigroup $(Q,\cdot)$ is simple.
\end{corollary}
\begin{proof}
The proof  follows from Lemmas \ref{NL1} and  \ref{LP_AND_PRINCIP_ISOT}, Corollary \ref{NC10}.
\end{proof}

\subsection{Direct products}

\begin{definition} \label{EXTERNAL_DIRECT}
If   $(Q_1, \cdot)$, $(Q_2,\circ)$ are binary  quasigroups, then their {\it (external) direct product}
$(Q,\ast)=(Q_1,\cdot)\times (Q_2,\circ)$   is the set of all ordered pairs $(a^{\prime}, a^{\prime\prime})$
where $a^{\prime} \in Q_1$, $a^{\prime\prime} \in Q_2$, and where the operation in $(Q,\ast)$ is defined
component-wise, that is, $(a_1 \ast a_2) = (a^{\prime}_1 \cdot a^{\prime}_2, a^{\prime\prime}_1 \circ
a^{\prime\prime}_2)$.
\end{definition}

Direct product of quasigroups is studied in many articles and books, see, for example, \cite{ PC, SM, pntk,
JEZEK_77, GB1, GB}. The concept of direct product of quasigroups  was used already in \cite{4}. In group case it
is possible to find these definitions, for example, in \cite{HER}.

In \cite{ BURRIS, SM, JDH_2007}  there is a definition of  the (internal) direct product of $\Omega$-algebras.
We recall that any quasigroup is an $\Omega$-algebra.

Let  $U$ and $W$ be equivalence relations on a set $A$, let  $U\vee W = \{ (x,y) \in A^2 \, | \, \exists \ n\in
N, \ \exists \ t_0, t_1, \dots, t_{2n}\in A, x = t_0 U t_1 $ $W $ $t_2 $ $U\dots U t_{2n-1} W t_{2n} = y \}$.
$U\vee W$ is an equivalence relation on $A$ called the join of $U$ and $W$. If $U$ and $W$ are equivalence
relations on $A$ for which $U\circ W = W\circ U$, then $U\circ W = U\vee W$, $U$ and $W$ are said to commute
\cite{SM}.

If $A$  is an $\Omega$-algebra  and $U,$ $W$ are congruences on $A$, then  $ U\vee W$, and $U\cap W$  are also
congruences on $A$.

\begin{definition}
If $U$ and $W$ are congruences on the algebra $A$ which commute and for which $U\cap W = \hat A = \{(a,a)| \,
\forall \, a\in A\}$, then the join $U\circ W = U \vee W$ of $U$ and $W$ is called {\it direct product $U\sqcap
W$ of $U$ and $W$} \cite{SM, JDH_2007}.
\end{definition}

The following theorem  establishes  the connection between concepts of internal and external direct product of
$\Omega$-algebras.

\begin{theorem}\label{t5.1} An $\Omega$-algebra $A$ is isomorphic to a direct product of
$\Omega$-algebras $B$ and $C$ with isomorphism $\varphi$, i.e. $\varphi: A\rightarrow B\times C$, if and only if
there exist such congruences $U$ and $W$ of $A$ that $A^2 = U \sqcap W$  (\cite{SM}, p.16, \cite{JDH_2007}).
\end{theorem}

We shall use the following easy proved fact.

\begin{lemma} \label{center_OF_DIR_PROD}
If a loop $Q$ is isomorphic to the direct product of the loops $A$ and $B$, then $C(Q) \cong C(A)\times C(B)$.
\end{lemma}

\begin{lemma} \label{Componets_OF_DIR_PROD}
If a left F-quasigroup  $Q$ is isomorphic to the direct product of a left F-quasigroup  $A$ and a quasigroup
$B$, then $B$ also is a left F-quasigroup.
\end{lemma}
\begin{proof}
Indeed, if $q = (a, b)$, where $q\in Q$, $a\in A$, $b\in B$, then  $e(q) = (e(a), e(b))$.
\end{proof}

\begin{remark}
An analog of Lemma \ref{Componets_OF_DIR_PROD} is true for right F-quasigroups, left and right SM- and
E-quasigroups.
\end{remark}

There exist various approaches to the concept of  semidirect product of quasigroups \cite{SABININ_91,
SABININ_99, Burdujan,  WIKI}. By  an analogy with group case \cite{KM} we give the following definition of the
semidirect product of quasigroups. Main principe is that a semidirect product is a cartesian product as a set
\cite{WIKI}.

\begin{definition}Let $Q$ be a quasigroup, $A$  a normal subquasigroup of $Q$ (i.e., $A\unlhd  Q$) and $B$ a subquasigroup  of
$Q$. A quasigroup $Q$ is the semidirect product of quasigroups $A$ and  $B$, if there exists a  homomorphism $h
: Q \longrightarrow  B$ which is the identity on $B$ and whose kernel is $A$, i.e. $A$ is a coset class of the
normal congruence $Ker \, h$. We shall denote this fact as follows $Q \cong A\leftthreetimes B$ \cite{WIKI_2}.
\end{definition}

\begin{remark}
From results of  A.I.~Mal'tsev  \cite{7}, see, also, \cite{SHCH_BR_BL_05}, it follows that normal subquasigroup
$A$ is a coset class of only one normal congruence of the quasigroup $Q$.
\end{remark}

\begin{lemma} \label{LEMMA_DIR_2}
If a quasigroup $Q$ is the semidirect product of quasigroups $A$ and $B$, $A \unlhd Q$, then there exists an
isotopy $T$ of $Q$ such that  $QT$ is a loop and  $QT \cong AT\leftthreetimes BT$.
\end{lemma}
\begin{proof}
If we take isotopy of the form $(R^{-1}_a, L^{-1}_a, \varepsilon)$, where $a\in A$, then we have that $QT$ is a
loop, $AT$ is its normal subloop (Lemma \ref{NL1}, Remark \ref{NORM_CONGR_ADMISSBLE_REL_TRANS}). Further we have
that $BT$ is a loop since $BT \cong QT \slash AT$. Therefore $BT$ is a subloop of the loop $QT$, since the set
$B$ is a subset of the set $Q$.
\end{proof}

\begin{corollary} \label{COROL_DIR_2}
If a quasigroup $Q$ is the direct product of quasigroups $A$ and $B$, then there exists an isotopy $T=(T_1,
T_2)$ of $Q$ such that  $QT \cong AT_1 \times BT_2$ is a loop.
\end{corollary}
\begin{proof}
The proof follows from Lemma \ref{LEMMA_DIR_2}.
\end{proof}

\begin{lemma} \label{LEMMA_LOOP_GROUP}
1. If a linear left loop $(Q, \cdot)$ with the form $x\cdot y = x + \psi y $, where $(Q,+)$ is a group, $\psi
\in Aut (Q,+)$, is the semidirect product of a normal  subgroup  $(H, \cdot) \trianglelefteqslant (Q, \cdot)$
and  a subgroup $(K,\cdot) \subseteq (Q, \cdot)$, $H \cap K = 0$, then $(Q, \cdot) = (Q,+)$.

2. If a linear right  loop $(Q, \cdot)$ with the form $x\cdot y = \varphi x + y $, where $(Q,+)$ is a group,
$\varphi \in Aut (Q,+)$, is the semidirect product of a normal  subgroup $(H,\cdot) \trianglelefteqslant (Q,
\cdot)$ and a subgroup  $(K, \cdot) \subseteq (Q, \cdot)$, $H \cap K = 0$, then $(Q, \cdot) = (Q,+)$.
\end{lemma}
\begin{proof}
1. Since $(Q, \cdot)$ is the semidirect product  of a normal  subgroup  $(H,\cdot)$ and  a subgroup $(K,
\cdot)$, then we can write any element $a$ of the loop $(Q,\cdot)$  in a unique way as a pair  $a= (k,0)\cdot(0,
h)$, where $(k,0) \in (K, \cdot)$, $(0,h) \in (H, \cdot)$.  We notice, $\psi (k,0) = (k,0)$, $\psi (0, h) = (0,
h)$, since $(K, \cdot)$, $(H, \cdot)$ are subgroups of the left loop $(Q,\cdot)$. Indeed, from $(k_1 \cdot k_2)
\cdot k_3 = k_1 \cdot (k_2 \cdot k_3)$ for all $k_1, k_2, k_3 \in K$ we have $k_1 + \psi k_2 + \psi k_3 = k_1 +
\psi k_2 + \psi^2 k_3$, $ k_3 = \psi k_3$ for all $k_3 \in K$.

Further we have $\psi a= \psi ((k,0)\cdot(0, h)) = \psi ((k,0)+ \psi(0, h)) = \psi (k,0)+ \psi^2 (0, h) = (k,0)+
\psi (0, h)  = ((k,0)\cdot(0, h)) = a$, $\psi = \varepsilon$, $(Q, \cdot) = (Q,+)$.

2. This case is proved similarly to Case 1.
\end{proof}

\begin{example}
\label{Cyclic_Group_9_ENDOM} Medial quasigroup $(Z_9, \circ)$, $x\circ y = x + 4\cdot y$, where
  $(Z_9, +)$ is the cyclic group, $Z_9 = \{0, \, 1, \, 2, \, 3, \, 4, \, 5, \, 6,  \, 7, \, 8 \} $,
 demonstrates that some restrictions in Lemma \ref{LEMMA_LOOP_GROUP} are essential.
\end{example}

\subsection{Parastrophe invariants and isostrophisms}

Parastrophe invariants and isostrophisms are studied in \cite{BELAS}.

\begin{lemma} \label{LEMMA_DIR}
If a quasigroup $Q$ is the direct product of a quasigroup $A$ and a quasigroup $B$, then $Q^{\, \sigma} =
A^{\,\sigma} \times B^{\, \sigma}$, where $\sigma$ is a parastrophy.
\end{lemma}
\begin{proof}
From Theorem \ref{t5.1} it follows that  the direct product $A \times B$ defines two quasigroup congruences.
From Theorem  \ref{NORM_QUAS_CONGR} it follows that these congruences are  normal. By Corollary
\ref{PARASTROPHY_INVAR_NORM_CONG}   these congruences are invariant relative to any parastrophy of the
quasigroup $Q$.
\end{proof}

\begin{lemma} \label{LEMMA_AUT_PAR}
If $Q$ is a quasigroup and  $\alpha\in Aut(Q)$, then $\alpha \in Aut (Q^{\, \sigma})$,  where $\sigma$ is a
parastrophy.
\end{lemma}
\begin{proof}
It is easy to check \cite{SC_89_1, SCERB_91}.
\end{proof}

\begin{lemma} \label{LEMMA_PARASTR}
(1) A quasigroup $(Q, \cdot)$ is a left F-quasigroup if and only if  its (12)-parastro\-phe is a right
F-quasigroup.

(2) A quasigroup $(Q, \cdot)$ is a left E-quasigroup if and only if  its (12)-parastro\-phe is a right
E-quasigroup.

(3) A quasigroup $(Q, \cdot)$ is a left SM-quasigroup if and only if  its (12)-parastro\-phe is a right
SM-quasigroup.

(4) A quasigroup $(Q, \cdot)$ is a left distributive quasigroup if and only if  its (12)-parastro\-phe  is a
right distributive quasigroup.

(5) A quasigroup $(Q, \cdot)$ is a left distributive quasigroup if and only if  its (23)-parastro\-phe  is a
left distributive quasigroup.

(6) A quasigroup $(Q, \cdot)$ is a   left SM-quasigroup if and only if  $(Q, \backslash)$ is a left
F-quasigroup.

(7) A quasigroup $(Q, \cdot)$ is a   right SM-quasigroup if and only if  $(Q, \slash)$ is a right F-quasigroup.

(8) A quasigroup $(Q, \cdot)$ is a   left E-quasigroup if and only if  $(Q, \backslash)$ is a left E-quasigroup
(\cite{Kin_PHIL_04}, Lemma 2.2).

(9)  A quasigroup $(Q, \cdot)$ is a   right E-quasigroup if and only if  $(Q, \slash)$ is a right E-quasigroup
(\cite{Kin_PHIL_04}, Lemma 2.2).
\end{lemma}
\begin{proof}
It is easy to check Cases 1--4.

Case 5. The fulfilment in a quasigroup $(Q,\cdot)$ of  the left distributive identity  is equivalent to the fact
that in this quasigroup  any left translation $L_x$ is an automorphism of this quasigroup. Indeed, we can
re-write left distributive identity in such manner $L_x yz = L_x y \cdot L_x z$. Using Table 1 we have that
$L^{\backslash}_x = L^{-1}_x$. Thus by Lemma \ref{LEMMA_AUT_PAR} $L^{\backslash}_x \in Aut (Q,\backslash)$.
Therefore, if $(Q,\cdot)$ is a left distributive quasigroup, then $(Q,\backslash)$ also is a left distributive
quasigroup and vice versa.

Case 6. Let $(Q, \backslash)$ be a left F-quasigroup. Then $$x\backslash (y \backslash z) = (x \backslash y)
\backslash (e^{(\backslash)}(x) \backslash z) = v.$$

If $x\backslash (y \backslash z) = v$, then $x\cdot v = (y \backslash z)$, $y\cdot (x\cdot v) = z$.  We notice,
if $x\backslash  e^{(\backslash)}(x) = x$, then $e^{(\backslash)}(x) = x\cdot x \overset{def}{=}s(x)$. See Table
2.

We can re-write equality  $(x \backslash y) \backslash (e^{(\backslash)}(x) \backslash z) = v$ in the following
form  $(x\backslash y)\cdot v = s(x)\backslash z$, $s(x)\cdot((x\backslash y)\cdot v)=z$. Now we have the
equality $s(x)\cdot((x\backslash y)\cdot v) = y\cdot (x\cdot v)$. If we denote $(x\backslash y)$ by $u$, then
$x\cdot u = y$.

Therefore we can re-write equality $s(x)\cdot((x\backslash y)\cdot v) = y\cdot (x\cdot v)$  in the form
$s(x)\cdot( u \cdot v) = (x\cdot u) \cdot (x\cdot v)$, i.e. in the form $(x\cdot x)\cdot( u \cdot v) = (x\cdot
u) \cdot (x\cdot v)$.

In the similar way it is possible  to check converse: if $(Q, \backslash)$ is  a left SM-quasigroup, then $(Q,
\cdot)$ is a left F-quasigroup.

Cases 7--9 are proved in the similar way.
\end{proof}

\begin{corollary}
If $(Q,\cdot)$ is a group, then: 1. $(Q,\backslash)$ is a left SM-quasigroup;  \quad  2.  $(Q,\slash)$ is a
right SM-quasigroup.
\end{corollary}
\begin{proof}
1. Any group is a left F-quasigroup since in this case $e(x)=1$ for all $x\in Q$. Therefore we can use Lemma
\ref{LEMMA_PARASTR}, Case 6.

2. We can use Lemma \ref{LEMMA_PARASTR}, Case 7.
\end{proof}

\begin{definition}\label{D2.3}
A quasigroup $(Q,B)$ is an isostrophic image of a quasigroup $(Q,A)$ if there exists a collection of
permutations $(\sigma, (\alpha_1, \alpha_2, \alpha_3)) = (\sigma, T)$, where $\sigma \in S_3 $, $T = (\alpha_1,
\alpha_2, \alpha_3)$ and $\alpha_1, \alpha_2, \alpha_3$ are permutations of the set $Q$ such that  $$  B(x_1,
x_2) = A(x_1, x_2)(\sigma, T) = A^{\sigma} (x_1, x_2)T = \alpha^{-1} _3{A}(\alpha _1 x_{\sigma^{-1} 1}, \alpha
_2x_{\sigma^{-1} 2}) $$ for all $x_{1}, x_{2}\in Q$ \cite{2}.
\end{definition}

A collection of permutations $(\sigma, (\alpha_1, \alpha_2, \alpha_3)) =  (\sigma, T)$ will be called an {\it
isostrophism} or an {\it isostrophy } of a quasigroup $(Q,A)$. We can re-write equality from Definition
\ref{D2.3} in the form $({A}^{\sigma}) T = B$.

\begin{lemma} \label{ISOSTROPHY}
An isostrophic image of a quasigroup is a quasigroup \cite{2}.
\end{lemma}
\begin{proof}
The proof follows from the fact that any parastrophic image of a quasigroup is a quasigroup and any isotopic
image of a quasigroup is a quasigroup.
\end{proof}

From Lemma \ref{ISOSTROPHY} it follows that it is possible to  define the multiplication of isostrophies of a
quasigroup operation defined on a set $Q$.

\begin{definition}\label{MULT_ISOS}
If $(\sigma, S)$ and $(\tau,T)$ are isostrophisms of a quasigroup $(Q,A)$, then $$(\sigma,S) (\tau, T) = (\sigma
\tau, S^{\tau}T),$$  where $A^{\sigma \tau}  = (A^\sigma)^{\tau}$  and $(x_{1}, x_{2},x_{3})(S^{\tau}T) =
((x_{1}, x_{2}, x_{3})$ $S^{\tau})T$  for any quasigroup triplet $(x_{1}, x_{2}, x_{3})$ \cite{SCERB_08}.
\end{definition}

 Slightly other  operation on the set of all  isostrophies (multiplication of quasigroup isostrophies)
is defined in \cite{2}.   Definition from \cite {IVL} is very close to  Definition \ref{MULT_ISOS}. See, also,
\cite{BELAS, ks3}.

\begin{corollary} \label{PERESTAN_ISOSTR}
 $(\varepsilon, S)(\tau,\varepsilon) = (\tau, S^{\tau}) = (\tau, \varepsilon)(\varepsilon, S^{\tau})$.
\end{corollary}

\begin{lemma} \label{INVERSE_ISOSTROPHY}
$(\sigma, S)^{-1} = (\sigma^{-1}, (S^{-1})^{\sigma^{-1}}).$
\end{lemma}
\begin{proof}
 Let $S = (\alpha_1, \alpha_2, \alpha_3)$ be an isotopy of a quasigroup $A$, $S^{-1} =
(\alpha_1^{-1}, $ $\alpha_2^{-1}, \alpha_3^{-1}),$ $S^{\sigma} = $ $ (\alpha_{\sigma^{-1} 1},$ $
\alpha_{\sigma^{-1} 2}, \alpha_{\sigma^{-1} 3}).$ Then
 $$
 \begin{array}{l}
 (\sigma, S) (\sigma^{-1}, (S^{-1})^{\sigma^{-1}}) = (\varepsilon^{\prime},
S^{\sigma^{-1} }(S^{-1})^{\sigma ^{-1}}) \overset{(Lemma \, \, \ref{L2.1})}{=} \\ (\varepsilon^{\prime},
(SS^{-1})^{\sigma ^{-1}}) = (\varepsilon^{\prime}, (\varepsilon, \varepsilon, \varepsilon)).
\end{array}
$$
\end{proof}

\subsection{Group isotopes and identities}

Information for this subsection has taken from  \cite{AC_BEL_HOS, FOUR_QUAS_TH, 2, vdb1, KRAPEZ_80, TABAR_92,
SOH_07}. We formulate famous Four quasigroups theorem \cite{AC_BEL_HOS, FOUR_QUAS_TH, 2, SOH_07} as follows.
\begin{theorem} \label{O_CHETYREH_QUAS}
A quadruple $(f_1, f_2, f_3, f_4)$ of binary quasigroup operation defined on a non\-empty  set $Q$  is the
general
 solution of the generalized associativity  equation
$$
A_1(A_2(x,y), z) = A_3 (x, A_4(y,z))
$$
if and only if there exists a group $(Q, +)$ and permutations $\alpha, \beta, \gamma, \mu, \nu$ of the set $Q$
such that $f_1(t,z) = \mu t + \gamma z$, $f_2(x,y) =\mu^{-1}( \alpha x + \beta y)$, $f_3 (x, u) = \alpha x + \nu
u$, $f_4(y,z) = \nu^{-1}(\beta y + \gamma z)$.
\end{theorem}

\begin{lemma} \label{COMMUT_VDB} Belousov criteria.
If in a group $(Q, +)$  the equality $\alpha x + \beta y = \gamma y + \delta x$ holds for all $x, y \in Q$,
where $\alpha,  \beta, \gamma, \delta$ are some fixed permutations of $Q$, then $(Q, +)$ is an abelian group
\cite{vdb1}.
\end{lemma}

There exists  also  the following adapted for our aims corollary from results of F.N.~So\-kha\-ts\-kii
(\cite{SOH_07}, Theorem 6.7.2).

\begin{corollary} \label{SOHA_REZ}
If in a principal group isotope $(Q, \cdot)$ of a group $(Q,+)$    the equality $\alpha x \cdot \beta y = \gamma
y \cdot \delta x$ holds for all $x, y \in Q$, where $\alpha,  \beta, \gamma, \delta$ are some fixed permutations
of $Q$, then $(Q, +)$ is an abelian group.
\end{corollary}
\begin{proof}
If $x\cdot y = \xi x + \chi y$, then we can re-write the equality $\alpha x \cdot \beta y = \gamma y \cdot
\delta x$ in the form $\xi \alpha x +\chi  \beta y = \xi \gamma y + \chi \delta x$. Now we can apply Belousov
criteria (Lemma \ref{COMMUT_VDB}).
\end{proof}

\begin{lemma}\begin{enumerate} \label{FORM_S_NULEM}
    \item For any principal group  isotope $(Q,\cdot)$ there exists its form  $x\cdot y = \alpha x + \beta y$ such that $\alpha\, 0 = 0$ \cite{SOH_95_I}.
\item
 For any principal group  isotope $(Q,\cdot)$ there exists its form  $x\cdot y = \alpha x +
\beta y$ such that $\beta\, 0 = 0$.
    \item For any  right linear quasigroup $(Q,\cdot)$ there exists its form  $x\cdot y = \alpha x + \psi y + c$ such that
$\alpha 0 = 0$.
    \item For any  left linear quasigroup $(Q,\cdot)$ there exists its form  $x\cdot y = \varphi x + \beta y + c$ such that
$\beta 0 = 0$.
    \item For any  left linear quasigroup $(Q,\cdot)$ with idempotent element $0$  there exists its form
    $x\cdot y = \varphi x + \beta y$ such that
$\beta 0 = 0$.
    \item For any  right  linear quasigroup $(Q,\cdot)$ with idempotent element $0$ there exists its form  $x\cdot y = \alpha  x + \psi y$ such that
$\alpha 0 = 0$.
\end{enumerate}
\end{lemma}

\begin{proof}
1). We have  $x\cdot y = \alpha x + \beta y = R_{-\alpha 0} \alpha x + L_{\alpha 0} \beta y = \alpha^{\prime} x
+ \beta^{\prime} y$, $\alpha^{\prime} 0 = 0$.

2). We have  $x\cdot y = \alpha x + \beta y = R_{\beta 0} \alpha x + L_{-\beta 0} \beta y = \alpha^{\prime} x +
\beta^{\prime} y$, $\beta^{\prime} 0 = 0$.

3). We have  $x\cdot y = \alpha x + \psi y + c = R_{-\alpha 0} \alpha x + I_{\alpha 0} \psi y + \alpha 0 + c =
\alpha^{\, \prime} x + \psi^{\, \prime} y + c^{\, \prime}$, where $I_{\alpha 0} \psi y = \alpha 0 + \psi y -
\alpha 0$, $\alpha^{\prime} 0 = 0$.  Since $I_{\alpha 0}$ is an inner automorphism of the group $(Q,+)$, we
obtain  $I_{\alpha 0} \psi \in Aut(Q, +)$.

4). We have $x\cdot y = \varphi x + \beta y + c = \varphi x  +  \beta y -\beta 0 + \beta 0 + c = \varphi x  +
R_{ -\beta 0}\beta y + \beta 0 + c = \varphi x  + \beta^{\, \prime} y +  c^{\,\prime}$, where $\beta^{\, \prime}
= R_{ -\beta 0}\beta$, $c^{\,\prime} = \beta 0 + c$.

5). If  $x\cdot y = \varphi x + \beta y + c$, then $0 = 0 \cdot 0 = \varphi\, 0 + \beta \, 0 + c = \beta \, 0 +
c$, $\beta \, 0 = - c$. Therefore $x\cdot y = \varphi x + R_c \beta y = \varphi x + \beta^{\prime}y$ and
$\beta^{\, \prime} 0 = R_c \beta 0 = -c +c =0$.

6). If  $x\cdot y = \alpha  x + \psi y + c$, then $0 = 0 \cdot 0 = \alpha \, 0 + \psi \, 0 + c = \alpha \, 0 +
c$, $\alpha \, 0 = - c$. Therefore $x\cdot y = \alpha  x + c - c + \psi y + c = R_c \alpha  x + I_{-c} \psi y =
R_c \alpha^{\prime}   x + \psi^{\prime} y$ and $\alpha ^{\, \prime} 0 = R_c \alpha  0 = -c +c =0$. Moreover,
$\psi^{\prime}$ is an automorphism of $(Q,+)$ as the product of two automorphisms of the group $(Q,+)$.
\end{proof}

\begin{lemma} \label{FORMS_LEFT_LIN_QUAS}
For any left linear quasigroup  $(Q,\cdot)$ there exists its form such that  $x\cdot y = \varphi x + \beta y$.

For any right linear quasigroup  $(Q,\cdot)$ there exists its form such that  $x\cdot y = \alpha x + \psi y$.
\end{lemma}
\begin{proof}
We can re-write the form $x\cdot y = \varphi x + \beta y+c$ of a left linear quasigroup  $(Q,\cdot)$ as follows
$x\cdot y = \varphi x + R_c\beta y = \varphi x + \beta^{\,\prime} y,$ where $\beta^{\,\prime} = R_c\beta$.

We can re-write the form $x\cdot y = \alpha  x + \psi y+c$ of a right  linear quasigroup  $(Q,\cdot)$ as follows
$x\cdot y = \alpha  x + c -c + \psi y+c  = R_c\alpha  x + I_c\psi y = \alpha^{\,\prime} x + \psi^{\, \prime} y$,
where $I_{-c}\psi y = -c + \psi y +c$.
\end{proof}

Classical criteria of a linearity of a quasigroup are given by V.D. Belousov in \cite{vdb1}.   We give a partial
case of F.N.~Sokhatskii  result (\cite{SOH_07}, Theorem 6.8.6; \cite{SOH_99}, Theorem 3; \cite{SOH_95_II}).

We recall, up to isomorphism any isotope is principal (Remark \ref{REMARK_ON_PRINC_ISOT}).

\begin{theorem} \label{LEFT_LIN_CRITER}
Let   $(Q, \cdot)$ be a principal  isotope of a group $(Q,+)$, $x\cdot y = \alpha x + \beta y$.

If  $(\alpha_1 x \cdot \alpha_2 y)\cdot a = \alpha_3 x \cdot \alpha_4 y$ is true for all $x, y\in Q$, where
  $\alpha_1,\alpha_2,\alpha_3, \alpha_4$  are permutations  of the set $Q$, $a$ is a fixed element of the set $Q$, then
 $(Q, \cdot)$  is a  left linear quasigroup.

If  $a \cdot (\alpha_1 x \cdot \alpha_2 y) = \alpha_3 x \cdot \alpha_4 y$ is true for all $x, y\in Q$, where
  $\alpha_1,\alpha_2,\alpha_3, \alpha_4$  are permutations  of the set $Q$, $a$ is a fixed element of the set $Q$, then
 $(Q, \cdot)$  is a  right linear quasigroup.
\end{theorem}
\begin{proof}
We follow \cite{SOH_07}. By Lemma \ref{FORM_S_NULEM} quasigroup $(Q, \cdot)$ can have the  form $x\cdot y =
\alpha x + \beta y$ over a group $(Q,+)$ such that $\alpha \, 0=0$. If we pass in the equality $(\alpha_1 x
\cdot \alpha_2 y)\cdot a = \alpha_3 x \cdot \alpha_4 y$ to the operation \lq\lq + \rq\rq, then we obtain $\alpha
(\alpha \alpha_1 x + \beta \alpha_2 y) + \beta a = \alpha \alpha_3 x + \beta \alpha_4 y$, $\alpha (x + y)  =
\alpha \alpha_3 \alpha_1^{-1} \alpha^{-1}
 x + \beta \alpha_4  \alpha_2^{-1} \beta^{-1} y -  \beta a.$

Then the permutation $\alpha$ is a group  quasiautomorphism. It is known that any group quasiautomorphism has
the form $L_a \varphi$, where $\varphi \in Aut(Q, +)$. See  \cite{1a, VD} or  Corollary \ref{QUASIAUT_FORM}.
Therefore $\alpha \in Aut(Q, +)$, since $\alpha 0 =0$.

By Lemma \ref{FORM_S_NULEM} there exists  the form $x\cdot y = \alpha x + \beta y$  of quasigroup $(Q, \cdot)$
such that $\beta \, 0=0$. If we pass in the equality  $a \cdot (\alpha_1 x \cdot \alpha_2 y) = \alpha_3 x \cdot
\alpha_4 y$ to the operation \lq\lq + \rq\rq, then we obtain $\alpha a + \beta (\alpha \alpha_1 x + \beta
\alpha_2  y) = \alpha \alpha_3 x + \beta \alpha_4 y$, $\beta (x + y)  = -\alpha a+  \alpha \alpha_3
\alpha_1^{-1} \alpha^{-1}
 x + \beta \alpha_4  \alpha_2^{-1} \beta^{-1} y.$

Then the permutation $\beta$ is a group  quasiautomorphism.  Therefore $\beta \in Aut(Q, +)$, since $\beta 0
=0$.
\end{proof}
\begin{corollary} \label{LEFT_GROUP_LINEARITY}
If a left F-quasigroup (E-quasigroup, SM-qua\-si\-gro\-up) is a group isotope, then this quasigroup is right
linear.

If a right F-quasigroup (E-quasigroup, SM-quasigroup) is a group isotope, then this quasigroup is left linear
\cite{SOH_99}.
\end{corollary}
\begin{proof}
The proof follows from Theorem \ref{LEFT_LIN_CRITER}.
\end{proof}

\begin{lemma} \label{COMMUTATIV_LIN_QUAS}
1. If in a right  linear quasigroup  $(Q,\cdot)$ over a group $(Q,+)$ the equality $k\cdot yx = xy\cdot b$ holds
for all $x, y \in Q$ and fixed $k,b\in Q$,  then $(Q, +)$ is an abelian group.

2. If in a left  linear quasigroup  $(Q,\cdot)$ over a group $(Q,+)$ the equality $k\cdot yx = xy\cdot b$ holds
for all $x, y \in Q$ and fixed $k,b\in Q$,  then $(Q, +)$ is an abelian group.
\end{lemma}
\begin{proof}
1. By Lemma \ref{FORMS_LEFT_LIN_QUAS} we can take the following form of $(Q,\cdot)$:  $x\cdot y = \alpha x +
\psi y$. Thus  we have $\alpha k + \psi (\alpha y + \psi x) = \alpha (\alpha x + \psi y) + \psi b$, $\alpha k +
\psi \alpha y + \psi^2 x - \psi b = \alpha (\alpha x + \psi y)$, $\alpha (x+y) = \alpha k + \psi \alpha
\psi^{-1} y + \psi^2 \alpha^{-1} x - \psi b$. Therefore $\alpha$ is a quasiautomorphism of the group $(Q,+)$.
Let $\alpha = L_d \varphi$, where $\varphi \in Aut(Q, +)$.

Further we have $d + \varphi x + \varphi y = \alpha k + \psi \alpha \psi^{-1} y + \psi^2 \alpha^{-1} x - \psi
b$, $L_d\varphi x + \varphi y = L_{\alpha k}\psi \alpha \psi^{-1} y + R_{ - \psi b}\psi^2 \alpha^{-1} x$.
Finally, we can apply Lemma \ref{COMMUT_VDB}.

Case 2 is proved in the similar way.
\end{proof}

Quasigroup $(Q, \cdot)$ with equality  $xy \cdot z = x \cdot (y \circ z)$ for all $x, y, z \in Q$ is called
\lq\lq quasigroup which fulfills Sushkevich postulate A\rq\rq.

Quasigroup $(Q, \cdot)$ with equality  $x \cdot y z = (x \circ y) \cdot z$ for all $x, y, z \in Q$  will be
called \lq\lq quasigroup which fulfills Sushkevich postulate A$^{\ast}$\rq\rq.

\begin{theorem} \label{SUSHKEVICH_POST_A}
1. If quasigroup $(Q, \cdot)$  fulfills Sushkevich postulate A, then $(Q, \cdot)$ is isotopic to the group $(Q,
\circ)$,
 $(Q, \cdot) = (Q,\circ)(\varphi, \varepsilon, \varphi)$ (\cite{1a}, Theorem 1.7).

2. If quasigroup $(Q, \cdot)$  fulfills Sushkevich postulate A$^{\ast}$, then $(Q, \cdot)$ is isotopic to the
group $(Q, \circ)$,
 $(Q, \cdot) = (Q,\circ)(\varepsilon, \psi, \psi)$.
\end{theorem}
\begin{proof}
Case 1 is proved in \cite{1a}.

The proof of Case 2 is similar to the proof of Case 1. It is easy to see that $(Q, \circ)$ is quasigroup.
Indeed, if $ z = c$, then we have $x \cdot R^{\cdot}_c y  = R^{\cdot}_c (x \circ y)$, $(Q, \circ)$ is isotope of
quasigroup $(Q, \cdot)$. Therefore $(Q, \circ)$ is a quasigroup. Moreover, $x \cdot y  = R_c (x \circ R^{-1}_c
y)$, $(Q, \cdot) = (Q,\circ)(\varepsilon, \psi, \psi)$, where $\psi = R^{-1}_c$.

Quasigroup $(Q, \circ)$ is a group. It is possible to use Theorem \ref{O_CHETYREH_QUAS} but we give direct proof
 similar to the proof from \cite{1a}. We have $(x\circ (y\circ z))\cdot w = x\cdot ((y\circ z)\cdot w) =
 x\cdot (y\cdot (z\cdot w)) =  (x\circ y)\cdot (z\cdot w) = ((x\circ y)\circ z)\cdot w$,
 $x\circ (y\circ z) = (x\circ y)\circ z$.
\end{proof}

Quasigroup $(Q, \cdot)$ with generalized identity $xy \cdot z = x \cdot y \delta (z)$, where $\delta$ is a fixed
permutation of the set $Q$, is called \lq\lq quasigroup which fulfills Sushkevich postulate B\rq\rq.

Quasigroup $(Q, \cdot)$ with generalized identity  $x \cdot y z = (\delta(x) \cdot y) \cdot z$, where $\delta$
is a fixed permutation of the set $Q$, will be called \lq\lq quasigroup which fulfills Sushkevich postulate
B$^{\ast}$\rq\rq.

It is easy to see that any quasigroup with postulate B (B$^{\ast}$) is a quasigroup with postulate A
(A$^{\ast}$).

\begin{theorem} \label{SUSHKEVICH_POST_B}
1. If quasigroup $(Q, \cdot)$  fulfills Sushkevich postulate B, then $(Q, \cdot)$ is isotopic to the group
$(Q,\circ)$,
 $(Q, \cdot) = (Q,\circ)(\varepsilon, \psi, \varepsilon)$, where $\psi \in Aut(Q,\circ)$, $\psi \in Aut(Q,\cdot)$
 (\cite{1a}, Theorem 1.8).

2. If quasigroup $(Q, \cdot)$  fulfills Sushkevich postulate B$^{\, \ast}$, then $(Q, \cdot)$ is isotopic to the
group $(Q,\circ)$,
 $(Q, \cdot) = (Q,\circ)(\varphi, \varepsilon,  \varepsilon)$, where $\varphi \in Aut(Q,\circ)$, $\varphi \in Aut(Q,\cdot)$.
\end{theorem}
\begin{proof}
Case 1 is proved in \cite{1a}. It is easy to see that quasigroup $(Q, \cdot)$ has the right identity element,
i.e. $(Q, \cdot)$ is right loop. Indeed,  $x\cdot 0 = x \circ \psi \, 0 = x$ for all $x\in Q$, where $0$ is zero
of group $(Q, \circ)$.

Case 2.  The proof of Case 2 is similar to the proof of Case 1. Here
 we give the  direct proof because the book \cite{1a} is rare. Since the
quasigroup $(Q, \cdot)$ fulfills postulates A$^{\, \ast}$ and B$^{\, \ast}$, then by Theorem
\ref{SUSHKEVICH_POST_A}, Case 2, groupoid (magma)  $(Q, \circ)$, $x\circ y =  \delta (x) \cdot y$, is a group
and $(Q, \cdot) = (Q, \circ) (\delta^{-1}, \varepsilon,  \varepsilon)$. By the same theorem $(Q, \cdot) =
(Q,\circ)(\varepsilon, \psi, \psi)$. Therefore $(\delta, \psi, \psi)$ is an autotopy of the group $(Q, \circ)$.
By Corollary \ref{RAVNYE_KOMPON_AVTOT} $\delta \in Aut(Q,\circ)$. Therefore $\varphi = \delta^{-1}\in Aut (Q,
\circ)$.

It is easy to see that  $(Q, \cdot)$ is left loop.
\end{proof}

\section{Direct decompositions}

\subsection{Left and right  F-quasigroups}\label{SECTION_LEFT_F_QUAS}

In order to study the structure of  left F-quasigroups we shall use approach from \cite{4, SC05}. As usual
$e(e(x)) = e^2(x)$ and so on.

\begin{lemma} \label{EM_END}
1. In a left F-quasigroup $(Q, \cdot)$ the map $e^i$ is an  endomorphism of $(Q,\cdot)$, $e^i(Q, \cdot)$ is a
subquasigroup of quasigroup $(Q,\cdot)$ for all suitable values of the index $i$ \cite{MURD_39, 1a}.

2. In a right F-quasigroup $(Q, \cdot)$ the map $f^i$ is an  endomorphism of $(Q,\cdot)$, $f^i(Q, \cdot)$ is a
subquasigroup of quasigroup $(Q,\cdot)$ for all suitable values of the index $i$  \cite{MURD_39, 1a}.
\end{lemma}
\begin{proof}
1. From identity  $x\cdot y z = x y \cdot e(x) z$ by $z=e(y)$ we have $xy  = x y \cdot e(x) e(y)$, i.e.
$e(x\cdot y) = e(x) \cdot e(y)$. Further we have $e^2(x\cdot y) = e(e(x\cdot y)) = e(e(x) \cdot e(y)) = e^2(x)
\cdot e^2(y) $ and so on. Therefore $e^m$ is an endomorphism of the quasigroup $(Q, \cdot)$.

The fact that  $e^m (Q, \cdot)$ is a subquasigroup of quasigroup $(Q, \cdot)$ follows from Lemma \ref{l_5.4}.

2. The proof is similar.
\end{proof}

The proof of the following lemma has taken from \cite{1a}, p. 33.
\begin{lemma} \label{LEMMA_3_F_QUAS}
1. Endomorphism  $e$ of a left F-quasigroup $(Q, \cdot)$ is zero endomorphism, i.e. $e(x) = k$ for all $x\in Q$,
if and only if  left F-quasigroup $(Q, \cdot)$ is a right loop, isotope of a group $(Q,+)$ of the form  $(Q,
\cdot) = (Q,+)(\varepsilon, \psi, \varepsilon)$, where  $\psi \in Aut (Q,+)$, $k = 0$.

2. Endomorphism  $f$ of a right F-quasigroup $(Q, \cdot)$ is zero endomorphism, i.e. $f(x) = k$ for all $x\in
Q$, if and only if  right F-quasigroup $(Q, \cdot)$ is a left loop, isotope of a group $(Q,+)$ of the form $(Q,
\cdot) = (Q,+)(\varphi, \varepsilon, \varepsilon)$, where  $\varphi \in Aut (Q,+)$, $k = 0$.
\end{lemma}
\begin{proof}
1. We can rewrite equality $x\cdot yz = xy \cdot L_kz$ in the form $xy \cdot z = x(y L^{-1}_kz) = x(y \cdot
\delta z)$, where $\delta  = L^{-1}_k$. Therefore   Sushkevich postulate B  is fulfilled in  $(Q, \cdot)$ and we
can apply Theorem \ref{SUSHKEVICH_POST_B}. Further we have $x\cdot \, 0  = x+ \psi \, 0 = x$. From the other
side $x\cdot k = x$. Therefore, $k = 0$. It is easy to see that converse also is true.

2. We can use \lq\lq mirror\rq\rq \, principles.
\end{proof}

\begin{lemma} \label{LEMMA_4_ED}
1. The endomorphism  $e$ of a left F-quasigroup $(Q, \cdot)$ is a permutation of the set $Q$ if and only if
quasigroup $(Q, \circ)$ of the form $x\circ y = x \cdot e(y)$ is a left distributive quasigroup and $e\in Aut
(Q, \circ)$ \cite{1a}.

2. The endomorphism  $f$ of a right F-quasigroup $(Q, \cdot)$ is a permutation of the set $Q$ if and only if
quasigroup $(Q, \circ)$ of the form $x\circ y = f(x) \cdot y$ is a right  distributive quasigroup and $f\in Aut
(Q, \circ)$ \cite{1a}.
\end{lemma}
\begin{proof}
1. Prove that $(Q, \circ)$ is left distributive. We have
\begin{equation}\begin{split} \label{eq_2_1}
& x\circ (y\circ z) = x \cdot e(y\cdot e(z)) = x \cdot (e(y)\cdot e^2(z)) = \\ & (x\cdot e(y))\cdot (e(x)\cdot
e^2(z)) = \\ & (x\cdot e(y))\cdot e(x\cdot e(z)) = (x\circ y)\circ (x\circ z). \end{split} \end{equation}

Prove that $e\in Aut (Q, \circ)$. We have $e(x\circ y) = e(x\cdot e(y)) = e(x) \cdot e^2(y) = e(x)\circ e(y)$
\cite{MARS}.

Converse. Let $(Q,\cdot)$ be an  isotope of the form $x\cdot y = x \circ \psi (y)$, where $\psi \in
Aut(Q,\circ)$, of a left distributive quasigroup $(Q, \circ)$.  The fact that $\psi\in Aut(Q, \cdot)$ follows
from Lemma \ref{ON AUTOMORPHISM_OF_IDEM_QUAS}.

We can use equalities (\ref{eq_2_1}) by the proving that $(Q,\cdot)$ is a left F-quasigroup. The fact that $\psi
= e^{-1}$ follows from Lemma \ref{LEMMA_ON_GEN_F}.

2. The proof is similar.
\end{proof}

Define in a left F-quasigroup $(Q, \cdot)$  the following (maybe and infinite) chain
\begin{equation} \label{chain}
 Q \supset e(Q) \supset e^2(Q)
\supset \dots  \supset e^m(Q) \supset \dots
\end{equation}

\begin{definition}
Chain (\ref{chain}) becomes stable means that there exists a number  $m$ (finite or infinite) such that $e^m(Q)
= e^{m+1}(Q) = e^{m+2}(Q) \dots $. We notice, in other words $$ e^m (Q) = \bigcap_{i=1}^{\infty} e^i(Q) =
\underset{i\rightarrow \infty}{\lim}  e^i(Q). $$ In this case we shall say that endomorphism $e$ has the order
$m$.
\end{definition}

\begin{lemma}\label{l5.7} In any left F-quasigroup $Q$  chain (\ref{chain}) becomes stable,
i.e. the map $ e |_{e^m (Q)}$ is an automorphism of quasigroup $e^m (Q)$.
\end{lemma}
\begin{proof}
We have two cases. Case 1. Chain (\ref{chain}) becomes stable on a finite step $m$. It is clear that in this
case
 $ e |_{e^m (Q)}$ is an automorphism of $e^m (Q, \cdot)$.

Case 2. Prove that chain (\ref{chain}) will be stabilize  on the step  $m = \infty$, if it is not stabilized on
a finite step $m$. Denote $\bigcap_{i=1}^{\infty} e^i(Q)$ by $C$.

Notice, if $A\subseteq Q$, $B\subseteq Q$, then $e(A\cap B) \subseteq e(A)\cap e(B)$. Indeed, if $x\in A\cap B$,
then $e(x) \in e( A\cap B)$. If $x\in A\cap B$, then $x\in A$ and $x\in B$. Therefore $e(x) \in e(A)$ and $e(x)
\in e(B)$, $e(x)\in e(A)\cap e(B)$, $e(A\cap B) \subseteq e(A)\cap e(B)$.
 Then  $$
e(C) = e(\bigcap_{i=1}^{\infty} e^i(Q)) \subseteq \bigcap_{i=1}^{\infty} e^{i+1}(Q) = C.$$

Prove that $e(C) = C$. Any element $c\in C$ has the form $c = \underset{i\rightarrow \infty}{\lim}  e^i(a)$,
where $a\in Q$. Then $e(c) = e\left(\underset{i\rightarrow \infty}{\lim}  e^i(a)\right) = \underset{i\rightarrow
\infty}{\lim}  e^{\, i+1}(a)\in C$ for any $c\in C$. Therefore does not exist element $x$ of the set $C$ such
that $e(x)\notin e(C)$.

Therefore for  any $m$ (finite or infinite) $ e |_{e^m (Q)}$ is an automorphism of $e^m (Q, \cdot)$.
\end{proof}

\begin{example} \label{INFINITE_LEFT_F_QUAS}
Quasigroup  $(Z,\cdot)$, where $x\cdot y = -x + y$, $(Z,+)$ is infinite cyclic group, is medial, unipotent, left
F-quasigroup such that $e(x) =  x+x = 2x$.  Notice, in this case $Ker \, e = \{ 0\}$. In \cite{MALTSEV}, p. 59 a
mapping similar to the mapping $e$ is called isomorphism and the embedding of an algebra in its subalgebra.
\end{example}

\begin{theorem} \label{MAIN_LEFT_F}
1. Any  left F-quasigroup $(Q, \cdot)$ has the following structure
$$
(Q, \cdot) \cong (A, \circ) \times (B, \cdot),
$$
where $(A, \circ)$ is a quasigroup with a unique idempotent element; $(B, \cdot)$ is isotope of a left
distributive quasigroup $(B, \star)$, $x \cdot y = x \star \psi y$ for all $x, y \in B$, $\psi \in Aut(B,
\cdot)$, $\psi \in Aut(B, \star)$.

2.  \label{MAIN_RIGHT_F} Any  right  F-quasigroup $(Q, \cdot)$ has the following structure
$$
(Q, \cdot) \cong (A, \circ) \times (B, \cdot),
$$
where $(A, \circ)$ is a quasigroup with a unique idempotent element; $(B, \cdot)$ is isotope of a right
distributive quasigroup $(B, \star)$, $x \cdot y = \varphi x \star  y$ for all $x, y \in B$, $\varphi \in Aut(B,
\cdot)$, $\varphi \in Aut(B, \star)$.
\end{theorem}
\begin{proof}
The proof of this theorem mainly repeats the proof of Theorem 6 from \cite{SC05}.

If the map $e$ is a permutation of the set $Q$, then  by Lemma \ref{LEMMA_4_ED} $(Q,\cdot)$ is isotope of left
distributive  quasigroup.

If  $e(Q) = k$, where $k$ is a fixed element of the set $Q$, then the quasigroup $(Q,\cdot)$ is a quasigroup
with right identity element $k$, i.e. it is a right loop, which is isotopic to a group $(Q, +)$ (Lemma
\ref{LEMMA_3_F_QUAS}).

Let us to suppose that $e^m = e^{m+1}$, where $m > 1$.

From Lemma \ref{l5.7} it follows that $e^m(Q,\cdot) = (B,\cdot)$ is a  subquasigroup of quasigroup $(Q,\cdot)$.
It is clear that $(B,\cdot)$ is a left  F-quasigroup in which the map $\overline{e} = e |_{e^m (Q)}$ is a
permutation of the set $B \subset Q$. In other words $e(B, \cdot) = (B, \cdot)$.

 Define binary relation $\delta$ on  quasigroup $(Q,\cdot)$ by the following rule: $x\delta y$ if and
only if $e^m(x) = e^m(y)$. Define binary relation $\rho$  on  quasigroup $(Q,\cdot)$ by the rule: $x\rho y$ if
and only if  $B \cdot  x = B \cdot y$, i.e. for any $b_1 \in B$ there exists exactly one element $b_2 \in B$
such that $b_1\cdot  x = b_2 \cdot y$ and vice versa,  for any $b_2 \in B$ there exists exactly one element $b_1
\in B$ such that $b_1\cdot  x = b_2 \cdot y$.

From Theorem \ref{NORM_QUAS_CONGR} and Lemma \ref{l_5.4} it follows that $\delta$ is a normal congruence.

It is easy to check that  binary relation  $\rho$ is  equivalence relation (see Theorem \ref{EQUIVALENCE_TH}).

We prove that binary relation $\rho$ is a  congruence, i.e. that the following implication is true: $ x_1 \rho
y_1, $ $ x_2 \rho y_2,  \Longrightarrow  (x_1 \cdot x_2) \rho (y_1 \cdot y_2).$

Using the definition of relation $\rho$ we can re-write the last implication in the following equivalent form:
if
\begin{equation} \label{eqno(14_1)}
\begin{split}
 B \cdot x_1 = B \cdot y_1,
  B \cdot x_2 = B \cdot y_2,
\end{split}
\end{equation}
 then $B \cdot (x_1 \cdot x_2) = B \cdot (y_1 \cdot y_2)$.

If we multiply both sides of equalities (\ref{eqno(14_1)}), respectively,  then we obtain the following equality
\begin{equation*} \begin{split}
(\overset{x}{B} \cdot \overset{y}{x_1}) \cdot (\overset{e(x)}B \cdot \overset{z}{x_2}) = (B \cdot y_1)\cdot (B
\cdot y_2).
\end{split}
\end{equation*}

Using left F-quasigroup equality ($x\cdot y z = x y \cdot e(x) z$) from the right to the left and, taking into
consideration that if $x\in B$, then $e(x) \in B$, i.e. $e B = B$,  we can re-write the last equality in the
following form
\begin{equation*}
B \cdot (x_1 \cdot x_2) = B \cdot (y_1 \cdot y_2),  \label{eqno(15_1)}
\end{equation*}
since $(B, \cdot)$ is a subquasigroup and, therefore,  $B\cdot B = B$.  Thus the binary relation $\rho$ is a
congruence.

Prove that $\delta \cap \rho = \hat{Q} = \{(x,x) \vert  \, \forall \, x \in Q\}$. From reflexivity of relations
$\delta, \rho$ it follows  that $\delta \cap \rho \supseteq \hat{Q}.$

Let $(x,y) \in \delta \cap \rho$, i.e. let  $x\ \delta \ y$ and $x\ \rho \ y$ where $x, y \in Q.$ Using the
definitions of relations $\delta, \rho$ we have $e^m(x) = e^m(y)$ and $(B, \cdot) \cdot x = (B, \cdot)\cdot y$.
Then there exist $a, b \in B$ such that $a\cdot x = b \cdot y$. Applying to both sides of last equality the map
$e^m$ we obtain $e^m(a)\cdot e^m(x) = e^m(b) \cdot e^m(y)$, $e^m(a) = e^m(b)$, $a =b$, since the map $e^m |_B$
is a permutation of the set $B$. If $a =b$, then from equality $a\cdot x = b \cdot y$ we obtain $x=y$.

Prove that $\delta \circ \rho = Q\times Q.$  Let $a, c$ be any fixed elements of the set $Q$. We prove the
equality if it will be shown that  there exists  element $y\in Q$ such that $a\delta y$ and $y\rho c$.

From definition of congruence $\delta$  we have  that  condition $a\delta y$ is equivalent to equality $e^m (a)
= e^m (y).$ From definition of congruence $\rho$  it follows that condition $y\rho c$ is equivalent to the
following condition: $y \in \rho (c) = B\cdot c$.

We prove the equality if it will be shown that  there exists  element $y\in B\cdot c$ such that $e^m (a) = e^m
(y).$ Such element $y$ there exists since  $ e^m(B\cdot c) = e^m (B) \cdot e^m (c) = B = e^m(Q)$.

Prove that $\rho \circ \delta = Q\times Q.$  Let $a, c$ be any fixed elements of the set $Q$. We prove the
equality if it will be shown that  there exists  element $y\in Q$ such that $a\rho  y$ and $y \delta c$.

From definition of congruence $\delta$  we have  that  condition $y\delta c$ is equivalent to equality $e^m (c)
= e^m (y).$ From definition of congruence $\rho$  it follows that condition $a\rho y$ is equivalent to the
following condition: $y \in \rho (a) = B\cdot a$.

We prove the equality if it will be shown that  there exists  element $y\in B\cdot a$ such that $e^m (c) = e^m
(y).$ Such element $y$ there exists since  $ e^m(B\cdot a) = e^m (B) \cdot e^m (a) = B = e^m(Q)$.

Therefore $\rho \circ \delta = Q\times  Q  = \delta \circ \rho$, $\delta \cap \rho = \hat{Q}$  and we can use
Theorem \ref{t5.1}. Now we can say that  quasigroup  $(Q,\cdot)$ is isomorphic to the direct product of  a
quasigroup $(Q, \cdot) \slash \delta \cong (B, \cdot)$  (Theorem  \ref{NORM_QUAS_CONGR}) and a division groupoid
$(Q, \cdot) \slash \rho \cong (A, \circ)$ \cite{BK, RHB}.

From Definition \ref{EXTERNAL_DIRECT} it follows, if  $(Q, \cdot) \cong (B, \cdot) \times (A, \circ)$, where
$(Q, \cdot)$, $(B, \cdot)$ are quasigroups, then $(A, \circ) $ also is a quasigroup. Then by Theorem
\ref{NORM_QUAS_CONGR} the congruence $\rho$ is normal, $(B,\cdot) \trianglelefteqslant (Q,\cdot)$.

Left F-quasigroup equality  holds in quasigroup  $(B,\cdot)$ since  $(B, \cdot)\subseteq (Q,\cdot).$

If the quasigroups $(Q, \cdot)$ and $(B, \cdot)$ are left F-quasigroups, $(Q, \cdot) \cong (A, \circ) \times (B,
\cdot),$  then  $(A, \circ)$ also is  a left F-quasigroup (Lemma \ref{Componets_OF_DIR_PROD}).

Prove that the quasigroup $(A,\circ) \cong (Q,\cdot)/(B,\cdot)$, where  $e^m(Q,\cdot) = (B,\cdot)$, has  a
unique idempotent element.

We can identify elements of quasigroup $(Q,\cdot)/(B,\cdot)$ with cosets of the form $B\cdot c$, where $c\in Q$.

 From properties of quasigroup $(A,\circ)$  we have  that  $e^m
(A) = a$, where the element $a$ is a fixed element of the set $A$ that corresponds to the coset class $B$.
Further, taking into consideration the properties of endomorphism $e$ of the quasigroup $(A,\circ)$, we obtain
$e^{m+1}A = e(e^m A ) = e(a) = a$. Therefore $e(a)=a$, i.e. the element $a$ is an idempotent element of
quasigroup $(A,\circ)$.

Prove that there exists exactly one idempotent element in  quasigroup $(A,\circ)$. Suppose that there exists an
element $c$ of the set  $A$ such that $c\circ c = c$, i.e. such that $e(c) = c$. Then we have $e^m(c) = c = a$,
since $e^m (A) =a$.

The fact that $(B, \cdot)$ is isotope of a left distributive quasigroup $(B, \star)$ follows from Lemma
\ref{LEMMA_4_ED}.

2. Properties of right F-quasigroups coincide with the \lq\lq mirror\rq\rq \, properties of left
F-qua\-si\-gro\-ups.
\end{proof}

We notice,   in finite case all congruences are normal and permutable  (Lemmas \ref{l5.4} and
\ref{NORMAL_LEMMA}). Therefore  for finite case Theorem \ref{MAIN_LEFT_F} can be  proved in more short way.

We add some details on the structure of left F-quasigroup $(Q, \cdot)$. By $e^j (Q, \cdot)$ we denote
endomorphic image of the quasigroup $(Q,\cdot)$ relative to the endomorphism $e^j$.
\begin{corollary} \label{LEFT_F_NORMAL_SUBQUAS}
If $(Q, \cdot)$ is a left F-quasigroup, then   $e^m(Q,\cdot) \trianglelefteqslant (Q,\cdot)$.
\end{corollary}
\begin{proof}
This follows from the fact that the binary relation $\rho$ from Theorem \ref{MAIN_LEFT_F} is a normal congruence
in $(Q, \cdot)$ and subquasigroup  $e^m(Q,\cdot) = (B, \cdot)$ is an equivalence class of $\rho$.
\end{proof}

\begin{remark}\label{OBOZNACH_ENDOM}
For brevity we shall denote the endomorphism   $e|_{\, e^j(Q,\cdot)}$ such that  $$e|_{\, e^j(Q,\cdot)}:
e^j(Q,\cdot) \rightarrow e^{j+1}(Q,\cdot)$$  by $e_j$, the endomorphism   $f|_{\, f^j(Q,\cdot)}$ by $f_j$,  the
endomorphism   $s|_{\, s^j(Q,\cdot)}$ by $s_j$.
\end{remark}

\begin{corollary} \label{LEFT_QUAS_MATRESHKA}
If $(Q, \cdot)$ is a left F-quasigroup with an idempotent element, then equivalence class (cell)  $\bar a $ of
the normal congruence $Ker\, e_j$ containing an idempotent element $a\in Q$ forms linear right loop $(\bar a,
\cdot)$ for all suitable values of $j$.
\end{corollary}
\begin{proof}
By Lemma  \ref{BINARY_SUB_QUAS} $(\bar a, \cdot)$ is a quasigroup. From properties of the endomorphism $e$ we
have  that in $(\bar a, \cdot)$ endomorphism $e$ is zero endomorphism. Therefore in this case we can apply Lemma
\ref{LEMMA_3_F_QUAS}. Then $(\bar a, \cdot)$ is isotopic to a group with isotopy of the form $(\varepsilon,
\psi, \varepsilon)$, where $\psi \in Aut(\bar a, \cdot)$.
\end{proof}

\begin{corollary}
If $(Q, \cdot)$ is a right F-quasigroup, then   $f^m(Q,\cdot) \trianglelefteqslant (Q,\cdot)$.
\end{corollary}
\begin{proof}
The proof is similar to the proof of Corollary \ref{LEFT_F_NORMAL_SUBQUAS}.
\end{proof}

\begin{corollary} \label{RIGHTF_QUAS_MATRESHKA}
If $(Q, \cdot)$ is a right  F-quasigroup with an idempotent element, then equivalence class   $\bar a $ of the
normal congruence $Ker\, f_j$ containing an idempotent element $a\in Q$ forms linear left  loop $(\bar a,
\cdot)$ for all suitable values of $j$.
\end{corollary}
\begin{proof}
The proof is similar to the proof of Corollary \ref{LEFT_QUAS_MATRESHKA}.
\end{proof}

\subsection{Left and right  SM- and  E-quasigroups}

 We can  formulate  theorem on the structure of left semimedial quasigroup   using connections
between a quasigroup and its (23)-parastrophe (Lemma \ref{LEMMA_PARASTR}), but in order to have more information
about left semimedial quasigroup we prefer to give direct formulations some  results from section
\ref{SECTION_LEFT_F_QUAS}.

\begin{lemma} \label{EM_END_MIDD}
\begin{enumerate}
    \item In a left semimedial quasigroup $(Q, \cdot)$ the map $s^i$ is an  endomorphism of $(Q,\cdot)$, $s^i(Q, \cdot)$
is a subquasigroup of quasigroup $(Q,\cdot)$ for all suitable values of the index i \cite{MURD_39, 1a}.
    \item In a right semimedial quasigroup $(Q, \cdot)$ the map $s^i$ is an  endomorphism of $(Q,\cdot)$, $s^i(Q, \cdot)$
is a subquasigroup of quasigroup $(Q,\cdot)$ for all suitable values of the index i.
    \item In a left E-quasigroup $(Q, \cdot)$ the map $f^i$ is an  endomorphism of $(Q,\cdot)$, $f^i(Q, \cdot)$
is a subquasigroup of quasigroup $(Q,\cdot)$ for all suitable values of the index m \cite{Kin_PHIL_04}.
    \item In a right E-quasigroup $(Q, \cdot)$ the map $e^i$ is an  endomorphism of $(Q,\cdot)$, $e^i(Q, \cdot)$
is a subquasigroup of quasigroup $(Q,\cdot)$ for all suitable values of the index m \cite{Kin_PHIL_04}.
\end{enumerate}
\end{lemma}
\begin{proof}
Case 1. From identity  $x x \cdot y z = x y \cdot x z$ by $z= y$ we have $xx \cdot yy  = x y \cdot xy$, i.e.
$s(x) \cdot (y) = s(x\cdot y)$.  Therefore $s^i$ is an endomorphism of the quasigroup $(Q, \cdot)$.

The fact that  $s^i (Q, \cdot)$ is a subquasigroup of quasigroup $(Q, \cdot)$ follows from Lemma \ref{l_5.4}.

Case 2. The proof of Case 2 is similar to the proof of Case 1 and we omit them.

Case 3. From identity $x \cdot yz = f(x)y \cdot xz $ by $y=f(y)$, $z = y$ we have $xy = f(x)f(y)  \cdot xy $.
But $f(xy) \cdot xy = xy$. Therefore $f(x)\cdot f(y) = f(x\cdot y)$ \cite{Kin_PHIL_04}.

Case 4. From identity $zy \cdot x = zx \cdot y e(x) $ by $y=e(y)$, $z=y$ we have $y  x = yx \cdot e(y) e(x)$.
\end{proof}

\begin{theorem} \label{THEOREM_3_F_QUAS_MIDD}\begin{enumerate}
    \item If the endomorphism  $s$ of a left semimedial quasigroup $(Q, \cdot)$ is zero endomorphism, i.e. $s(x)=0$ for all
$x\in Q$, then  $(Q, \cdot)$ is an unipotent quasigroup, $(Q,\cdot) \cong (Q, \circ)$, where  $x\circ y = -
\varphi x + \varphi y$,  $(Q, +)$ is a group, $\varphi \in Aut (Q,+)$.
    \item If the endomorphism  $s$ of a right semimedial quasigroup $(Q, \cdot)$ is zero endomorphism, i.e. $s(x)=0$ for all
$x\in Q$, then  $(Q, \cdot)$ is an unipotent quasigroup, $(Q,\cdot) \cong (Q, \circ)$, where  $x\circ y =
\varphi x - \varphi y$,  $(Q, +)$ is a group, $\varphi \in Aut (Q,+)$.
    \item If the endomorphism  $f$ of a left E-quasigroup $(Q, \cdot)$ is zero endomorphism, i.e. $f(x)=0$ for all
$x\in Q$, then up to isomorphism  $(Q, \cdot)$ is a left  loop,  $x\cdot y = \alpha x + y$,  $(Q, +)$ is an
abelian group, $\alpha 0 =0$.
    \item If the endomorphism  $e$ of a right E-quasigroup $(Q, \cdot)$ is zero endomorphism, i.e. $e(x)=0$ for all
$x\in Q$, then up to isomorphism  $(Q, \cdot)$ is a right  loop,  $x\cdot y =  x +  \beta y$,  $(Q, +)$ is an
abelian group,  $\beta 0 =0$.
\end{enumerate}
\end{theorem}
\begin{proof} \textbf{Case 1.}
We can rewrite equality $xx \cdot yz = xy \cdot xz$ in the form $k  \cdot y z = x y \cdot x z$, where $s(x)=k$
for all $x\in Q$. If we denote $xz$ by $v$, then $z = x\backslash v$ and equality $k  \cdot y z = x y \cdot x z$
takes the form $k \cdot (y \cdot (x\backslash v)) = xy  \cdot v$, $k  \cdot (y \cdot (x\backslash v)) = (y\ast
x) \cdot v$.

Then  the last equality has  the form $A_1(y, A_2 (x, v)) = A_3 (A_4(y,x), v)$, where $A_1$, $A_2$, $A_3$, $A_4$
are quasigroup operations, namely, $A_1(y,t) = L_k (y \cdot t)$, $t= A_2 (x, v) = x\backslash v$, $A_3 (u, v) =
u\cdot v,$ $u = A_4 (y, x) = x\cdot  y$.

From Four quasigroups theorem  (Theorem \ref{O_CHETYREH_QUAS})  it follows that quasigroup $(Q, \cdot)$ is an
isotope of a group $(Q,+)$.

If in the equality  $k  \cdot y z = x y \cdot x z$ we fix the variable $x=b$, then we obtain the following
equality $k \cdot y z = b y \cdot b z$, $k \cdot y z = L_b y \cdot L_b z$. From Theorem \ref{LEFT_LIN_CRITER} it
follows that $(Q, \cdot)$ is a right linear quasigroup.

If in $k  \cdot y z = x y \cdot x z$ we put $x=z$, then we obtain $k  \cdot y x = x y \cdot k$. From Lemma
\ref{COMMUTATIV_LIN_QUAS} it follows that $(Q,+)$ is a commutative group.

From Lemma \ref{FORM_S_NULEM} we have that there exists a group $(Q,+)$ such that $x\cdot y = \alpha x + \psi y
+ c$, where $\alpha$ is a permutation of the set $Q$, $\alpha 0 =0$, $\psi \in Aut(Q,+)$.

Further we have $s(0) = k = 0\cdot 0 = c$, $k = c$. Then   $s(x) = k = x\cdot x =  \alpha x + \psi x +k$.
Therefore,  $\alpha x + \psi x =0$ for all $x\in Q$. Then  $\alpha = I\psi$, where $x + I (x) = 0$ for all $x\in
Q$. Therefore $\alpha$ is an antiautomorphism of the group $(Q,+)$, $x\cdot y = I\psi x + \psi y + k$.

Finally  $L^{-1}_k(L_k x\cdot L_k y) = L^{-1}_k (I\psi x + I \psi k + \psi k + \psi y +k) = L^{-1}_k (I\psi x +
 \psi y +k) = - k + I\psi x + k - k + \psi y + k = I_k I \psi x + I_k \psi y = I I_{k} \psi x + I_k \psi y = x \circ y$,
 where  $I_k x = - k + x + k$ is an inner automorphism of $(Q,+)$.
 It is easy to see that $s^{\circ} (x) = 0$ for all $x\in Q$.

Below we shall suppose that any left semimedial quasigroup $(Q, \cdot)$ with  zero endomorphism $s$ is an
unipotent quasigroup with the form  $x\cdot y = - \varphi x + \varphi y$, where $(Q, +)$ is a group, $\varphi
\in Aut (Q,+)$.

\medskip

\textbf{Case 2.} We can rewrite equality $zy \cdot s(x) = zx \cdot yx$ in the form $zy \cdot k = zx \cdot yx$,
where $s(x) =k$. If we denote $zx$ by $v$, then $z = v\slash x$ and the equality $zy \cdot k = zx \cdot yx$
takes the form $((v\slash x)y)k  = v  \cdot yx = v  \cdot (x\ast y)$.

We  re-write the last equality in the form $$A_1(A_2 (v, x), y) = A_3 (v, A_4(x,y)),$$ where $A_1$, $A_2$,
$A_3$, $A_4$ are quasigroup operations, namely, $A_1(t,y) = R_k (t \cdot y)$, $t= A_2 (v, x) = v\slash x$, $A_3
(v, u) = v\cdot u,$ $u = A_4 (x, y) = x\ast y$.

From Four quasigroups theorem  it follows that quasigroup $(Q, \cdot)$ is an isotope of a group $(Q,+)$.

If in the equality  $zy \cdot k = zx \cdot yx$ we fix the variable $x=b$, then we obtain the following equality
$zy \cdot k = zb \cdot yb$, $zy \cdot k = R_b z \cdot R_b y$. From Theorem \ref{LEFT_LIN_CRITER} it follows that
$(Q, \cdot)$ is a left  linear quasigroup.

If in the equality  $zy \cdot k = zx \cdot yx$ we put $x=z$, then we obtain $xy \cdot k = k \cdot yx$. Thus from
Lemma \ref{COMMUTATIV_LIN_QUAS} it follows that $(Q,+)$ is a commutative group.

From Lemma \ref{FORM_S_NULEM} we have that there exists a group $(Q,+)$ such that $x\cdot y = \varphi x + \beta
y + c$, where $\beta$ is a permutation of the set $Q$, $\beta 0 =0$, $\varphi  \in Aut(Q,+)$.

Further we have $s(0) = k = 0\cdot 0 = c$, $k=c$. Then   $s(x) = k = x\cdot x =  \varphi x + \beta x +k$.
Therefore,  $\varphi x + \beta x =0$ for all $x\in Q$. Then  $\beta = I\varphi$, where $Ix + x = 0$ for all
$x\in Q$.

Therefore $\beta = I\varphi \in Aut(Q,+)$, $x\cdot y = \varphi x - \varphi y + k$.

We  have $R^{-1}_k(R_k x\cdot R_k y) = R^{-1}_k (\varphi x + \varphi k - \varphi k - \varphi y +k) = \varphi x -
\varphi y + k - k = \varphi x - \varphi y = x\circ y$. It is easy to see that $s^{\circ} (x) = 0$ for all $x\in
Q$.

Below we shall suppose that any right semimedial quasigroup $(Q, \cdot)$ with  zero endomorphism $s$ is an
unipotent quasigroup with the form  $x\cdot y = \varphi x - \varphi y$, where $(Q, +)$ is a  group, $\varphi \in
Aut (Q,+)$.

\medskip

\textbf{Case 3.} We can rewrite the equality $x \cdot  yz = f(x)y \cdot xz $ in the form $x \cdot  yz = k y
\cdot xz =  y \cdot xz$, $x \cdot  (z\ast y) = xz \ast y$, where $f(x) = k$ for all $x\in Q$.

Then  $A_1(x, A_2 (z, y)) = A_3 (A_4(x, z),y)$, where $A_1$, $A_2$, $A_3$, $A_4$ are quasigroup operations,
namely, $A_1(x,t) = x \cdot t$, $t= A_2 (z, y) = z\ast y$, $A_3 (u, y) = u\ast y,$ $u = A_4 (x, z) = x\cdot z$.
From Four quasigroups theorem  it follows that quasigroup $(Q, \cdot)$ is a group isotope.

If in the equality $x \cdot  yz =  y \cdot xz $ we fix  variable $z$, i.e. if we take  $z = a$, then we have
$x\cdot R_a y = y \cdot  R_a x$. From  Corollary \ref{SOHA_REZ} it follows that the  group $(Q, +)$ is
commutative.

If in the equality $x \cdot  yz =  y \cdot xz $ we fix  variable $x$, i.e. if we take  $x = a$, then we have $a
\cdot  yz =  y \cdot az $, $a \cdot (yz) =  y \cdot L_az$. The application of Theorem \ref{LEFT_LIN_CRITER} to
the last equality gives us that $(Q, \cdot)$ is a right linear quasigroup, i.e. $x\cdot y = \alpha x + \psi y
+c$. 

Then  $f(x)\cdot x = k \cdot x = \alpha k + \psi x +c  = x$. By $x=0$ we have $\alpha k + \psi 0 +c = 0$,
$\alpha k = -c$. Therefore, $k \cdot x = x=  \psi x$  for all $x\in Q$. Then  $\psi  = \varepsilon$, $x\cdot y =
\alpha  x +  y + c = L_c\alpha x + y$ for all $x, y \in Q$. In other words  $x\cdot y = \alpha x + y$ for all
$x, y \in Q$.

Further let $a+\alpha 0 = 0$. Then $L^{-1}_a (L_a \alpha x + L_a y) = -a + a + \alpha x +a + y = a + \alpha x +
y = \alpha^{\, \prime} x + y = x\circ y$, where $\alpha^{\, \prime} = L_a\alpha$ and $\alpha^{\, \prime} 0 = 0$.

\medskip

\textbf{Case 4.} Case 4 is a \lq\lq mirror\rq\rq \, case of Case 3, but we give the direct proof. We can rewrite
 equality $zy \cdot x = zx \cdot y e(x) $ in the form $zy \cdot x = zx \cdot y k = zx \cdot y$, $(y \ast z)
\cdot x = y \ast zx$, where $e(x) = k$.

Then  $A_1 (A_2(y, z),x)= A_3(y, A_4 (z, x))$ , where $A_1$, $A_2$, $A_3$, $A_4$ are quasigroup operations,
namely, $A_1(t,x) = t \cdot x$, $t= A_2 (y, z) = y\ast z$, $A_3 (y, v) =  y\ast v,$ $v = A_4 (z, x) = z\cdot x$.

From Four quasigroups theorem  it follows that quasigroup $(Q, \cdot)$ is an  isotope of a group $(Q,+)$.

If in the equality $zy \cdot x = zx \cdot y $ we fix  variable $z$, i.e. if we take  $z = a$, then we have $L_a
y \cdot x = L_a x \cdot y $. From  Corollary \ref{SOHA_REZ} it follows that the group $(Q, +)$ is commutative.

If in the equality $zy \cdot x =  zx \cdot y$ we fix  variable $x$, i.e. if we take  $x = a$, then we have $zy
\cdot a =  za \cdot y$, $zy \cdot a =  R_a z \cdot y$. The application of Theorem \ref{LEFT_LIN_CRITER} to the
last equality gives us that $(Q, \cdot)$ is a left  linear quasigroup, i.e. $x\cdot y = \varphi x + \beta y +c$.

Then  $x\cdot e(x) = x \cdot k = \varphi x + \beta k +c  = x$. By $x = 0$ we have $\varphi 0 + \beta k + c = 0$,
$\beta k = - c$. Therefore, $x \cdot k = x =  \varphi x$  for all $x\in Q$. Then  $\varphi  = \varepsilon$,
$x\cdot y = x + \beta y + c = x + R_c \beta y$ for all $x, y \in Q$. In other words $x\cdot y = x + \beta y$ for
all $x, y \in Q$.

Further let $a+\beta 0 = 0$. Then $L^{-1}_a (L_a  x + L_a \beta y) = -a + a +  x +a + \beta y = x +  a + \beta y
=  x + \beta^{\, \prime} y = x\circ y$, where $\beta^{\, \prime} = L_a\beta$ and $\beta^{\, \prime} 0 = 0$.
\end{proof}

By proving of the following lemma we use ideas from  \cite{1a}.
\begin{lemma} \label{LEMMA_24_ED}
\begin{enumerate}
    \item If the endomorphism  $s$ of a left semimedial quasigroup $(Q, \cdot)$ is a permutation of the set $Q$, then
quasigroup $(Q, \circ)$ of the form $x\circ y = s^{-1}( x \cdot y)$ is a left distributive quasigroup and $s\in
Aut (Q, \circ)$.
    \item If the endomorphism  $s$ of a right semimedial quasigroup $(Q, \cdot)$ is a permutation of the set $Q$, then
quasigroup $(Q, \circ)$ of the form $x\circ y = s^{-1}( x \cdot y)$ is a right  distributive quasigroup and
$s\in Aut (Q, \circ)$.
    \item If the endomorphism  $f$ of a left E-quasigroup $(Q, \cdot)$ is a permutation of the set $Q$, then
quasigroup $(Q, \circ)$ of the form $x\circ y = f(x) \cdot y$ is a left distributive quasigroup and $f\in Aut
(Q, \circ)$.
    \item If the endomorphism  $e$ of a right E-quasigroup $(Q, \cdot)$ is a permutation of the set $Q$, then
quasigroup $(Q, \circ)$ of the form $x\circ y = x \cdot e(y)$ is a right distributive quasigroup and $e\in Aut
(Q, \circ)$.
\end{enumerate}
\end{lemma}
\begin{proof}
{Case 1.} We prove that $(Q, \circ)$ is left distributive. It is clear that $s^{-1}  \in Aut(Q, \cdot)$. We have
$$
\begin{array}{l}
 x\circ (y\circ z) = s^{-1}(x \cdot s^{-1}(y\cdot z)), \\
(x \circ y)\circ (x \circ z) = s^{-2}((x\cdot y)\cdot (x\cdot z)) = s^{-2}(s(x) \cdot (y \cdot z)) = \\ s^{-1}(x
\cdot s^{-1}(y\cdot z)).
\end{array}
$$
Prove that $s\in Aut (Q, \circ)$. We have $s(x\circ y) = x\cdot y$, $s(x)\circ s(y) = s^{-1} (s(x)\cdot s(y)) =
x\cdot y.$  See, also \cite{MARS}.

{Case 2.} We prove that $(Q, \circ)$ is right distributive. It is clear that $s^{-1}  \in Aut(Q, \cdot)$. We
have
$$
\begin{array}{l}
( x\circ y) \circ z = s^{-1}(s^{-1}(x \cdot y)\cdot z), \\
(x \circ z)\circ (y \circ z) = s^{-2}((x\cdot z)\cdot (y\cdot z)) = \\ s^{-2}((x \cdot y) \cdot s( z)) =
s^{-1}(s^{-1}(x \cdot y)\cdot z).
\end{array}
$$
Prove that $s\in Aut (Q, \circ)$. We have $s(x\circ y) = x\cdot y$, $s(x)\circ s(y) = s^{-1} (s(x)\cdot s(y)) =
x\cdot y.$

{Case 3.} If the endomorphism $f$ is a permutation of the set $Q$, then $f, f^{-1} \in Aut(Q,\cdot)$. We have
$$
\begin{array}{l}
 x\circ (y \circ z) = f(x) \cdot (f(y)\cdot z), \\
(x \circ y)\circ (x \circ z) = f(f(x)\cdot y)\cdot (f(x)\cdot z) = \\ (f^2(x)\cdot f(y))\cdot (f(x)\cdot z) =
f(x) \cdot (f(y)\cdot z).
\end{array}
$$
Prove that $f\in Aut (Q, \circ)$. We have $f(x\circ y) = f(f(x)\cdot y) = f^2(x) \cdot f(y) = f(x)\circ f(y)$.

{Case 4.} If the endomorphism $e$ is a permutation of the set $Q$, then $e, e^{-1} \in Aut(Q,\cdot)$. We have
$$
\begin{array}{l}
( x\circ y) \circ z = (x \cdot e(y))\cdot e(z), \\
(x \circ z)\circ (y \circ z) = (x\cdot e(z))\cdot e(y\cdot e(z)) = \\ (x\cdot e(z))\cdot (e(y)\cdot e^2(z)) = (x
\cdot e(y))\cdot e(z).
\end{array}
$$
Prove that $e\in Aut (Q, \circ)$. We have $e(x\circ y) = e(x\cdot e(y)) = e(x) \cdot e^2(y) = e(x)\circ e(y)$.
\end{proof}

 \begin{remark}
By the proof of Lemma \ref{LEMMA_24_ED} it is possible to use Lemma \ref{LEMMA_4_ED} and parastrophe invariant
arguments.
 \end{remark}

\begin{theorem} \label{MAIN_MIDDLE_F}
\begin{enumerate}
    \item Any  left SM-quasigroup $(Q, \cdot)$ has the following structure
$$
(Q, \cdot) \cong (A, \circ) \times (B, \cdot),
$$
where $(A, \circ)$ is a quasigroup with a unique idempotent element and there exists a  number $m$ such that
$|s^m(A, \circ)| =1$; $(B, \cdot)$ is an isotope of a left distributive quasigroup $(B, \star)$, $x \cdot y =
s(x \star y)$ for all $x, y \in B$, $s \in Aut(B, \cdot)$, $s \in Aut(B, \star)$.
    \item Any  right SM-quasigroup  $(Q, \cdot)$ has the following structure
$$
(Q, \cdot) \cong (A, \circ) \times (B, \cdot),
$$
where $(A, \circ)$ is a quasigroup with a unique idempotent element and there exists an ordinal number $m$ such
that $|s^m(A, \circ)| =1$; $(B, \cdot)$ is an isotope of a right distributive quasigroup $(B, \star)$, $x \cdot
y = s(x \star  y)$ for all $x, y \in B$, $s \in Aut(B, \cdot)$, $s \in Aut(B, \star)$.
    \item Any  left E-quasigroup $(Q, \cdot)$ has the following structure
$$
(Q, \cdot) \cong (A, \circ) \times (B, \cdot),
$$
where $(A, \circ)$ is a quasigroup with a unique idempotent element and there exists a number $m$ such that
$|f^m(A, \circ)| =1$; $(B, \cdot)$ is an isotope of a left distributive quasigroup $(B, \star)$, $x \cdot y =
f^{-1}(x) \star  y$ for all $x, y \in B$, $f \in Aut(B, \cdot)$, $f \in Aut(B, \star)$.
         \item Any  right E-quasigroup  $(Q, \cdot)$ has the following structure
$$
(Q, \cdot) \cong (A, \circ) \times (B, \cdot),
$$
where $(A, \circ)$ is a quasigroup with a unique idempotent element and there exists a number $m$ such that
$|e^m(A, \circ)| =1$; $(B, \cdot)$ is an isotope of a right distributive quasigroup $(B, \star)$, $x \cdot y = x
\star  e^{-1}(y)$ for all $x, y \in B$, $e \in Aut(B, \cdot)$, $e \in Aut(B, \star)$.
\end{enumerate}
\end{theorem}
\begin{proof}
The proof is similar to the proof of Theorem \ref{MAIN_LEFT_F}. It is possible also  to use parastrophe
invariance ideas.
\end{proof}

\begin{corollary}
If $(Q, \cdot)$ is a left SM-quasigroup, then   $s^m(Q,\cdot) \trianglelefteqslant (Q,\cdot)$; if $(Q, \cdot)$
is a right SM-quasigroup, then   $s^m(Q,\cdot) \trianglelefteqslant (Q,\cdot)$; if $(Q, \cdot)$ is a left
E-quasigroup, then   $f^m(Q,\cdot) \trianglelefteqslant (Q,\cdot)$; if $(Q, \cdot)$ is a right E-quasigroup,
then   $e^m(Q,\cdot) \trianglelefteqslant (Q,\cdot)$.
\end{corollary}

\begin{corollary} \label{SLOI_POLUMED}
If $(Q, \cdot)$ is a left SM-quasigroup with an idempotent element, then equivalence class   $\bar a $ of the
normal congruence $Ker\, s_j$ containing an idempotent element $a\in Q$ forms an unipotent quasigroup $(\bar a,
\cdot) $
 isotopic to a  group with isotopy of the form $(-\psi, \psi, \varepsilon)$, where   $\psi
\in  Aut(\bar a, \cdot)$   for all suitable values of $j$.

If $(Q, \cdot)$ is a right  SM-quasigroup with an idempotent element, then equivalence class   $\bar a $ of the
normal congruence $Ker\, s_j$ containing an idempotent element $a\in Q$ forms an unipotent quasigroup $(\bar a,
\cdot) $  isotopic to a  group with isotopy of the form $(\varphi, -\varphi, \varepsilon)$, where $\varphi \in
Aut(\bar a, \cdot)$   for all suitable values of $j$.

If $(Q, \cdot)$ is a left E-quasigroup with an idempotent element, then equivalence class $\bar a $ of the
normal congruence $Ker\, f_j$ containing an idempotent element $a\in Q$ forms a left loop
 isotopic to an abelian  group with isotopy of the form $(\alpha, \varepsilon,  \varepsilon)$
   for all suitable values of $j$.

If $(Q, \cdot)$ is a right E-quasigroup with an idempotent element, then equivalence class $\bar a $ of the
normal congruence $Ker\, e_j$ containing an idempotent element $a\in Q$ forms a right loop isotopic to an
abelian  group with isotopy of the form $(\varepsilon, \beta, \varepsilon)$ for all suitable values of $j$.
\end{corollary}
\begin{proof}
Mainly the proof repeats the proof of Corollary \ref{LEFT_QUAS_MATRESHKA}.  It is possible to use Theorem
\ref{THEOREM_3_F_QUAS_MIDD}.
\end{proof}

\subsection{CML as an  SM-quasigroup}

In this subsection we give information (mainly well known) about commutative Moufang loops (CML)  which  is
possible to obtain from the fact that a loop $(Q, \cdot)$ is  left semimedial  if and only if it is  a
commutative Moufang loop. Novelty of information from this subsection is in the fact that some well known
theorems about CML are obtained quit  easy using quasigroup approach.

\begin{lemma}
1. A left F-loop is a group. 2. A right F-loop is a group. 3. A loop $(Q, \cdot)$ is  left semimedial  if and
only if it is  a commutative Moufang loop. 4. A loop $(Q, \cdot)$ is right semimedial  if and only if it is  a
commutative Moufang loop. 5. A left E-loop $(Q, \cdot)$ is a commutative group. 6. A right E-loop $(Q, \cdot)$
is a commutative group.
\end{lemma}
\begin{proof}
Case 1. From $x \cdot  yz = xy \cdot e(x)z $ we have $x \cdot  yz = x \cdot yz$. Case 2. From $xy \cdot  z = x
f(z) \cdot yz $ we have $xy \cdot  z = x \cdot yz$.

Case 3. We use the proof from (\cite{VD}, p. 99). Let $(Q, \cdot)$ be a  left semimedial loop. If $y=1$, then we
have $x^2 \cdot z = x\cdot xz$. If $z=1$, then $x^2 y = xy\cdot x$. Then $x\cdot xy = xy\cdot x$. If we denote
$xy$ by $y$, then we obtain that $xy=yx$, i.e. the loop $(Q, \cdot)$ is commutative.

It is clear that a commutative Moufang loop is left semimedial.

Case 4. For the proof of Case 3 it is possible to use \lq\lq mirror\rq\rq \,  principles.  We give the  direct
proof. Let $(Q, \cdot)$ be a  right semimedial loop, i.e. $zy \cdot x^2  = zx \cdot yx$ for all $x, y, z \in Q$.
If $y=1$, the we have $z\cdot x^2  = zx\cdot x$.

If $z=1$, then $y x^2  = x \cdot yx$. Then $zx\cdot x = x\cdot zx$. If we denote $zx$ by $z$, then we obtain
that $zx=xz$, i.e. the loop $(Q, \cdot)$ is commutative. Moreover, we have  $x^2 \cdot yz =  zy \cdot x^2$, $
xy\cdot zx = xz \cdot yx$.

It is clear that a commutative Moufang loop is right semimedial.

Case 5. From $x \cdot  yz = f(x)y \cdot xz $ we have $x \cdot  yz = y \cdot xz $. From the last identity by
$z=1$ we obtain $x \cdot  y = y \cdot x$. Therefore we can re-write identity $x \cdot  yz = y \cdot xz $ in the
following form $yz \cdot x  = y \cdot zx $.

Case 6 is proved in the similar way to Case 5.
\end{proof}

Commutative Moufang loop in which any element has the order 3 is called 3-CML.

\begin{remark} \label{CENTER_OF_CML}
Center $C(Q,+)$ of a CML $(Q,+)$ is a normal abelian subgroup of  $(Q,+)$ and it coincides with the left nucleus
of $(Q,+)$ \cite{VD, RHB}.
\end{remark}

\begin{lemma} \label{MOUF_CENTRAL_AUTOM}
In a commutative Moufang loop the map $\delta : x\mapsto 3x$ is  central endomorphism \cite{RHB, VD}.
\end{lemma}
\begin{proof}
In a CML $(Q, +)$ we have  $n(x+y) = nx+ ny$ for any natural number $n$ since by Moufang theorem \cite{MUFANG,
RHB, VD} CML is diassociative (any two elements generate an associative subgroup). Therefore the map $\delta$ is
an endomorphism. See \cite{KIN_VOITEH_07} for many details on commutative diassociative loops.

The proof of centrality of the endomorphism $\delta$ is standard \cite{HOP, VD, RHB, kepka_06} and we omit it.
\end{proof}

A quasigroup  $(Q,\cdot)$ with identities  $x y = y x$, $x\cdot x y = y$, $x\cdot yz = x y \cdot x z$ is called
a distributive Steiner quasigroup \cite{VD, 1a}.

\begin{theorem} \label{MAIN_Moufang_Commut}
1.  Any  commutative Moufang loop $(Q, +)$ has the following structure
$$
(Q, +) \cong (A, \oplus) \times (B, +),
$$
where $(A, \oplus)$ is an abelian group and there exists a number $m$ such that  $|s^m(A, \oplus)| = 1$; $(B,
+)$ is an isotope of a  distributive quasigroup $(B, \star)$, $x + y = s(x \star y)$ for all $x, y \in B$, $s
\in Aut(B, +)$, $s \in Aut(B, \star)$.

2.  $C(Q, +) \cong (A, \oplus) \times C(B, +)$.

3. $(Q, +)\slash C(Q, +)\cong (B, +)\slash C(B, +) \cong (D, +)$ is 3-CML in which the endomorphism $s$ is
permutation $I$ such that $I x= -x$.

4. Quasigroup $(D, \star)$,  $x \star y = -x - y$, $ x, y \in (D, +)$,  is a distributive Steiner quasigroup.
\end{theorem}
\begin{proof}
Case 1. The existence of decomposition of $(Q, +)$ into two factors follows from Theorem \ref{MAIN_MIDDLE_F}.

From Corollary \ref{SLOI_POLUMED} it follows that  any equivalence class   $\bar a \equiv H_j$ of the normal
congruence $Ker\, s_j$ containing an idempotent element $0\in Q$ is  an unipotent loop  $(H_j, \cdot) $
 isotopic to an abelian  group with isotopy of the form $(\varphi, -\varphi, \varepsilon)$, where   $\varphi
\in  Aut(H_j, \cdot)$   for all suitable values of $j$.

 Since $(H_j,\cdot)$ is a commutative loop, we have that
 $x\cdot 0 = \varphi\, x  -\varphi\, 0 = \varphi \, x = x$, $\varphi = \varepsilon$,
$0\cdot x = \varphi \, 0 -\varphi \, x = -\varphi \, x = x$, $-\varphi = \varepsilon$. Thus $x\cdot y = x + y $
for all $x, y \in (H_j, \cdot)$.

We notice, in a commutative Moufang loop $(Q,+)$ the map $s^{\,i}$ takes the form $2^{\,i}$, i.e. $s^i(x) =
2^{\,i} (x)$. Then in the loop $(A, \oplus)$ any non-zero element has the order $2^{\,i}$ or  infinite order.

If an element $x$ of the loop $(A, \oplus)$ has a finite order, then $x\in C (A, \oplus)$, where $C (A, \oplus)$
is a center of $(A, \oplus)$ since $G.C.D. (2^{\,i}, 3) = 1$.

If an element $x$ of the loop $(A, \oplus)$ has infinite order, then by Lemma \ref{MOUF_CENTRAL_AUTOM} $3x\in C
(A, \oplus)$, $\left< x\right> \cong \left< 3x\right>$.

Therefore $(A, \oplus) \cong 3(A, \oplus)\subseteq C(A, \oplus)$, $(A, \oplus)$ is an abelian group.

From Cases 1, 2 of Theorem \ref{MAIN_MIDDLE_F} it follows that $(B, +)$ is an isotope of left and right
distributive quasigroup. Therefore, $(B, +)$ is an isotope of distributive quasigroup.

Case 2. From Lemma \ref{center_OF_DIR_PROD} it follows that $C(Q, +) \cong C(A, \oplus) \times C(B, +)$.
Therefore $C(Q, +) \cong (A, \oplus) \times C(B, +)$ since $C(A, \oplus) = (A, \oplus)$.

Case 3. The fact that $(Q, +)\slash C (Q, +)$ is 3-CML is  well known and it follows from Lemma
\ref{MOUF_CENTRAL_AUTOM}. Isomorphism $(Q, +)\slash C (Q, +) \cong ((A, \oplus) \times (B, +))\slash ((A,
\oplus) \times C(B, +))$ follows from Cases 1, 2.

Isomorphism $$((A, \oplus) \times (B, +))\slash ((A, \oplus) \times C(B, +)) \cong (B, +)\slash C(B, +) $$
follows from the Second Isomorphism Theorem (\cite{BURRIS}, p. 51,   for group case see \cite{KM}) and the fact
that $(A, \oplus) \cap C(B, +) = \{ (0,0) \}$, i.e. $|(A, \oplus) \cap C(B, +)| = 1$.

Case 4. It is clear that in 3-CML $(D, +)$ the map $s$ takes the form $s(x) =2x =-x  = Ix$. Moreover, $I^{-1} =
I$.  It is easy to see that the quasigroup $(D, \star)$ is a distributive Steiner quasigroup.
\end{proof}

\begin{corollary} \label{MAIN_Moufang_Comm}
If in CML $(Q,+)$ the endomorphism $s$ has finite order $m$, then:  (i)  any non-zero element of the group $(A,
\oplus)$ has the order  $2^{\,i}$, $1\leqslant i \leqslant m$; (ii) $Aut(Q,+) \cong Aut (A, \oplus) \times Aut
(B, +)$.
\end{corollary}
\begin{proof}
Case (i). It is easy to see.

Case (ii).  Let  $(Q, +)$ be   a commutative Moufang loop, $\alpha \in Aut(Q,+)$.   Then the order of an element
$x$ coincides with the order of  element $\alpha (x)$. Indeed, if $nx=0$, then  $n(\alpha x) = \alpha (n x) =0$.

The loops $(A, \oplus)$ and $(B, +)$  have elements of different orders. Indeed, the orders of elements of the
loop $(A, \oplus)$ are powers of the number $2$ and  orders of the elements of the loop $(B, +)$ are some odd
numbers or, possibly,  $\infty$.

Therefore loops $(A, \oplus)$ and $(B, +)$ are invariant relative to any automorphism of the loop $(Q,+)$. Then
$Aut(Q,+) \cong Aut (A, \oplus) \times Aut (B, +)$.
\end{proof}

\section{The structure}

 Theorems \ref{MAIN_LEFT_F} and  \ref{MAIN_MIDDLE_F} give us a possibility to receive some
information on left and right F-, SM-, E-quasigroups and some combinations of these classes.

\subsection{Simple left and right F-, E- and SM-quasigroups}

 Simple quasigroups of some classes of finite left distributive  quasigroups are described  in \cite{GALKIN_79}.
 The structure and properties of F-quasigroups are described in \cite{kepka_05, kepka_06}.

 We give Je\u zek-Kepka Theorem \cite{JEZEK} in the following  form   \cite{SC05_a, SHCH_STR_05}.

\begin{theorem} \label{SIMPLE_MEDIAL_QUAS} If a  medial quasigroup  $(Q,\cdot)$ of  the form $x\cdot y = \alpha x +
\beta y + a $ over an abelian group $(Q, +)$ is simple, then
 \begin{enumerate}
\item
the group $(Q,+)$ is the additive group of a finite Galois field $GF(p^k)$;
\item
the group $\left< \alpha, \beta \right> $ is the multiplicative group of  the field $GF(p^k)$ in the case $k>1$,
the group $\left< \alpha, \beta \right> $ is any subgroup of the group $Aut(Z_p,+)$  in the case $k=1$;
\item
the quasigroup $(Q,\cdot)$ in the case $|Q|>1$ can be quasigroup from one of the following disjoint quasigroup
classes:
 \begin{enumerate}
\item $\alpha + \beta  =\varepsilon, a=0$; in this case the quasigroup $(Q,\cdot)$ is an idempotent quasigroup;
\item
 $\alpha +\beta  =\varepsilon$  and $a\neq 0$; in this case the quasigroup $(Q, \cdot)$ does not have any
idempotent element,  the quasigroup $(Q,\cdot)$ is isomorphic to the quasigroup $(Q,\ast)$ with the form $ x\ast
y = \alpha x  + \beta y +1$ over the same abelian group $(Q,+)$;
 \item
$\alpha + \beta \neq  \varepsilon $;     in this case the quasigroup $(Q,\cdot )$ has exactly one idempotent
element, the quasigroup $(Q,\cdot )$     is isomorphic to the quasigroup $(Q,\circ)$     of the form $x\circ y =
\alpha x + \beta y$  over the group $(Q,+)$.
\end{enumerate}
\end{enumerate}
\end{theorem}

\begin{theorem} \label{SCERB_THEO}
If a simple distributive quasigroup $(Q,\circ)$ is isotopic to  finitely generated commutative Moufang loop
$(Q,+)$, then $(Q,\circ)$ is a finite  medial distributive quasigroup \cite{vs2, SCERB_91, AVTOREF_SC}.
\end{theorem}
\begin{proof}
It is known \cite{RHB} that any finitely generated CML $(Q, +)$ has a non-identity center $C(Q, +)$ (for short
$C$).

We check that center of CML $(Q, +)$ is invariant  (is a characteristic subloop) relative to any automorphism of
the loop $(Q,+)$ and the quasigroup $(Q, \circ)$.

Indeed, if $\varphi \in Aut (Q,+)$, $a\in C(Q,+)$ (see Remark \ref{CENTER_OF_CML}), then we have $\varphi
(a+(x+y))= \varphi ((a+x)+y) = \varphi a+ (\varphi x+ \varphi y) = (\varphi a+ \varphi x)+ \varphi y$. Thus
$\varphi a \in C(Q,+)$, $\varphi C(Q,+) \subseteq C(Q,+)$.

For any distributive quasigroup $(Q, \circ)$ of the form $ x\circ y = \varphi x + \psi y$ we have $$Aut(Q,\circ)
\cong M(Q,+)  \leftthreetimes (\mathbb{C} \slash I),$$ where  $I$ is the group of inner permutations of
commutative Moufang loop $(Q,+)$, $$\mathbb{C} = \left\{ \omega \in Aut(Q, +) \, | \, \omega \varphi = \varphi
\omega \right\}.$$ Therefore any automorphism of $(Q, \circ)$ has the form $L^+_a \alpha$, where $\alpha \in
Aut(Q,+)$ \cite{16}.

The center $C$ defines normal congruence $\theta$  of the loop $(Q, +)$ in the following way $x\theta y
\Longleftrightarrow x+C= y +C$. We give a little part  of this standard proof:  $(x+a) + C= (y + a) + C
\Longleftrightarrow (x+ C) + a = (y + C) + a \Longleftrightarrow x+ C  = y + C$.  In fact $C$ is coset class of
$\theta$ containing zero element of $(Q,+)$.

The congruence  $\theta$ is admissible relative to any permutation of the form $L^+_a \alpha$, where $\alpha\in
Aut(Q,+)$, since $\theta$ is central congruence. Therefore, $\theta$ is congruence in the quasigroup
$(Q,\circ)$.

Since $(Q,\circ)$ is simple quasigroup and $\theta$ cannot be diagonal congruence, then $\theta = Q\times Q$,
$C(Q,+) = (Q,+)$, $(Q,\circ)$ is medial. From Theorem \ref{SIMPLE_MEDIAL_QUAS} it follows that $(Q,\circ)$ is
finite.
\end{proof}

We notice, it is possible to prove Theorem \ref{SCERB_THEO}  using Theorem \ref{DLYA_CLM} \cite{vs2}.

\begin{lemma} \label{SCERB_THEO_COR}
If simple quasigroup $(Q,\cdot)$ is isotope either of the form $(f, \varepsilon, $ $\varepsilon)$,  or of the
form $(\varepsilon, e,  \varepsilon)$, or  of the form $(\varepsilon, \varepsilon, s)$ of a distributive
quasigroup $(Q,\circ)$, where $f, e, s \in Aut(Q,\circ) $ and $(Q,\circ)$ is isotopic to finitely generated
commutative Moufang loop $(Q,+)$, then $(Q,\cdot)$ is finite medial quasigroup.
\end{lemma}
\begin{proof}
Since  $(Q,+)$ is finitely generated,  then $|\, C(Q,+)\, | > 1$ \cite{RHB}. From the proof of Theorem
\ref{SCERB_THEO} it follows that  $C(Q,+)$ is invariant relative to  any automorphism of $(Q,\circ)$.

Therefore necessary condition of simplicity of $(Q,\cdot)$ is the fact that $C(Q,+) = (Q,+)$. Then $(Q,\circ)$
is medial.

 Prove that $(Q, \cdot)$  is medial, if   $x\cdot y =  f x\circ  y$. We have $xy \cdot uv = f(fx\circ y)\circ
(f u \circ  v) = (f^2 x\circ f u)\circ (f y \circ v) = (xu)\cdot (yv)$ \cite{4}.

 Prove that $(Q, \cdot)$  is medial, if  $x\cdot y = x\circ e y$.  We have $xy \cdot uv = (x\circ e y)\circ (e
u \circ e^{2} v) = (x\circ e u)\circ (e y \circ e^{2} v) = (xu)\cdot (yv)$ \cite{4}.

 Prove that $(Q, \cdot)$  is medial, if  $x\cdot y = s (x\circ y)$.  We have $xy \cdot uv = (s^2 x\circ s^2
y)\circ (s^2 u \circ s^2 v) = (s^2 x\circ s^2 u)\circ (s^2 y \circ s^2 v) = (xu)\cdot (yv)$ \cite{4}.
\end{proof}

We can obtain some information on simple left and right F-, E- and SM-quasigroups.

\begin{theorem}\begin{enumerate} \label{F_SM_E_SIMPLE}
    \item Left F-quasigroup $(Q, \cdot)$ is simple if and only if
it  lies in one from the following quasigroup classes:

(i) $(Q, \cdot)$ is a right loop of the form  $x\cdot y = x +  \psi y$, where  $\psi \in Aut (Q,+)$ and the
group $(Q,+)$ is $\psi$-simple;

(ii) $(Q, \cdot)$ has the the form $x \cdot y = x\circ \psi y $, where $\psi \in Aut (Q, \circ)$ and  $(Q,
\circ)$ is $\psi $-simple left distributive quasigroup.

    \item Right  F-quasigroup $(Q, \cdot)$ is simple if and only if it
lies in one from the following quasigroup classes:

  (i) $(Q, \cdot)$ is a left  loop of the form  $x\cdot y = \varphi x +   y$, where  $\varphi \in Aut
(Q,+)$ and the group $(Q,+)$ is  $\varphi$-simple;

(ii) $(Q, \cdot)$ has the the form $x \cdot y = \varphi x\circ y $, where $\varphi \in Aut (Q, \circ)$ and  $(Q,
\circ)$ is $\varphi $-simple left distributive quasigroup.

    \item
Left SM-quasigroup $(Q, \cdot)$ is simple if and only if it  lies in one from the following quasigroup classes:

(i)  $(Q, \cdot)$  is  a  unipotent quasigroup of the form $x\circ y = -\varphi x + \varphi y$,  $(Q, +)$ is a
group, $\varphi \in Aut (Q,+)$ and the group $(Q,+)$ is $\varphi$-simple;

(ii) $(Q, \cdot)$ has the the form $x \cdot y = \varphi (x\circ y) $, where $\varphi \in Aut (Q, \circ)$ and
$(Q, \circ)$ is $\varphi $-simple left distributive quasigroup.

    \item  Right  SM-quasigroup $(Q, \cdot)$ is simple if and only if
it lies in one from the following quasigroup classes:

(i)  $(Q, \cdot)$  is  a   unipotent quasigroup of the form   $x\circ y = \varphi x - \varphi y$,  $(Q, +)$ is a
group, $\varphi \in Aut (Q,+)$ and the group $(Q,+)$ is $\varphi$-simple;

(ii) $(Q, \cdot)$ has the the form $x \cdot y = \varphi (x\circ y) $, where $\varphi \in Aut (Q, \circ)$ and
$(Q, \circ)$ is $\varphi $-simple right  distributive quasigroup.

    \item Left E-quasigroup $(Q, \cdot)$ is simple if and only if
it lies in one from the following quasigroup classes:

(i)  $(Q, \cdot)$  is a left  loop of the form  $x\cdot y = \alpha x +   y$, $\alpha 0 =0$,  and  $(Q,+)$ is
  $\alpha$-simple abelian group;

(ii)  $(Q, \cdot)$ has the the form $x \cdot y = \varphi x\circ y $, where $\varphi  \in Aut (Q, \circ)$ and
$(Q, \circ)$ is $\varphi$-simple left distributive quasigroup.

    \item
Right  E-quasigroup $(Q, \cdot)$ is simple if and only if it lies in one from the following quasigroup classes:

(i)  $(Q, \cdot)$  is a right   loop of the form  $x\cdot y =  x +  \beta y$, $\beta 0 =0$,  and  $(Q,+)$ is
  $\beta$-simple abelian group;

  (ii) $(Q, \cdot)$ has the the form $x \cdot y =  x\circ \psi y $,
where $\psi \in Aut (Q, \circ)$ and  $(Q, \circ)$ is $\psi$-simple right  distributive quasigroup.
\end{enumerate}
\end{theorem}
\begin{proof}
Case 1. Suppose that $(Q,\cdot)$ is simple left F-quasigroup. From Theorem \ref{MAIN_LEFT_F}  it follows that
$(Q,\cdot)$  can be a  quasigroup with a unique idempotent element or an isotope of  a left distributive
quasigroup.

By Theorem \ref{NORM_QUAS_CONGR} the endomorphism  $e$ defines the corresponding normal congruence $Ker\,  e$.
Since $(Q, \cdot)$ is simple, then this congruence is  the diagonal $\hat{Q} = \{(q,q)\, | q\in Q\}$ or the
universal congruence $Q\times Q$.

From   Theorem \ref{MAIN_LEFT_F}  it follows that in simple left F-quasigroup the map $e$ is  zero endomorphism
or a permutation.

Structure of  left F-quasigroups in the case when $e$ is  zero endomorphism follows from Lemma
\ref{LEMMA_3_F_QUAS}.

Structure of  left F-quasigroups in the case when $e$ is an automorphism follows from Lemma \ref{LEMMA_4_ED}.
Additional properties of quasigroup $(Q,\circ)$ follow from Lemma \ref{NL1}.

Converse. Using Corollary  \ref{DLYA_PROSTYH_F_QUAS} we can say that that left F-quasigroups from these
quasigroup classes are simple.

Cases 2-6 are proved in the similar way.
\end{proof}

\begin{remark}
Left F-quasigroup  $(Z,\cdot)$, where $x\cdot y = -x + y$, $(Z, +)$ is the infinite cyclic group, (Example
\ref{INFINITE_LEFT_F_QUAS})  is not simple. Indeed, in this quasigroup the endomorphism $e$ is not a permutation
(a bijection) of the set $Z$ or a zero endomorphism.

We  can also apply Case 3 of Theorem \ref{F_SM_E_SIMPLE}, since $(Z,\cdot)$ is a left SM-quasigroup. And so on.
\end{remark}

\subsection{F-quasigroups}

Simple F-quasigroups isotopic to groups (FG-quasigroups) are described in \cite{kepka_06}. The authors prove
that any simple FG-quasigroup is a simple group or a simple medial quasigroups. We notice that simple medial
quasigroups are described in \cite{JEZEK}. See also \cite{SC05_a, SHCH_STR_05}. Conditions when a group isotope
is a left (right) F-quasigroup there are in \cite{Kir_95, SOH_99}.

The following examples demonstrate that in an F-, E-, SM-quasigroup the order of map $e$ does not coincide with
the order of map $f$, i.e. there exists some independence of the orders of  maps $e$,  $f$ and $s$.

\begin{example}
By $(Z_3, +)$ we denote the cyclic group of order 3 and  we take $Z_3 = \{ 0, \, 1, \, 2  \}$. Groupoid  $(Z_3,
\cdot)$, where $x\cdot y = x-y$, is a medial E-, F-, SM-quasigroup and $e^{\cdot}(Z_3) = s^{\cdot}(Z_3) = \{ 0
\}$, $f^{\cdot}(Z_3) = Z_3$.
\end{example}

\begin{example}
By $(Z_6, +)$ we denote the cyclic group of order 6 and  we take $Z_6 = \{ 0, \, 1, \, 2,\, 3,\,$ $ 4,\, 5  \}$.
Groupoid   $(Z_6, \cdot)$, where $x\cdot y = x-y$, is a medial  E-, F-, SM-quasigroup and $e^{\cdot}(Z_6) =
s^{\cdot}(Z_6) = \{ 0 \}$, $f^{\cdot}(Z_6) = \{ 0, \, 2, \, 4\}$.
\end{example}

The following lemmas give connections between the maps $e$ and $f$  in F-quasigroups.

\begin{lemma} \label{NEW_LEFT_AND RIGHT_ZERO_ENDOM}
1. Endomorphism  $e$ of an  F-quasigroup $(Q, \cdot)$ is zero endomorphism, i.e. $e(x) = 0$ for all $x\in Q$ if
and only if  $x\cdot y = x + \psi y$, $(Q,+)$ is a group, $\psi \in Aut(Q,+)$, $(Q, \cdot)$ contains unique
idempotent element $0$, $x + f y = f y + x$ for all $x, y \in Q$.

2. Endomorphism  $f$ of an  F-quasigroup $(Q, \cdot)$ is zero endomorphism, i.e. $f(x) = 0$ for all $x\in Q$ if
and only if  $x\cdot y = \varphi x + y$, $(Q,+)$ is a group, $\varphi \in Aut(Q,+)$, $(Q, \cdot)$ contains
unique idempotent element $0$, $x + e y = e y + x$ for all $x, y \in Q$.
\end{lemma}
\begin{proof}
1. From Lemma \ref{LEMMA_3_F_QUAS} Case 1 it follows that $(Q, \cdot)$ is a right loop, isotope of a group
$(Q,+)$  of the form $x\cdot y = x+ \psi y$, where $\psi \in Aut(Q,+)$.

If $a\cdot a=a$, then $a + \psi a = a$, $\psi a=0$, $a=0$.

If we  rewrite right F-quasigroup equality in terms of the operation +, then we obtain $x + \psi y + \psi z = x
+ \psi f(z) + \psi y + \psi^2 z$, $\psi y + \psi z = \psi f(z) + \psi y + \psi^2 z$. If we take $y=0$ in the
last equality, then $\psi z = \psi f(z) + \psi^2 z$. Therefore $\psi y + \psi f(z) + \psi^2 z = \psi f(z) + \psi
y + \psi^2 z$, $\psi y + \psi f(z)
 = \psi f(z) + \psi y $, $y + f(z)  = f(z) + y $.

 Converse. From $x\cdot y = x + \psi y$ we have $x\cdot e (x) = x + \psi e (x) = x$, $e (x) = 0$ for all $x\in Q$.

2. This Case is proved in the similar way with Case 1.
\end{proof}

\begin{lemma} \label{LEFT_AND RIGHT_ZERO_ENDOM}
1. If endomorphism  $e$ of an  F-quasigroup $(Q, \cdot)$ is zero endomorphism, i.e. $e(x) = 0$ for all $x\in Q$,
then

(i) $f(x) = x -  \psi x$,  $f\in End (Q, +)$;

(ii) $f(Q, +) \subseteq C(Q, +)$;

(iii)   $(H, +) \unlhd (Q, +)$,  $f(Q, +) \unlhd  (Q, +)$, $(Q,+)\slash (H, +) \cong f(Q, +)$, where $(H, +)$ is
equivalence  class of the congruence $Ker\, f$ containing identity element of $(Q, +)$;

(iv)  $f(Q, \cdot)$ is a medial F-quasigroup; $(H, \cdot) = (H,+)$ is a group; $(\bar a, \cdot)$, where $\bar a
$ is equivalence class  of the normal congruence $Ker\, f_j$ containing an idempotent element $a\in Q$, $i
\geqslant 1$, is an abelian group.

2. If endomorphism  $f$ of an  F-quasigroup $(Q, \cdot)$ is zero endomorphism, i.e. $f(x) = 0$ for all $x\in Q$,
then

(i) $e(x) =  -  \varphi x + x$,  $e\in End (Q, +)$;

(ii) $e(Q, +) \subseteq C(Q, +)$.

(iii)   $(H, +) \unlhd (Q, +)$,  $e(Q, +) \unlhd  (Q, +)$, $(Q,+)\slash (H, +) \cong e(Q, +)$, where $(H, +)$ is
equivalence class of the congruence $Ker\, e$ containing identity element of $(Q, +)$;

(iv)   $e(Q, \cdot)$ is a medial F-quasigroup; $(H, \cdot) = (H,+)$ is a group; $(\bar a, \cdot)$, where $\bar a
$ is equivalence class  of the normal congruence $Ker\, e_j$ containing an idempotent element $a\in Q$, $i
\geqslant 1$, is an abelian group.
\end{lemma}
\begin{proof}
1. (i) From Lemma \ref{NEW_LEFT_AND RIGHT_ZERO_ENDOM}, Case 1  we have $f(x)\cdot x = f(x)+ \psi x = x$, $f(x) =
x - \psi x$. We can rewrite equality  $f(x \cdot y) = f(x) \cdot f(y)$ in the form  $f(x + \psi y) = f(x) + \psi
f(y)$. If $x = y = 0$, then  we have $f(0) = 0$. If $x=0$, then $f\psi (y) = \psi f(y)$. Therefore
\begin{equation} \label{ENDOM_F}
f(x + \psi y) = f(x) + f\psi (y). \end{equation}

(ii) If we apply to equality (\ref{ENDOM_F}) the  equality  $f(z) = z -  \psi z$, then we obtain  $x + \psi y -
\psi (x + \psi y)  = x -\psi x + \psi y - \psi^2 y$, $x + \psi y - \psi^2 y - \psi x   = x - \psi x + \psi y -
\psi^2 y$, $\psi y - \psi^2 y - \psi x = - \psi x + \psi y - \psi^2 y$, $y - \psi y -  x   = - x + y - \psi y$,
$f y - x = - x + f y$, $  x + f y = f y + x$,  i.e. $f(Q, +) \subseteq C(Q, +)$.

(iii) From definitions and Case (ii) it follows that $(H, +) \unlhd (Q, +)$,  $f(Q, +) \unlhd  (Q, +)$. The last
follows from definition of $(H, +)$.

(iv) $f(Q, \cdot)$ is a medial F-quasigroup since from Case (ii) it follows that  $f(Q, +)$ is an abelian group.
Quasigroup $(H, \cdot)$ is a group since in this quasigroup the maps $e$ and $f$ are zero endomorphisms and we
can use Case (i).

$(\bar a, \cdot) \cong f^i(Q, \cdot)\slash f^{i+1}(Q, \cdot)$ is an abelian group since in this quasigroup the
maps $e$ and $f$ are zero endomorphisms and $f^i(Q, \cdot)$ is a medial quasigroup for any suitable value of the
index $i$. Moreover, it is well known that a medial quasigroup any its subquasigroup is normal \cite{tkpn}. Then
 $f^{i+1}(Q, \cdot) \unlhd f^i(Q, \cdot)$.

2. This Case is proved in the similar way with Case 1.
\end{proof}

\begin{corollary} \label{COROL_LEFT_AND RIGHT_ZERO_ENDOM}
Both endomorphisms  $e$ and $f$ of an  F-quasigroup $(Q, \cdot)$ are zero endomorphisms if and only if
$(Q,\cdot)$ is a group.
\end{corollary}
\begin{proof}
By Lemma \ref{NEW_LEFT_AND RIGHT_ZERO_ENDOM}, Case 1 $x\cdot y = x +\psi y$. By Lemma \ref{LEFT_AND
RIGHT_ZERO_ENDOM}, Case 1, (i), $f(x) = x - \psi x$. Since $f(x) = 0$ for all $x\in Q$, further we have $\psi =
\varepsilon$.

Converse. It is clear that in any group $e(x) = f(x) = 0$ for all $x\in Q$.
\end{proof}

\begin{example}
By $(Z_4, +)$ we denote the cyclic group of order 4 and  we take $Z_4 = \{ 0, \, 1, \, 2,\, 3  \}$. Groupoid
$(Z_4, \cdot)$, where $x\cdot y = x+ 3y$, is a medial  E-, F-, SM-quasigroup,  $e^{\cdot}(Z_4) = s^{\cdot}(Z_4)
=  \{ 0 \}$ and $f^{\cdot}(Z_4) = \{ 0, \, 2 \} = H$.
\end{example}

\begin{corollary} \label{LEFT_ZERO_ENDOM_AND_RIGHT_ZERO}
1. If in F-quasigroup $(Q, \cdot)$ endomorphism e is zero endomorphism and the group $(Q, +)$ has  identity
center, then $(Q, \cdot)= (Q,+)$.

2. If in F-quasigroup $(Q, \cdot)$ endomorphism f is zero endomorphism and the group $(Q, +)$ has  identity
center, then $(Q, \cdot)= (Q,+)$.
\end{corollary}
\begin{proof}
The proof follows from Cases (iii) and  (ii) of Lemma \ref{LEFT_AND RIGHT_ZERO_ENDOM}.
\end{proof}

\begin{corollary} \label{LEFT_ZERO_ENDOM_AND_RIGHT_PERM}
1. If endomorphism  $e$ of an  F-quasigroup $(Q, \cdot)$ is zero endomorphism, i.e. $e(x) = 0$ for all $x\in Q$,
 endomorphism $f$ is a permutation of the set $Q$, then
 $x\cdot y = x+\psi y$, $(Q,+)$ is an abelian  group, $\psi \in Aut(Q,+)$ and $(Q, \circ)$,
 $x\circ y = f x + \psi y$, is a medial distributive quasigroup.

2. If endomorphism  $f$ of an  F-quasigroup $(Q, \cdot)$ is zero endomorphism, i.e. $f(x) = 0$ for all $x\in Q$,
 endomorphism $e$ is a permutation of the set $Q$, then
 $x\cdot y = \varphi x + y$, $(Q,+)$ is an abelian  group, $\varphi \in Aut(Q,+)$
 and $(Q, \circ)$, $x\circ y = \varphi x + e y$, is a medial distributive quasigroup.
\end{corollary}
\begin{proof}
The proof follows from Lemma \ref{LEFT_AND RIGHT_ZERO_ENDOM}. It is a quasigroup folklore  that idempotent
medial quasigroup is distributive \cite{88, SCERB}.
\end{proof}

\begin{remark}
It is easy to see that  condition \lq\lq  $(D, \cdot)$ is a medial F-qua\-si\-gro\-up of the form $x\cdot y =
x+\psi y$ such that  $(D, \circ)$, $x\circ y = f x + \psi y$, is a medial distributive quasigroup\rq\rq \, in
Corollary \ref{LEFT_ZERO_ENDOM_AND_RIGHT_PERM} is equivalent to the condition that the automorphism $\psi$ of
the group $(D,+)$ is complete (Definition \ref{Automorphism_full}).
\end{remark}

\begin{lemma} \label{LEFT_AND RIGHT_PERMUT_ENDOM}
1. If endomorphism  $e$ of an  F-quasigroup $(Q, \cdot)$ is a permutation of the set $Q$, i.e. $e$ is an
automorphism of $(Q, \cdot)$, then $(Q, \circ)$, $x\circ y = x\cdot e(y)$,   is a left distributive quasigroup
which satisfies the equality   $ (x \circ y) \circ z = (x \circ f z) \circ (y \circ e^{-1} z)$,  for all $x, y,
z \in Q$.

2. If endomorphism  $f$ of an  F-quasigroup $(Q, \cdot)$ is a permutation of the set $Q$, i.e. $f$ is an
automorphism of $(Q, \cdot)$, then $(Q, \circ)$,  $x\circ y = f(x) \cdot y$,   is a right distributive
quasigroup which satisfies the equality   $ x \circ (y \circ z) = (f^{-1} x \circ y) \circ (e x \circ  z)$, for
all $x, y, z \in Q$.
\end{lemma}
\begin{proof}
1. The fact that $(Q, \circ)$, $x\circ y = x\cdot e(y)$, is a left distributive quasigroup, follows from Lemma
\ref{LEMMA_4_ED}.  If we rewrite right F-quasigroup equality in terms of the operation $\circ$, then $(x \circ
e^{-1} y)\circ e^{-1} z = (x \circ e^{-1}f z)\circ (e^{-1} y \circ e^{-2} z) $. If we replace $e^{-1} y$ by $y$,
$e^{-1} z$ by $z$ and take into consideration that $e^{-1}f = fe^{-1}$, then we obtain the equality $ (x \circ
y) \circ z = (x \circ f z) \circ (y \circ e^{-1} z)$.

2. The proof is similar to Case 1.
\end{proof}

\begin{corollary} \label{LEFT_AND RIGHT_PERMUT_ENDOM_COR}
1. If endomorphism  $e$ of an  F-quasigroup $(Q, \cdot)$ is identity  permutation of the set $Q$, then $(Q,
\cdot)$ is a distributive quasigroup.

2. If endomorphism  $f$ of an  F-quasigroup $(Q, \cdot)$ is identity permutation of the set $Q$, then $(Q,
\cdot)$ is a distributive quasigroup.
\end{corollary}
\begin{proof}
1. If $f x \cdot x = x$, then $f x \circ e^{-1} x = x$. Further  proof follows from Lemma \ref{LEFT_AND
RIGHT_PERMUT_ENDOM}. Indeed from $f x \circ e^{-1} x = x$ it follows $f x \circ x = x$, $f x = x$, since
$(Q,\circ)$ is idempotent quasigroup. Then $f = \varepsilon$.

2. The proof is similar to Case 1.
\end{proof}

The following proof belongs to the OTTER 3.3  \cite{MAC_CUNE}. The author of this program is Professor
W.~McCu\-ne. We also have used much J.D. Phillips' article \cite{PHILL_03}. Here we give the adopted (humanized)
form of this proof.

\begin{theorem} \label{MACQUNE_1}
If in a left distributive quasigroup $(Q, \circ)$ the equality
\begin{equation} \label{eQ_3-1}
(x \circ y) \circ z = (x \circ f z) \circ  (y \circ e z)
\end{equation}
is fulfilled for all $x, y, z \in Q$, where  $f, e$ are the  maps  of $Q$, then  the following equality is
fulfilled in $(Q, \circ)$:  $(x \circ y) \circ f z = (x \circ f z) \circ  (y \circ f z)$.
\end{theorem}
\begin{proof}
If we pass in  equality (\ref{eQ_3-1}) to operation $\slash$, then we obtain
\begin{equation} \label{eQ_3-11}
((x \circ y) \circ z)  / (y \circ e (z)) =  x \circ f z.
\end{equation}

 From  equality (\ref{eQ_3-1}) by $x = y$ we obtain
$x \circ z = (x \circ f z) \circ  (x \circ e z)$ and using left distributivity we have $x \circ z = x \circ (f z
\circ  e (z))$,
\begin{equation} \label{eQ_30}
 z = f z     \circ e (z), \qquad  e(z)  = f z \,  \backslash \,  z.
\end{equation}

If we change  in equality (\ref{eQ_3-11}) the expression $e(z)$ using equality (\ref{eQ_30}), then we obtain
\begin{equation} \label{eQ_3_41}
((x \circ y) \circ z)  / (y \circ (f z \,  \backslash \,  z)) =  x \circ f z.
\end{equation}

We make the following replacements in (\ref{eQ_3_41}): $x \rightarrow x \slash z$, $y \rightarrow z$, $z
\rightarrow y$. Then we obtain $(x\circ y) \circ z \rightarrow ((x / z)\circ z) \circ y = x\circ y$ and the
following equality is fulfilled
\begin{equation} \label{eQ_3_5}
(x\circ y) \slash (z\circ  (f(y)\backslash  y))= (x\slash z)\circ f(y).
\end{equation}

Using the operation $\slash $  we can rewrite left distributive identity in the following form
\begin{equation} \label{eQ_31}
 (x\circ (y\circ z))/ (x\circ z)=x\circ y.
\end{equation}

If we change in identity (\ref{eQ_31}) $(y\circ z)$ by $y$, then variable $y$  passes in $y/z$. Indeed, if
$y\circ z = t$, then $y = t/z$. Therefore, we have
\begin{equation} \label{eQ_32}
(x\circ y)/ (x\circ z)=x\circ  (y/z).
\end{equation}

From  equality (\ref{eQ_3-1}) using left distributivity to the right side of this equality  we obtain $ (x \circ
y) \circ z = ((x \circ f z) \circ y) \circ ((x \circ f z) \circ e z)$. After applying of the operation $/$ to
the last equality we obtain
\begin{equation} \label{eQ_33}
((x\circ y)\circ z) / ((x\circ f(z))\circ e(z))= (x\circ f(z))\circ y.
\end{equation}

After substitution  of (\ref{eQ_30}) in (\ref{eQ_33}) we obtain
\begin{equation} \label{eQ_34}
((x\circ y)\circ z) / ((x\circ f(z))\circ (f z \,  \backslash \,  z))  = (x\circ f(z))\circ y.
\end{equation}

Now we show  the most unexpected  OTTER's step. We apply the left side of  equality (\ref{eQ_3_5}) to the left
side equality (\ref{eQ_34}). In this case expression $((x\circ y)\circ z)$  from (\ref{eQ_34}) plays the role of
$(x\circ y)$, $(x\circ f(z))$ the role of $z$ and $(f z \,  \backslash \,  z)$ plays the role of
$(f(y)\backslash  y)$.

Therefore we obtain
\begin{equation} \label{eQ_4147}
((x\circ y) \slash (x \circ f(z)))\circ f(z) = (x\circ f z)\circ y.
\end{equation}

After application to the left side of equality (\ref{eQ_4147})  equality (\ref{eQ_32}) we have
\begin{equation} \label{KEY_EQ}
((x\circ  (y \slash f z))\circ f(z) = (x\circ f z)\circ y.
\end{equation}

If we change in equality  (\ref{KEY_EQ}) $(y\slash f z)$ by $y$, then variable $y$  passes in $y \circ f z$.
Therefore $(x\circ  y)\circ f z = (x\circ  f z)\circ (y \circ f z)$.
\end{proof}

\begin{corollary} \label{MACQUNE}
If in a left distributive quasigroup $(Q, \circ)$ the equality
\begin{equation*}
(x \circ y) \circ z = (x \circ f z) \circ  (y \circ e z)
\end{equation*}
is fulfilled for all $x, y, z \in Q$, where $e$ is a map,   $f$ is a  permutation of the set $Q$,  then $(Q,
\circ)$ is a distributive quasigroup.
\end{corollary}
\begin{proof}
The proof follows from Theorem \ref{MACQUNE_1}.
\end{proof}

\begin{theorem} \label{LEFT_DISTRIB_Isotopic_to_F_QUAS}
If in F-quasigroup $(Q,\cdot)$ endomorphisms  $e$ and $f$ are  permutations of the set $Q$,  then $(Q,\cdot)$ is
isotope of the form $x\cdot y = x\circ e^{-1} y$  of a  distributive quasigroup $(Q,\circ)$.
\end{theorem}
\begin{proof}
Quasigroup $(Q, \circ)$  of the form $x\circ y = x \cdot e(y)$ is a left distributive quasigroup (Lemma
\ref{LEMMA_4_ED}) in which the equality $ (x \circ y) \circ z = (x \circ f z) \circ (y \circ e^{-1} z)$, is true
(Lemma \ref{LEFT_AND RIGHT_PERMUT_ENDOM}). By Corollary \ref{MACQUNE} $(Q, \circ)$ is distributive.
\end{proof}

\begin{theorem} \label{SIMPLE_F_QUAS}
An   F-quasigroup  $(Q, \cdot )$ is simple if and only if $(Q,\cdot)$ lies in  one from the following quasigroup
classes:

(i) $(Q, \cdot )$ is a simple group in the case when  the maps $e$ and $f$ are zero endomorphisms;

(ii) $(Q, \cdot )$ has the form $x\cdot y = x +\psi y$, where $(Q,+)$ is a $\psi$-simple abelian group, $\psi
\in Aut(Q,+)$,  in the case when the map $e$ is a zero endomorphism and the map $f$ is a permutation; in this
case $e =- \psi, f x + \psi x = x$ for all $x\in Q$;

(iii) $(Q, \cdot )$ has the form $x\cdot y = \varphi x + y$, where $(Q,+)$ is a $\varphi$-simple abelian group,
$\varphi \in Aut(Q,+)$, in the case when the map $f$ is a zero endomorphism and the map $e$ is a permutation; in
this case $f =-\varphi, \varphi  x + e x = x$ for all $x\in Q$;

(iv) $(Q, \cdot )$ has the form $x\cdot y = x \circ\psi y$, where $(Q,\circ)$ is a $\psi$-simple distributive
quasigroup $\psi \in Aut(Q, \circ)$,   in the case when the maps $e$ and $f$ are permutations; in this case $e
=\psi^{-1}$, $fx \circ \psi x = x$ for all $x\in Q$.
\end{theorem}
\begin{proof} $( \Longrightarrow )$
(i) It is clear that in this case left and right F-quasigroup equalities are transformed in the identity of
associativity.

(ii) From Lemma \ref{LEFT_AND RIGHT_ZERO_ENDOM} (iii) and the fact that the map $f$ is a permutation of the set
$Q$ it follows that $(Q, +)$ is an abelian group.

(iii) This case is similar to Case (ii).

(iv) By  Belousov result \cite{1a} (see Lemma \ref{LEMMA_4_ED} of this paper)  if the endomorphism  $e$ of a
left F-quasigroup $(Q, \cdot)$ is a permutation of the set $Q$, then quasigroup $(Q, \cdot)$ has  the form
$x\cdot y = x \circ \psi y$, where $(Q, \circ)$  is a left distributive quasigroup and $\psi \in Aut (Q,
\circ)$, $\psi \in Aut (Q, \cdot)$.  The right distributivity of $(Q, \circ)$ follows  from Theorem
\ref{LEFT_DISTRIB_Isotopic_to_F_QUAS}.

$( \Longleftarrow )$ Using Corollary  \ref{DLYA_PROSTYH_F_QUAS} we can say that that F-quasigroups from these
quasigroup classes are simple.
\end{proof}

\begin{remark}
There exists a possibility to formulate Case  (iv) of Theorem \ref{SIMPLE_F_QUAS} in the following form.

$(iv)^{\ast}$  $(Q, \cdot )$ has the form $x\cdot y = \varphi x \circ y$, where $(Q,\circ)$ is a
$\varphi$-simple distributive quasigroup, in the case when the maps $e$ and $f$ are permutations; in this case
$f =\varphi^{-1}, \varphi x \circ e x = x$ for all $x\in Q$.
\end{remark}

\begin{corollary} \label{FINITE_SIMPLE_F_QUAS}
Finite simple F-quasigroup  $(Q, \cdot)$ is a simple group or a simple medial quasigroup.
\end{corollary}
\begin{proof}
Case (i) of Theorem \ref{SIMPLE_F_QUAS} demonstrates  us that simple F-quasigroup can be a simple group.

Taking into consideration Toyoda Theorem (Theorem \ref{Toyoda_Theorem}) we see that  Cases (ii) and (iii) of
Theorem \ref{SIMPLE_F_QUAS} provide us that simple F-quasigroups can be simple medial quasigroups.

We shall prove that in Case (iv) of Theorem \ref{SIMPLE_F_QUAS}  we also obtain medial quasigroups.

The quasigroup $(Q,\cdot)$ is isotopic to distributive quasigroup $(Q,\circ)$, quasigroup  $(Q,\circ)$ is
isotopic to CML $(Q,+)$. Therefore $(Q,\cdot)$ is isotopic to the $(Q,+)$ and we can apply Lemma
\ref{SCERB_THEO_COR}.
\end{proof}

Taking into consideration Lemma \ref{medial_F_SM_E} we can say that some properties of finite simple medial
F-quasigroups are described in Theorem \ref{SIMPLE_MEDIAL_QUAS}.

Using obtained in this section results we can add information on the structure of  F-quasigroups
\cite{kepka_06}.
\begin{theorem} \label{STRUCTURE_FINITE_F_QUAS}
Any  finite F-quasigroup $(Q,\cdot)$  has the following structure
$$
(Q, \cdot) \cong (A, \circ) \times (B, \cdot),
$$
where $(A, \circ)$ is a quasigroup with a unique idempotent element; $(B, \cdot)$ is isotope of a left
distributive quasigroup $(B, \star)$, $x \cdot y = x \star \psi y$, $\psi \in Aut(B, \cdot)$, $\psi \in Aut(B,
\star)$. In the quasigroups $(A, \circ)$ and $(B, \cdot)$ there exist the following chains
\begin{equation*} \begin{split}
& A \supset e(A) \supset  \dots \supset e^{m-1}(A) \supset e^m(A) = 0, \\
&  B \supset  f(B) \supset \dots  \supset f^r(B) = f^{r+1}(B),
\end{split}
\end{equation*}
where:
\begin{enumerate}
  \item  Let $D_i$ be an equivalence class  of the normal congruence $Ker\, e_i$ containing an idempotent
element $a\in A$, $i \geqslant 0$. Then:
\begin{enumerate}
\item  $(D_i, \circ)$ is  linear right  loop of the form $x\circ y = x+\psi y$, where $\psi \in Aut(D_i,+)$;
 \item $Ker\left(  f|_{(D_i,\circ)}\right) $  is a group;
\item if  $j\geqslant 1$, then
$Ker\left(  f_j|_{(D_i,\circ)}\right) $  is an abelian  group;
\item  if $f$ is a permutation of $f^l(D_i,\circ)$, then
 $f^l(D_i,\circ)$ is a medial right loop  of the form $x\circ y =  x+\psi y$, where $\psi$ is a
 complete automorphism of the group $f^l(D_i,+)$;
\item  $(D_i,\circ) \cong (E_i, +) \times f^l(D_i,\circ)$,
  where $(E_i, +)$ is a linear right  loop, an
 extension of an abelian group by   abelian groups and by a group.
\end{enumerate}
 \item   Let $H_j$ be an equivalence class  of the normal congruence $Ker\, f_j$ containing an idempotent
element $b\in B$, $j \geqslant 0$. Then:
\begin{enumerate}
 \item $(H_0, \cdot)$ is a linear left loop  of the form  $x\cdot y = \varphi x + y$;
\item $f(B, \cdot)$ is isotope of a distributive quasigroup $f(B, \star)$ of the form $x\cdot y = x\star e^{-1}
y$;
 \item  if $0 <j < r$, then $(H_j, \cdot)$
 is medial left loop of  the form  $x\cdot y = \varphi x + y$, where
$(H_j,+)$ is an abelian  group, $\varphi \in Aut(H_j,+)$ and $(H_j, \star)$, $x\star y = \varphi x + e y$, is a
medial distributive quasigroup;
  \item $(B,\cdot) \cong (G, +) \times f^r(B,\cdot)$, where $(G, +)$  has a unique idempotent element, is
  an  extension of an abelian group by   abelian groups and by a linear left  loop $(H_0, \cdot)$, $f^r(B,\cdot)$
  is a distributive quasigroup.
\end{enumerate}
\end{enumerate}
\end{theorem}
\begin{proof}
From Theorem \ref{MAIN_LEFT_F}, Case 1 it follows that F-quasigroup $(Q,\cdot)$ is isomorphic to the direct
product of quasigroups $(A, \circ)$ and $(B, \cdot)$.

In F-quasigroup $(A, \circ)$   the  chain
\begin{equation*}
 A \supset e(A) \supset e^2(A)
\supset \dots \supset e^{m-1}(A) \supset e^m(A) =  e^{m+1} (A) = 0
\end{equation*}
becomes stable on a number $m$, where $0$ is idempotent element.

Case 1, (a).  If  $0 \leqslant j < m$, then by Lemma \ref{NEW_LEFT_AND RIGHT_ZERO_ENDOM} any quasigroup $(D_j,
\circ)$ is a right loop, isotope of a group $(D_j,+)$ of the form $(D_j, \circ) = (D_j,+)(\varepsilon, \psi,
\varepsilon)$, where $\psi \in Aut (D_j,+)$.

Case 1, (b).  \lq\lq Behaviour\rq\rq \, of the map $f$  in the right loop $(D_j, \circ)$  is described by Lemma
\ref{LEFT_AND RIGHT_ZERO_ENDOM}. If $f$ is zero endomorphism, then $(D_j, \circ)$ is a   group in  case $j=0$ (
Lemma \ref{LEFT_AND RIGHT_ZERO_ENDOM}, Case (i)) and  it is an abelian group in the case $j>0$ ( Lemma
\ref{LEFT_AND RIGHT_ZERO_ENDOM}, Case (ii)).

If $f$ is a non-zero endomorphism of $(D_j, \circ)$, then information on the structure of $(D_j, \circ)$ follows
from Lemma \ref{LEFT_AND RIGHT_ZERO_ENDOM}  and Corollary \ref{LEFT_ZERO_ENDOM_AND_RIGHT_PERM}.

Case 1, (c). The proof  follows from Lemma \ref{LEFT_AND RIGHT_ZERO_ENDOM}, (ii), (iv)  and the fact that in the
quasigroup $Ker\left(  f_j|_{(D_j,\circ)}\right) $ the maps $e$ and $f$ are zero endomorphisms.

Case 1, (d). The proof  follows from Corollary \ref{LEFT_ZERO_ENDOM_AND_RIGHT_PERM}, Case 1.

Case 1, (e). The proof  follows from results of the previous Cases of this theorem and Theorem
\ref{MAIN_LEFT_F}, Case 2.

Using Lemma \ref{LEFT_AND RIGHT_PERMUT_ENDOM} we can state that that F-quasigroup $(B,\cdot)$ is isotopic to
left distributive quasigroup $(B, \star)$, where $x\star y = x\cdot e(y)$.

In order to have more detailed information on the structure of the quasigroup  $e^m(Q, \cdot)$ we study the
following chain
\begin{equation*}
B \supset  f(B) \supset \dots  \supset f^r(B) = f^{r+1}(B), \end{equation*} which becomes stable on a  number
$r$.

Case 2, (a). The proof  follows from Corollary \ref{LEFT_ZERO_ENDOM_AND_RIGHT_PERM}, Case 2.

Case 2, (b). The proof  follows from   Theorem \ref{MACQUNE_1}.

Case 2, (c). Since    $f$ is zero endomorphism of quasigroup $(H_j, \cdot)$, $e_j|_{H_j}$ is a permutation of
the set $H_j$,   then by Corollary \ref{LEFT_ZERO_ENDOM_AND_RIGHT_PERM}  quasigroup $(H_j, \cdot)$ has the form
$x\cdot y = \varphi x + y$, where $(H_j,+)$ is an abelian  group, $\varphi \in Aut(H_j,+)$ and $(H_j, \circ)$,
$x\circ y = \varphi x + e y$, is a medial distributive quasigroup.

Case 2, (d). The existence of direct decomposition follows from  Theorem \ref{MAIN_LEFT_F}, Case 2.
\end{proof}

We notice that information on the structure of finite medial quasigroups there is in \cite{SHCH_STR_05}.

\subsection{E-quasigroups}

We recall, a quasigroup $(Q,\cdot)$ is trimedial  if and only if $(Q,\cdot)$ is an E-quasigroup
\cite{Kin_PHIL_04}. Any trimedial quasigroup is isotopic to CML \cite{kepka76}. Structure of trimedial
quasigroups have been studied in \cite{BENETEU_KEPKA_85, kepka_BEN_LAC_86, SHCHUKIN_86, kepka_90}. Here slightly
other point of view on the structure of trimedial quasigroups is presented.

\begin{lemma} \label{LEFT_AND RIGHT_ZERO_ENDOM_E_QAUS}
1. If endomorphism  $f$ of an  E-quasigroup $(Q, \cdot)$ is zero endomorphism, i.e. $f(x) = 0$ for all $x\in Q$,
then  $x\cdot y = \varphi x+ y$, $(Q,+)$ is a abelian  group, $\varphi \in Aut(Q,+)$.

2. If endomorphism  $e$ of an  E-quasigroup $(Q, \cdot)$ is zero endomorphism, i.e. $e(x) = 0$ for all $x\in Q$,
then  $x\cdot y =  x+ \psi y$, $(Q,+)$ is a abelian  group, $\psi \in Aut(Q,+)$.
\end{lemma}
\begin{proof}
1. From Theorem \ref{THEOREM_3_F_QUAS_MIDD} Case 3 it follows that $(Q, \cdot)$ is a left  loop,
 $x\cdot y = \alpha x + y$,  $(Q, +)$ is an abelian group, $\alpha \in S_Q$,  $\alpha \, 0 = 0$.

 Further we have $x\cdot e(x) = \alpha x + e(x) = x$, $\alpha x = x - e(x) = (\varepsilon - e) x$.
Therefore  $\alpha$ is an endomorphism of $(Q,+)$, moreover, it is an automorphism of $(Q,+)$, since $\alpha$ is
a permutation of the set $Q$.

2. The proof of Case 2 is similar to the proof of Case 1.
\end{proof}
\begin{corollary} \label{LEFT_AND RIGHT_BOTH_ZERO_ENDOM_E_QAUS}
If endomorphisms  $f$ and $e$ of an  E-quasigroup $(Q, \cdot)$ are zero endomorphisms, i.e. $f(x) = e(x) = 0$
for all $x\in Q$, then  $x\cdot y = x + y$, $(Q,+)$ is a abelian  group.
\end{corollary}
\begin{proof}
From equality $\alpha x + e(x) = x$ of Lemma \ref{LEFT_AND RIGHT_ZERO_ENDOM_E_QAUS} we have $\alpha x  = x$,
$\alpha = \varepsilon $.
\end{proof}

\begin{corollary} \label{LEFT_ZERO_ENDOM_AND_RIGHT_PERM_E}
1. If endomorphism  $f$ of an  E-quasigroup $(Q, \cdot)$ is zero endomorphism and
 endomorphism $e$ is a permutation of the set $Q$, then
 $x\cdot y = \varphi x + y$, $(Q,+)$ is an abelian  group, $\varphi \in Aut(Q,+)$
 and $(Q, \circ)$, $x\circ y = \varphi x + e y$, is a medial distributive quasigroup.

2. If endomorphism  $e$ of an  E-quasigroup $(Q, \cdot)$ is zero endomorphism and
 endomorphism $f$ is a permutation of the set $Q$, then
 $x\cdot y = x+\psi y$, $(Q,+)$ is an abelian  group, $\psi \in Aut(Q,+)$ and $(Q, \circ)$,
 $x\circ y = f x + \psi y$, is a medial distributive quasigroup.
\end{corollary}
\begin{proof}
Case 1. From Lemma \ref{LEFT_AND RIGHT_ZERO_ENDOM_E_QAUS} it follows that in this case  $(Q, \cdot)$ has the
form $x\cdot y = \varphi x+ y$ over abelian group $(Q,+)$. Then  $x\cdot e(x) = \varphi x + e(x) = x$, $e(x) = x
- \varphi x$,  $e(0) = 0$. We can rewrite equality $e(x \cdot y) = e(x) \cdot e(y)$ in the form $e(\varphi x +
y) = \varphi e (x)  + e(y)$.  By $y = 0$ we have $e\varphi (x) = \varphi e(x)$. Then $e(\varphi x + y) = e
\varphi x + e y$, the map $e$ is an endomorphism of $(Q, +)$. Moreover, the map $e$ is an automorphism of $(Q,
+)$.

From  Toyoda Theorem and equality $e(x) = x - \varphi x$  it follows that quasigroup $(Q, \circ)$ is medial
idempotent. It is well known that a  medial idempotent quasigroup is distributive.

Case 2 is proved in the similar way to Case 1.
\end{proof}

\begin{theorem} \label{LEFT_DISTRIB_Isotopic_to_E_QUAS}
If the endomorphisms  $f$ and $e$  of an  E-quasigroup $(Q, \cdot)$ are permutations of the set $Q$, then
quasigroup $(Q, \circ)$ of the form $x\circ y = f(x) \cdot y$ is a  distributive quasigroup and $f, e \in Aut
(Q, \circ)$.
\end{theorem}
\begin{proof}
The proof of this theorem is similar to the proof of Theorem \ref{LEFT_DISTRIB_Isotopic_to_F_QUAS}.

By Lemma  \ref{LEMMA_24_ED} $(Q, \cdot)$ is isotope of the form $x\cdot y = f^{-1} x \circ y$ of  a left
distributive quasigroup $(Q, \circ)$ and $f\in Aut (Q, \circ)$.

Moreover, by Lemma  \ref{LEMMA_24_ED} $(Q, \cdot)$ is isotope of the form $x\cdot y =  x \diamond e^{-1} y$ of a
right distributive quasigroup and $e\in Aut (Q, \diamond)$. Therefore $f^{-1} x \circ y = x \diamond e^{-1} y$,
$ x \circ y = f x \diamond e^{-1} y$.

 Automorphisms $e, f$ of the quasigroup $(Q, \cdot)$ lie in $Aut
(Q,\circ)$ (Lemma \ref{ON AUTOMORPHISM_OF_IDEM_QUAS} or \cite{MARS}, Corollary 12). We recall, $ef=fe$ (Lemma
\ref{COMMUTING_ENDOMORPH}).

 Now we  need to rewrite right distributive identity in terms of operation $\circ$. We have
 \begin{equation*}
 \begin{split}
&f(f x \circ e^{-1} y) \circ e^{-1} z = f(f x \circ e^{-1} z) \circ e^{-1} (f y \circ e^{-1} z),\\
&(f^{2} x \circ f e^{-1} y) \circ e^{-1} z = (f^{2} x \circ f e^{-1} z) \circ  (e^{-1} f y \circ e^{-2} z).
\end{split}
\end{equation*}

 If in the last equality we change  element $f^{2} x$ by  element $x$, element $f e^{-1} y = e^{-1} f  y$ by
 element $y$, element $e^{-1}z$ by element $z$, then we obtain
\begin{equation*}
(x \circ y) \circ z = (x \circ f z) \circ  (y \circ e^{-1} z).
\end{equation*}
In order to finish this  proof we shall apply Corollary  \ref{MACQUNE}.
\end{proof}

\begin{corollary} \label{SIMPLE_E_QUAS}
An E-quasigroup  $(Q, \cdot )$ is  simple if and only if this quasigroup lies in one from the following
quasigroup classes:

(i) $(Q, \cdot )$ is a simple abelian group in the case when  the maps $e$ and $f$ are zero endomorphisms;

(ii) $(Q, \cdot )$ is a simple medial quasigroup of the form $x\cdot y = \varphi x + y$ in the case when the map
$f$ is a zero endomorphism and the map $e$ is a permutation;

(iii) $(Q, \cdot )$ is a simple medial quasigroup of the form $x\cdot y = x +\psi y$ in the case when the map
$e$ is a zero endomorphism and the map $f$ is a permutation;

(iv) $(Q, \cdot )$ has the form $x\cdot y = x \circ\psi y$, where $(Q,\circ)$ is a $\psi$-simple distributive
quasigroup, $\psi \in Aut(Q, \circ)$,   in the case when the maps $e$ and $f$ are permutations.
\end{corollary}
\begin{proof}
($\Longrightarrow$) (i) The proof follows from Corollary  \ref{LEFT_AND RIGHT_BOTH_ZERO_ENDOM_E_QAUS}. (ii)  The
proof follows from Lemma \ref{LEFT_AND RIGHT_ZERO_ENDOM_E_QAUS} Case 1. (iii) The proof follows from Lemma
\ref{LEFT_AND RIGHT_ZERO_ENDOM_E_QAUS} Case 2. (iv) The proof is similar to the proof of Case (iv) of Theorem
\ref{SIMPLE_F_QUAS}.

($\Longleftarrow$) It is clear that any quasigroup from these quasigroup classes is simple E-quasigroup.
\end{proof}

\begin{corollary} \label{FINITE_SIMPLE_E_QUAS}
Finite simple E-quasigroup  $(Q, \cdot)$ is a simple medial quasigroup.
\end{corollary}
\begin{proof}
The proof follows from Corollary  \ref{SIMPLE_E_QUAS} and is similar to the proof of Corollary
\ref{FINITE_SIMPLE_F_QUAS}. We can use Lemma \ref{SCERB_THEO_COR}.
\end{proof}

Taking into consideration Corollary \ref{FINITE_SIMPLE_E_QUAS} we can say that properties of finite simple
E-qua\-si\-gro\-ups are described by Theorem \ref{SIMPLE_MEDIAL_QUAS}.

\begin{lemma} \label{E_LEFT_ZERO_AND RIGHT_PERMUT_ENDOM_FULL_THEO}
1. If endomorphism  $f$ of an   E-quasigroup $(Q, \cdot)$  is zero endomorphism, then $(Q, \cdot) \cong
(A,\circ) \times  (B, \cdot)$, where $(A,\circ)$ a medial E-quasigroup of the form $x\cdot y = \varphi x + y$
and there exists a number $m$ such that $|e^m(A, \circ)| =1$, $(B, \cdot)$ is a medial E-quasigroup of the form
$x\cdot y = \varphi x + y$ such that  $(B, \star)$, $x\star y = \varphi x + e y$, is a medial distributive
quasigroup.

2. If endomorphism  $e$ of an E-quasigroup $(Q, \cdot)$ is  zero endomorphism, then $(Q, \cdot) \cong (A, +)
\times  (B, \cdot)$, where $(A, \circ)$ is a medial E-quasigroup of the form $x\cdot y = x + \psi y$ and there
exists a number $m$ such that $|f^m(A, \circ)| =1$, $(B, \cdot)$ is a medial E-quasigroup of the form $x\cdot y
= x+\psi y$ such that  $(B, \star)$, $x\star y = f x + \psi y$, is a medial distributive quasigroup.
\end{lemma}
\begin{proof}
 1. By Theorem  \ref{MAIN_MIDDLE_F}  Case 4 any  right E-quasigroup  $(Q, \cdot)$ has the following structure
$(Q, \cdot) \cong (A, \circ) \times (B, \cdot), $ where $(A, \circ)$ is a quasigroup with a unique idempotent
element and there exists a number $m$ such that $|e^m(A, \circ)| =1$; $(B, \cdot)$ is an isotope of a right
distributive quasigroup $(B, \star)$, $x \cdot y = x \star  e^{-1}(y)$ for all $x, y \in B$, $e \in Aut(B,
\cdot)$, $e \in Aut(B, \star)$.

From Lemma \ref{LEFT_AND RIGHT_ZERO_ENDOM_E_QAUS} it follows that  $(Q, \cdot)$ has the form $x\cdot y = \varphi
x+ y$ over an abelian group $(Q,+)$.

We recall that $e = \varepsilon - \varphi$,  $e \varphi  = \varphi e $ (Corollary
\ref{LEFT_ZERO_ENDOM_AND_RIGHT_PERM_E}). From equalities $x\cdot y = \varphi x + y$ and $x \cdot y = x \star
e^{-1}(y)$ we have $x\star y = \varphi x + e y$. Then $(B, \star)$ is medial, idempotent, therefore it is
distributive.

2. The proof is similar to Case 1.
\end{proof}

\begin{remark}
If $m=1$, then $(A,\circ)$ is an abelian group (Corollary \ref{LEFT_AND RIGHT_BOTH_ZERO_ENDOM_E_QAUS}).

If $m=2$, then  $(A,\circ)$ is an extension of an abelian group by an abelian group. If, in addition, the
conditions of Lemma \ref{LEMMA_LOOP_GROUP} are fulfilled, then   $(A,\circ)$ is an abelian group.

If the number $m$ is finite and the conditions of Lemma \ref{LEMMA_LOOP_GROUP} are fulfilled, then after
application of Lemma \ref{LEMMA_LOOP_GROUP} $(m-1)$ times we obtain that $(A,\circ)$ is an abelian group.
\end{remark}

Now we have a possibility to give in more details information on the structure of finite E-quasigroups. The
proof of the following theorem in many details is similar to the proof of Theorem \ref{STRUCTURE_FINITE_F_QUAS}.

Let $D_i$ be an equivalence class  of the normal congruence $Ker\, e_i$ containing an idempotent element $a\in
A$, $i \geqslant 0$. Let $H_j$ be an equivalence class  of the normal congruence $Ker\, f_j$ containing an
idempotent element, $j \geqslant 0$.

\begin{theorem}
In any finite E-quasigroup $(Q,\cdot)$  there exist the following finite  chain
\begin{equation*} \begin{split}
& Q \supset e(Q) \supset  \dots \supset e^{m-1}(Q) \supset e^m(Q) =  e^{m+1} (Q), \\
&  e^m(Q) \supset  fe^m(Q) \supset \dots  \supset f^re^m(Q) = f^{r+1}e^m(Q),
\end{split}
\end{equation*}
where
\begin{enumerate}
  \item If $i < m$, then   $ (D_i, \cdot)
   \cong (H_i, +)
\times  (G_{i}, \cdot)$, where  right loop $(H_i, +)$ is an extension of an abelian group by abelian groups,
$(G_{i}, \cdot)$ is a medial E-quasigroup of the form $x\cdot y = x + \psi y$ such that  $\psi$  is complete
automorphism of the group $(G_i,+)$;
 \item  If $i = m$,    then  $(e^m Q, \cdot)$ is  isotope of  right distributive quasigroup $(e^m Q, \circ)$,
  where  $x\circ y = x\cdot e y $;
\begin{enumerate}
 \item if $j < r$, then  $(H_j, \cdot)$  is medial left loop, $(H_j, \cdot)$ has the form
  $x\cdot y = \varphi x + y$, where $(H_j,+)$ is an abelian  group, $\varphi \in Aut(H_j,+)$ and
  $(H_j, \circ)$, $x\circ y = \varphi x + e y$, is a medial distributive quasigroup;
   \item if $j = r$, then $(f^re^m Q,  \cdot)$ is isotope of the form $x\circ y = f(x) \cdot y$
    of a  distributive quasigroup $(f^re^m Q,  \circ)$.
\end{enumerate}
\end{enumerate}
\end{theorem}
\begin{proof}
It is clear that in  E-quasigroup $(Q, \cdot)$    chain (\ref{chain})
\begin{equation*}
 Q \supset e(Q) \supset e^2(Q)
\supset \dots \supset e^{m-1}(Q) \supset e^m(Q) =  e^{m+1} (Q)
\end{equation*}
becomes stable.

Case 1,  $i < m$. By Lemma \ref{LEFT_AND RIGHT_ZERO_ENDOM_E_QAUS}, Case 2 any quasigroup $(D_i, \cdot)$ is a
medial right loop, isotope of an abelian group $(D_i, +)$ of the form $(D_i, \cdot) = (D_i,+)(\varepsilon, \psi,
\varepsilon)$, where $\psi \in Aut (D_i,+)$, for all suitable values of index $i$, since in the quasigroup
$(D_i, \cdot)$ endomorphism $e$ is zero endomorphism.

If $f$ is zero endomorphism, then in this case $(D_i, \cdot)$ is an abelian  group (Corollary \ref{LEFT_AND
RIGHT_BOTH_ZERO_ENDOM_E_QAUS}).

 If $f$ is a non-zero endomorphism of $(D_i, \cdot)$, then  we can use Lemma
 \ref{E_LEFT_ZERO_AND RIGHT_PERMUT_ENDOM_FULL_THEO} Case 2.

Case 1,  $i = m$. From Lemma \ref{LEMMA_24_ED} Case 4 it follows that E-quasigroup $(e^m Q,\cdot)$ is isotopic
to right distributive quasigroup $(e^m Q, \circ)$, where $x\circ y = x\cdot e(y)$.

In order to have more detailed information on the structure of the quasigroup  $e^m(Q, \cdot)$ we study the
following chain
\begin{equation*}
 e^m(Q) \supset  fe^m(Q) \supset f^{\,2}e^m(Q)  \dots  \supset f^{\,r}e^m(Q) = f^{\, r+1}e^m(Q).  \end{equation*}

Case 2,  $j < r$. From Lemma \ref{LEMMA_24_ED} Case 4 it follows that E-quasigroup $(H_j, \cdot)$ is isotopic to
right distributive quasigroup $(H_j, \circ)$,  $x\circ y = x\cdot e(y)$.

From Lemma \ref{LEFT_AND RIGHT_ZERO_ENDOM_E_QAUS} Case 1 it follows that $(H_j, \cdot) $ has the form   $x\cdot
y = \varphi x+ y$, where $(H_j, +)$ is an abelian  group, $\varphi \in Aut(H_j, +)$.

From equalities  $x\circ e^{-1} y = x\cdot y$ and  $x\cdot y = \varphi x+ y$, we have $x\circ e^{-1} y = \varphi
x + y$, $x\circ y = \varphi x + e y$. Then  right distributive quasigroup $(H_j, \circ)$ is isotopic to abelian
group $(H_j,+)$.

If we rewrite identity $(x\circ y) \circ z = (x\circ z) \circ (y \circ  z)$ in terms of the operation $+$, then
$\varphi^2 x + \varphi e y + e z = \varphi^2 x + \varphi e z +  e \varphi y + e^2 z$, $\varphi e y + e z =
\varphi e z +  e \varphi y + e^2 z$. By $z = 0$ from  the last equality it follows that $\varphi e = e \varphi
$. Then $(H_j, \circ )$  is a medial  quasigroup. Moreover, $(H_j, \circ )$  is a medial distributive
quasigroup, since any medial right distributive quasigroup is distributive.

Case 2,  $j = r$. If $e$ and  $f$ are permutations of the set $f^re^m Q$, then by  Theorem
\ref{LEFT_DISTRIB_Isotopic_to_E_QUAS} $(f^re^m Q, \cdot)$ is isotope of the form $x\circ y = f(x) \cdot y$  of a
distributive quasigroup $(f^re^m Q, \circ)$.
\end{proof}

\subsection{SM-quasigroups}

We recall left and right SM-quasigroup is called an SM-quasigroup. The structure theory of SM-quasigroups mainly
has been developed by T. Kepka and K.~K.~Shchukin \cite{kepka78, kepka78_a, 12, BEGLARYAN_85}.

If an SM-quasigroup $(Q, \cdot)$ is simple, then the endomorphism $s$ is zero endomorphism or a permutation of
the set $Q$.

If $s(x)=0$, then from Theorem \ref{THEOREM_3_F_QUAS_MIDD} it follows
\begin{corollary} \label{SIMPLE_SM_ZERO_KER}
 If the endomorphism  $s$ of a semimedial quasigroup $(Q, \cdot)$ is zero endomorphism, i.e. $s(x)=0$ for all
$x\in Q$, then  $(Q, \cdot)$ is a medial  unipotent quasigroup, $(Q,\cdot) \cong (Q, \circ)$, where  $x\circ y =
\varphi x - \varphi y$,  $(Q, +)$ is an abelian group, $\varphi \in Aut (Q,+)$.
\end{corollary}

\begin{remark}
 By Corollary \ref{SIMPLE_SM_ZERO_KER} equivalence class $D_i$ of the congruence    $ Ker\, s_i$
containing an idempotent element is a medial unipotent quasigroup $(D_i, \cdot)$ of the form   $x\circ y =
\varphi x - \varphi y$, where $(D_i, +)$ is an abelian group, $\varphi \in Aut (D_i,+)$ for all suitable values
of index $i$.
\end{remark}

Information on the structure of  medial unipotent quasigroups there is in \cite{SHCH_STR_05}.

If $s(x)$ is a permutation of the set $Q$, then from Lemma \ref{LEMMA_24_ED} it follows the following

\begin{lemma} \label{SIMPLE_SM_PERMUT_KER}
 If the endomorphism  $s$ of a semimedial quasigroup $(Q, \cdot)$ is a permutation of the set $Q$, then
quasigroup $(Q, \circ)$ of the form $x\circ y = s^{-1}( x \cdot y)$ is a  distributive quasigroup and $s\in Aut
(Q, \circ)$.
\end{lemma}

\begin{corollary} \label{MAIN_SM_QUAS}
 Any   semimedial quasigroup $(Q, \cdot)$ has the following structure
$$
(Q, \cdot) \cong (A, \circ) \times (B, \cdot),
$$
where $(A, \circ)$ is a quasigroup with a unique idempotent element and there exists a  number $m$ such that
$|s^m(A, \circ)| =1$; $(B, \cdot)$ is an isotope of a  distributive quasigroup $(B, \star)$, $x \cdot y = s(x
\star y)$ for all $x, y \in B$, $s \in Aut(B, \cdot)$, $s \in Aut(B, \star)$.
\end{corollary}
\begin{proof}
The proof follows from Cases 3 and 4  of Theorem  \ref{MAIN_MIDDLE_F}.
\end{proof}

\begin{corollary} \label{MAIN_SM_SIMPLE_QUAS}
An  SM-quasigroup $(Q, \cdot)$ is simple if and only if it  lies in one from the following quasigroup classes:

(i)  $(Q, \cdot)$  is  a medial  unipotent quasigroup of the form $x\circ y = \varphi x - \varphi y$,  $(Q, +)$
is abelian group, $\varphi \in Aut (Q,+)$ and the group $(Q,+)$ is $\varphi$-simple;

(ii) $(Q, \cdot)$ has the the form $x \cdot y = \varphi (x\circ y) $, where $\varphi \in Aut (Q, \circ)$ and
$(Q, \circ)$ is $\varphi $-simple distributive quasigroup.
\end{corollary}
\begin{proof}
The proof follows from Cases 3 and 4 of Theorem \ref{F_SM_E_SIMPLE}.
\end{proof}

The similar result on  properties of simple SM-quasigroups there is in \cite{12}, 4.13 Corollary.

\begin{corollary} \label{MAIN_SM_FINITE_SIMPLE_QUAS}
 Any finite simple  semimedial quasigroup $(Q, \cdot)$ is a simple medial quasigroup \cite{12}.
\end{corollary}
\begin{proof}
Conditions of Lemma \ref{SCERB_THEO_COR} are fulfilled and we can apply it.
\end{proof}

\subsection{Simple left FESM-quasigroups}

M.~Kinyon and J.D.~Phillips have defined left FESM-quasigroups in \cite{Kin_PHIL_04}.

\begin{definition} \label{FESM_QUAS}
 A quasigroup $(Q,\cdot)$ which simultaneously is left F-, E- and SM-quasi\-gro\-up we shall name
 \textit{left FESM-quasigroup}.
\end{definition}
From Definition \ref{FESM_QUAS} it follows that in FESM-quasigroup the maps $e, f, s$ are its endomorphisms.

\begin{lemma} \label{LEFT_AND RIGHT_ZERO_ENDOM_FESM_QAUS}
1. If endomorphism  $e$ of a left FESM-quasigroup $(Q, \cdot)$ is zero endomorphism, then  $(Q, \cdot)$ is a
medial right loop,   $x\cdot y =  x+ \psi y$, $(Q,+)$ is a abelian  group, $\psi \in Aut(Q,+)$, $\psi^2 =
\varepsilon$, $\psi s = s, \psi f = f \psi = - f$.

2. If endomorphism  $f$ of an  FESM-quasigroup $(Q, \cdot)$ is zero endomorphism, then $(Q, \cdot)$ is a medial
left loop, i.e.  $x\cdot y = \varphi x+ y$, $(Q,+)$ is a abelian  group, $\varphi \in Aut(Q,+)$, $\varphi^2 =
\varepsilon$, $\varphi s = s, \varphi e = e \varphi = - e$.

3. If  endomorphism  $s$ of a left left FESM-quasigroup  $(Q, \cdot)$ is zero endomorphism, then $(Q, \cdot)$ is
medial unipotent quasigroup of the form   $x\cdot y = \varphi x - \varphi y$,  where $(Q, +)$ is an abelian
group, $\varphi \in Aut (Q,+)$, $\varphi f = f \varphi $, $\varphi e = e \varphi $.
\end{lemma}
\begin{proof}
1. From Lemma \ref{LEMMA_3_F_QUAS} it follows that $(Q, \cdot)$ has the form $x\cdot y = x + \psi y$, where
$(Q,+)$ is a group, $\psi \in Aut (Q,+)$.

Then $s(x) = x\cdot x = x + \psi x$. Since $s$ is an endomorphism of $(Q,\cdot)$, further we have $s(x\cdot y) =
x+y + \psi x + \psi y$,   $s x\cdot s y = s x + \psi s y = x + \psi x + \psi y + \psi^2 y$. Then $x+y + \psi x +
\psi y = x + \psi x + \psi y + \psi^2y$,  $y + \psi x + \psi y =  \psi x + \psi y + \psi^2y$. By $x=0$ we have
$y +  \psi y =   \psi (y + \psi y)$, $s y = \psi s y$. Then  $y + \psi x + \psi y =  \psi x + \psi y + \psi^2y =
\psi x + \psi s y = \psi x +  s y = \psi x +   y + \psi y$. Therefore $y + \psi x + \psi y = \psi x +   y + \psi
y$, $y + \psi x  = \psi x +   y $, the group $(Q, +)$ is commutative. From equality $y + \psi x + \psi y =  \psi
x + \psi y + \psi^2y$ we obtain $y  =  \psi^2y$, $\psi^2 = \varepsilon$.

Further we have  $f(x)\cdot x = f x + \psi x = x$, $f x =  x -  \psi x$, $ \psi f x = \psi x -  x = - f x$, $f
\psi x = \psi x - x$.

\smallskip

2. From Theorem \ref{THEOREM_3_F_QUAS_MIDD} Case 3 it follows that $(Q, \cdot)$ is a left  loop,
 $x\cdot y = \varphi x + y$,  $(Q, +)$ is an abelian group, $\varphi \in S_Q$,  $\varphi 0 = 0$.

 Further we have $x\cdot e(x) = \varphi x + e(x) = x$, $\varphi x = x - e(x) = (\varepsilon - e) x$.
Therefore  $\varphi$ is an endomorphism of $(Q,+)$, moreover, it is an automorphism of $(Q,+)$, since $\varphi$
is a permutation of the set $Q$.

Then $s x = x\cdot x = \varphi x + x$, $s(x\cdot y) = \varphi x + \varphi y + x + y = sx \cdot sy = \varphi s x
+ s y = \varphi^2 x + \varphi x + \varphi y + y$. From equality $\varphi x + \varphi y + x + y = \varphi^2 x +
\varphi x + \varphi y + y$ we obtain $\varphi^2 = \varepsilon$. Then $\varphi sx = \varphi(\varphi x + x) = s
x$.

Further, $x \cdot e x = \varphi  x + e x = x$. Then $e x = x- \varphi  x$, $\varphi e x = \varphi x-  x = -e x$,
$e \varphi x = \varphi x -  x$.

\smallskip

3. From Theorem  \ref{THEOREM_3_F_QUAS_MIDD} Case 1 it follows that $(Q, \cdot)$ is  unipotent quasigroup of the
form   $x\cdot y = -\varphi x + \varphi y$,  where $(Q, +)$ is a group, $\varphi \in Aut (Q,+)$.

Since $f$ is an endomorphism of quasigroup $(Q,\cdot)$ we have $f(x\cdot y) = e (x)\cdot f (y)$, $f(- \varphi x
+ \varphi y) = -\varphi f (x) + \varphi f (y)$. If $y =0$, then $f(- \varphi) = - \varphi f$. If $x=0$, then $f
\varphi = \varphi f$. Then $f$ is an endomorphism of the group $(Q,+)$. Similarly, $e(- \varphi) = - \varphi e$,
$e \varphi = \varphi e$, $e$ is an endomorphism of the group $(Q,+)$.

From $x\cdot e x = x$ we have $-\varphi x + e \varphi x = x$, $ e \varphi x = \varphi x + x$,  $ e x =  x +
\varphi^{-1} x$. Then
\begin{equation} \label{EQ_23}
\begin{split}
& e(x\cdot y) = x\cdot y + \varphi^{-1} (x\cdot y) =  -\varphi x + \varphi y - x + y, \\ & e x \cdot e y =
-\varphi (x + \varphi^{-1} x) + \varphi (y + \varphi^{-1} y) = -\varphi x -  x + \varphi y +  y.
\end{split}
\end{equation}
Comparing the right sides of equalities (\ref{EQ_23}) we obtain that $Q,+)$ is a commutative group.
\end{proof}

\begin{lemma} \label{LEMMA_FESM_PODST}
If endomorphisms  $e, f$ and $s$ of a left FESM-quasigroup $(Q, \cdot)$ are permutations of the set $Q$, then
quasigroup $(Q, \circ)$ of the form $x\circ y = x \cdot e(y)$ is  a left distributive quasigroup and $e, f, s
\in Aut (Q, \circ)$.
\end{lemma}
\begin{proof}
By Lemma \ref{LEMMA_4_ED}  endomorphism  $e$ of a left F-quasigroup $(Q, \cdot)$ is a permutation of the set $Q$
if and only if quasigroup $(Q, \circ)$ of the form $x\circ y = x \cdot e(y)$ is a left distributive quasigroup
and $e\in Aut (Q, \circ)$ \cite{1a}. Then $x \cdot y = x\circ e^{-1}y $, $s(x) = x\circ e^{-1}x $, $f(x)\cdot x
=f x\circ e^{-1}x = x$.

The fact that $e, f, s \in Aut (Q, \circ)$ follows from Lemma \ref{ON AUTOMORPHISM_OF_IDEM_QUAS} Case 2.
\end{proof}

\begin{theorem} \label{SIMPLE_FESM_QUAS}
If $(Q, \cdot )$ is a simple left FESM-quasigroup, then

(i) $(Q, \cdot )$ is simple  medial quasigroup  in the case when at least one from the maps $e$, $f$ and $s$ is
zero endomorphism;

(ii) $(Q, \cdot )$ has the form $x\cdot y = x \circ\psi y$, where $(Q,\circ)$ is a $\psi$-simple left
distributive quasigroup, $\psi \in Aut(Q, \circ)$,   in the case when the maps $e$, $f$ and $s$ are
permutations; in this case $e =\psi^{-1}, fx \circ \psi x = x$, $s(x) = x\circ \psi x$ for all $x\in Q$.
\end{theorem}
\begin{proof}
It is possible to use Lemma \ref{LEFT_AND RIGHT_ZERO_ENDOM_FESM_QAUS} for the proof of Case (i) and Lemma
\ref{LEMMA_FESM_PODST} for the proof of  Case (ii).
\end{proof}

\begin{example}
By $(Z_7, +)$ we denote cyclic group of order 7 and  we take $Z_7 = \{ 0, \, 1, \, 2,\, 3,\,$ $ 4,\, 5, \, 6
\}$.

Quasigroup  $(Z_7, \circ)$, where $x\circ y = x + 6 y = x - y$, is simple  medial  FESM-quasigroup in which the
maps  $e$ and $s$ are zero endomorphisms, the map $f$ is a permutation of the set $Z_7$ ($f(x) = 2 x$ for all $x
\in Z_7$).

Quasigroup  $(Z_7, \cdot)$, where $x\cdot y = 2 x + 3 y$, is simple  medial  FESM-quasigroup in which
endomorphisms $e,\, f, \, s $ are permutations of the set $Z_7$.
\end{example}

\section{Loop isotopes}

In this section we give some  results on the  loops and left loops which  are isotopic to left  F-, SM-, E- and
FESM-quasigroups.

We recall that  any  F-quasigroup is isotopic to a Moufang loop \cite{kepka_05, kepka_07}, any  SM-quasigroup is
isotopic to a commutative Moufang loop \cite{kepka78_a}. Since any E-quasigroup is an SM-quasigroup
\cite{kepka78, Kin_PHIL_04}, then any E-quasigroup also is isotopic to a commutative Moufang loop.

\subsection{Left F-quasigroups}

 Taking into consideration Theorems \ref{MAIN_LEFT_F},    \ref{MAIN_MIDDLE_F},  Lemma
\ref{LEMMA_DIR_2} and Corollary \ref{COROL_DIR_2} we can study loop isotopes of the factors of  direct
decompositions of left and right F- and E-quasigroups.

\begin{theorem} \label{LOOP_ISOTOPES_LEFT_F_QUAS}
1. A left F-quasigroup $(Q, \cdot )$ is isotopic to the direct product of a group $(A, \oplus )$   and a left
S-loop $(B, \diamond)$, i.e.  $(Q, \cdot) \sim (A, \oplus) \times (B, \diamond). $

2. A  right F-quasigroup $(Q, \cdot )$ is isotopic to the direct product of a group $(A, \oplus)$  and right
S-loop $(B, \diamond)$, i.e.  $(Q, \cdot) \sim (A, \oplus) \times (B, \diamond). $
\end{theorem}
\begin{proof}
1. By Theorem  \ref{MAIN_LEFT_F}, Case 1  any  left F-quasigroup $(Q, \cdot)$ has the following structure $(Q,
\cdot) \cong (A, \circ) \times (B, \cdot),$ where $(A, \circ)$ is a quasigroup with a unique idempotent element;
$(B, \cdot)$ is isotope of a left distributive quasigroup $(B, \star)$, $x \cdot y = x \star \psi y$ for all $x,
y \in B$, $\psi \in Aut(B, \cdot)$, $\psi \in Aut(B, \star)$.

By Corollary \ref{COROL_DIR_2},  if a quasigroup $Q$ is the direct product of quasigroups $A$ and $B$, then
there exists an isotopy $T=(T_1, T_2)$ of $Q$ such that  $Q\, T \cong AT_1 \times BT_2$ is a loop.

Therefore we have a possibility to divide our proof on two steps.

\textbf{Step 1.}  Denote a unique idempotent element of $(A, \circ)$ by $0$. We notice that  $e^{\circ} \, 0 =
0$. Indeed, from $(e^{\circ})^{m} A = 0$ we have $(e^{\circ})^{m+1} A = e^{\circ}\, 0 = 0$.

From left F-equality   $x \circ (y \circ z) = (x\circ y) \circ (e^{\circ}(x)\circ z)$ by $x=0$ we have $0 \circ
(y \circ z) = (0\circ y) \circ (0 \circ z)$. Then $L_0 \in Aut(A, \circ)$.

Consider  isotope $(A,\oplus)$ of the quasigroup $(A, \circ)$: $x \oplus y = x \circ L^{-1}_0 y$.  We notice
that
 $(A,\oplus)$ is a left loop.   Indeed, $0 \oplus y = 0 \circ L^{-1}_0 y = y$. Further we have  $x \circ \, y
= x \oplus L_0 \, y $, $x\oplus e^{\oplus} x = x = x\circ L^{-1}_0 e^{\oplus} x = x\circ  e^{\circ} x$,
$e^{\oplus} (x) = L_0 e^{\circ}(x) $, $e^{\oplus} (0) = L_0 e^{\circ}(0) = 0\circ 0 = 0$.

Prove that $L_0 \in Aut (A, \oplus)$. From equality $L_0(x\circ y)= L_0 x \circ L_0 y $ we have $L_0(x\circ y)=
L_0(x\oplus L_0 y)$, $L_0 x \circ L_0 y = L_0 x \oplus L^2_0 y$, $L_0(x\oplus L_0 y)= L_0 x \oplus L^2_0 y$.

If we pass in the left F-equality to the operation $\oplus$, then we obtain $x \oplus (L_0 y \oplus L^2_0 z) =
(x\oplus L_0 y) \oplus (L_0 e^{\circ}(x)\oplus L^2_0 z)$. If we change $L_0 y$ by $y$, $L^2_0 z$ by $z$, then we
obtain
\begin{equation}
x \oplus (y \oplus z) = (x\oplus y) \oplus (L_0 e^{\circ}(x)\oplus z) = (x\oplus y) \oplus (e^{\oplus}(x)\oplus
z). \label{equat_LF}
\end{equation}

Then $(A, \oplus)$ is a left F-quasigroup with the left identity element. For short below in this theorem we
shall use denotation $e$ instead of $e^{\oplus}$.

Further we pass from the operation $\oplus$ to the operation $+$:  $x + y = R^{-1}_0 x \oplus y$, $x \oplus y =
(x \oplus 0) + y $. Then $x + y = (x / 0) \oplus y$, where $x / y = z$ if and only if $z \oplus y = x$. We
notice, $R^{-1}_0 \,  0 =0$, since  $R_0 0 =0$, $0 \oplus 0 =0$, $0=0$.

It is well known, \cite{HOP, VD, 1a}, that $(A,+)$ is a loop. Indeed,   $0 + y = R^{-1}_0 \, 0 \oplus y = 0
\oplus y = y$; $x + 0 = R^{-1}_0 x \oplus 0 = R_0 R^{-1}_0 x = x$.

We express the map $e(x)$ in terms of the operation $+$. We have $x \oplus e(x) = x$. Then $(x \oplus 0) + e(x)
= x$, $e(x) = (x \oplus 0) \backslash \backslash x$, where $x\backslash \backslash y =z$ if and only if $x + z =
y$.

If we denote the map $R^{\oplus}_0$ by $\alpha$, then $x\oplus y = \alpha x + y$,  $e(x) = \alpha x \backslash
\backslash x$. We can rewrite (\ref{equat_LF}) in terms of the loop operation + as follows
\begin{equation}
\alpha x  + (\alpha y  + z) =  \alpha (\alpha x + y) + (\alpha e (x)  + z). \label{equat_LF_1}
\end{equation}

From  $e (x\oplus y) = e x \oplus e y$ we have  $e (\alpha x + y) = \alpha e (x) + e (y)$. By $y=0$ from the
last equality we have
\begin{equation}
e \alpha   = \alpha e. \label{LF_11}
\end{equation}

Therefore $e (\alpha x + y) = e \alpha (x) + e (y)$,  $e$ is a normal endomorphism of $(A,+)$. Changing $\alpha
x$ by $x$ and taking into consideration (\ref{LF_11})  we obtain from equality (\ref{equat_LF_1}) the following
equality
 \begin{equation}
x  + (\alpha y  + z) =  \alpha (x + y) + (e x  + z). \label{equat_LF_2}
\end{equation}

Next part of the proof was obtained using Prover 9 which is developed by Prof. W. McCune \cite{MAC_CUNE_PROV}.

If we put in equality  (\ref{equat_LF_2})  $y = z = 0$, then $\alpha x + e x = x,$ or, equivalently,
 \begin{equation}
\alpha x =  x \slash\slash  e x. \label{IULequat_LF_3}
\end{equation}

If we put in equality  (\ref{equat_LF_2})  $y =  0$, then
 \begin{equation}
\alpha x + (e x + z)  =  x + z. \label{IULequat_LF_4}
\end{equation}

If we apply equality (\ref{IULequat_LF_3}) to equality (\ref{IULequat_LF_4}), then
 \begin{equation}
(x\slash\slash e x) + (e x + z)  =  x + z. \label{IULequat_LF_5}
\end{equation}

If we apply equality (\ref{IULequat_LF_3}) to equality (\ref{equat_LF_2}), then
 \begin{equation}
x  + ((y \slash\slash  e y)  + z) =  ((x + y) \slash\slash e(x + y)) + (e x  + z). \label{IULIequat_LF_22}
\end{equation}

If we change in  equality (\ref{IULequat_LF_5}) $x$ by $x+y$, then
 \begin{equation}
((x+y)\slash\slash e(x+y)) + (e(x + y) + z)  =  (x + y) + z. \label{IULequat_LF_6}
\end{equation}

Taking into consideration Lemma \ref{l5.7} and   Theorem  \ref{MAIN_LEFT_F} we can say that there exists a
minimal number $n$ (finite or infinite) such that $e^n(a) = 0$ for any $a \in A$.

If we change in equality (\ref{equat_LF_2}) $x$ by $e^{n-1} x$, then
 \begin{equation}
e^{n-1} x   + (\alpha y  + z) =  \alpha (e^{n-1} x + y) + z. \label{equat_LF_7}
\end{equation}

If we change in  (\ref{equat_LF_7}) $\alpha x$ by $x \slash\slash  e x$ (equality \ref{IULequat_LF_3}), then
 \begin{equation}
e^{n-1} x   + ((y \slash\slash e y) + z) =   ((e^{n-1} x + y)\slash\slash e y) + z. \label{equat_LF_8}
\end{equation}

We change in  equality (\ref{IULequat_LF_6}) $x$ by $e^{n-1}x$. Then
 \begin{equation}
((e^{n-1}x + y)\slash\slash e y) + (e y + z)  =  (e^{n-1}x + y) + z. \label{IULequat_LF_9}
\end{equation}

We rewrite  the left side of equality  (\ref{IULequat_LF_9}) as follows
\begin{equation}
 ((e^{n-1}x + y)\slash\slash e y) + (e y + z) \overset{\ref{equat_LF_8}}{=}  e^{n-1} x   + ((y
\slash\slash e y) + (e y + z)) \overset{\ref{IULequat_LF_5}}{=} e^{n-1}x + (y + z). \label{IULequat_LF_10}
\end{equation}

From (\ref{IULequat_LF_9}) and (\ref{IULequat_LF_10}) we have
 \begin{equation}
e^{n-1}x + (y  + z)  =  (e^{n-1}x + y) + z. \label{ASSOTSIATIVITY}
\end{equation}

 We change in  equality (\ref{IULequat_LF_6}) $x$ by $e^{n-2}x$, then
 \begin{equation}
((e^{n-2}x+y)\slash\slash e(e^{n-2}x+y)) + (e(e^{n-2}x + y) + z)  =  (e^{n-2}x + y) + z. \label{IULequat_LF_11}
\end{equation}

We rewrite  the left side of equality  (\ref{IULequat_LF_11}) as follows
 \begin{equation}
 \begin{split}
&((e^{n-2}x+y)\slash\slash e(e^{n-2}x+y)) + (e(e^{n-2}x + y) + z)  \overset{\ref{ASSOTSIATIVITY}}{=} \\
&((e^{n-2}x+y)\slash\slash e(e^{n-2}x+y)) + (e^{n-1}x + (ey + z)) \overset{\ref{IULIequat_LF_22}}{=}
\\
& e^{n-2}x+((y\slash\slash e y) + (ey + z))  \overset{\ref{IULequat_LF_5}}{=} e^{n-2}x+(y +
z).\label{IULequat_LF_12}
\end{split}
\end{equation}

From (\ref{IULequat_LF_11}) and (\ref{IULequat_LF_12}) we have
 \begin{equation}
e^{n-2}x + (y  + z)  =  (e^{n-2}x + y) + z. \label{ASSOTSIATIV_1}
\end{equation}

\textbf{Begin Cycle}

 We change in  equality (\ref{IULequat_LF_6}) $x$ by $e^{n-3}x$. Then
 \begin{equation}
((e^{n-3}x+y)\slash\slash e(e^{n-3}x+y)) + (e(e^{n-3}x + y) + z)  =  (e^{n-3}x + y) + z. \label{IULequat_LF_111}
\end{equation}

We rewrite  the left side of equality  (\ref{IULequat_LF_111}) as follows
 \begin{equation}
 \begin{split}
&((e^{n-3}x+y)\slash\slash e(e^{n-3}x+y)) + (e(e^{n-3}x + y) + z)  \overset{\ref{ASSOTSIATIV_1}}{=} \\
&((e^{n-3}x+y)\slash\slash e(e^{n-3}x+y)) + (e^{n-2}x + (ey + z)) \overset{\ref{IULIequat_LF_22}}{=}
\\
& e^{n-3}x+((y\slash\slash e y) + (ey + z))  \overset{\ref{IULequat_LF_5}}{=} e^{n-3}x+(y +
z).\label{IULequat_LF_112}
\end{split}
\end{equation}

From (\ref{IULequat_LF_111}) and (\ref{IULequat_LF_112}) we have
 \begin{equation}
e^{n-3}x + (y  + z)  =  (e^{n-3}x + y) + z.
\end{equation}
\textbf{End Cycle}

Therefore  \begin{equation} e^{n-i}x + (y  + z)  =  (e^{n-i}x + y) + z
\end{equation}
for any natural number $i$. If the number $n$ is finite, then repeating \textbf{Cycle} necessary number of times
 we shall obtain that  $x + (y + z) = (x + y) + z$ for all $x, y, z \in A$.

Since  $n$ is a fixed number (maybe and an infinite), then $ \underset{i \rightarrow \infty}{\textrm{lim}} (n-i)
\longrightarrow 0 $, where $i\in \mathbb N$. We can  apply \textbf{Cycle} necessary  number of times to obtain
associativity. Indeed, suppose that $\lambda$ is a minimal number such that
\begin{equation} e^{\lambda}x + (y  + z)  =  (e^{\lambda}x + y) + z \label{equal_4_21}
\end{equation}
and there exists $a, b, c \in A$ such that
\begin{equation} e^{\lambda - 1}a + (b  + c)  \neq  (e^{\lambda - 1}a + b) + c.
\end{equation}
But from the other side, if we  apply \textbf{Cycle} to  equality (\ref{equal_4_21}), then  we obtain that
\begin{equation} e^{\lambda-1}x + (y  + z)  =  (e^{\lambda-1}x + y) + z
\end{equation}
for all $x, y, z \in A$. I.e.,  $\lambda$ is not a  minimal number  with declared properties.

Therefore our supposition is not true and
\begin{equation} e^{\lambda}x + (y  + z)  =  (e^{\lambda}x + y) + z
\end{equation}
for all  suitable $\lambda$ and all $x, y, z \in A$.

\textbf{Step 2.} From Theorems \ref{MAIN_LEFT_F} and  \ref{AT1} it follows that  $$(B, \diamond) = (B,
\cdot)(\varepsilon, \psi, \varepsilon) ((R_a^{\star})^{ -1},  (L_a^{\star})^{-1}, \varepsilon)=  (B,
\cdot)((R_a^{\star})^{ -1}, \psi (L_a^{\star})^{-1}, \varepsilon)$$ is a left  S-loop.

2. This case is proved similarly to Case 1.
\end{proof}

\begin{lemma} \label{LEFT_F_AUTOTOPY}
The fulfilment of equality (\ref{equat_LF_2}) in the group $(A,+)$  is equivalent to the fact that the triple
$T_x =(\alpha L_x \alpha^{-1}, \varepsilon, L_{\alpha x})(\varepsilon, L_{e(x)}, L_{e (x)})$ is an autotopy of
$(A,+)$ for all $x\in A$.
\end{lemma}
\begin{proof}
From  (\ref{equat_LF_2}) by $y=0$ we have
 \begin{equation} x  +  z =  \alpha x  + (e x  + z),
\label{equat_LF_232}
\end{equation}
i.e.  $L_x = L_{\alpha x} L_{e (x)}$.

If we change in (\ref{equat_LF_2}) $y$ by $\alpha^{-1} y$,  then
 \begin{equation}
x  + (y  + z) =  \alpha (x + \alpha^{-1} y) + (e x  + z). \label{equat_LF_532}
\end{equation}
Equality (\ref{equat_LF_532}) means that the group  $(A,+)$ has an autotopy of the form $$T_x = (\alpha L_x
\alpha^{-1}, L_{e(x)}, L_x)$$ for all $x\in A$. Taking into consideration that $L_x = L_{\alpha x} L_{e (x)}$,
we can rewrite $T_x$ in the form $$T_x = (\alpha L_x \alpha^{-1}, L_{e(x)}, L_{\alpha x} L_{e (x)}) = (\alpha
L_x \alpha^{-1}, \varepsilon, L_{\alpha x})(\varepsilon, L_{e(x)}, L_{e (x)}).$$
\end{proof}

\begin{corollary} \label{LEFT_F_LOOP_GROUP}
If the group  $(A,+)$ has the property   $[L_d , \alpha^{-1}] \in LM(A,+)$ for all $d\in A$, then

(i) $e(A,+)\unlhd C(A,+)\unlhd (A,+)$;

(ii) $\alpha \in Aut(A,+)$;

(iii) $\alpha|_{(Ker\, e, +)} = \varepsilon$.
\end{corollary}
\begin{proof}
(i) It is well known that any autotopy of a group $(A,+)$ has the form
$$(L_a \delta, R_b \delta, L_a R_b \delta),$$ where
$L_a$ is a left translation of the group $(A,+)$, $R_b$ is a right translation of this group, $\delta$ is an
automorphism of this group \cite{1a}.

Therefore if the triple $T_d$ is an autotopy of the loop $(A,+)$, then we have
\begin{equation} \label{AUTOT_GROUP}
\alpha L_d \alpha^{-1} = L_a \delta, L_{e(d)} = R_b \delta, L_d = L_a R_b \delta.
\end{equation}

Then $L_{e(d)}0  = R_b \delta 0$, $e(d) = b$. From $L_d 0 = L_a R_b \delta 0$ we have $d = a + b$, $d = a +
e(d)$. But $d = \alpha d + e(d)$. Therefore, $a = \alpha d$.

We can rewrite equalities (\ref{AUTOT_GROUP}) in the form
\begin{equation} \label{AUTOT_GROUP_1}
\alpha L_d \alpha^{-1} = L_{\alpha d} \delta, L_{e(d)} = R_{e(d)} \delta, L_d = L_{\alpha d} R_{e(d)} \delta.
\end{equation}
Then
\[
\begin{array}{l}
\delta = R_{-e(d)}L_{e(d)} = L_{e(d)} R_{-e(d)}, \\
\alpha L_d \alpha^{-1} =  L_{\alpha d} L_{e(d)} R_{-e(d)} = L_{d}R_{-e(d)}, \\
L_{-d} \alpha L_d \alpha^{-1}  = R_{-e(d)}$, $[L_d , \alpha^{-1}] = R_{-e(d)}
\end{array}
\]
 We notice, all permutations of the form  $\{ R_{-e(d)} \mid d\in A \}$ form a  subgroup $H^{\prime}$ of the
group  $RM(A, +)$,  since $e$ is an  endomorphism of the group $(A,+)$.

By our supposition   $H^{\prime}  \subseteq \, LM(A,+)$. Then $$H^{\prime} \, \subseteq \, RM(A, +) \, \cap \,
LM (A,+).$$ But $LM\left< A, \alpha \right> \cap RM\left< A, \alpha \right> \subseteq C\left< A, \alpha \right>$
\cite{HALL, vs0}. Therefore $R_{-e(d)} = L_{-e(d)}$ for all $d \in A$, $e(A)\subseteq C(A)$.

(ii)  From (i) it follows that  the triple $(\varepsilon, L_{e(b)}, L_{e (b)})$ is an autotopy of $(A,+)$.
Indeed, equality  $ y + (e(b) + z) = e(b) + (y + z)$ is true for all $b, y, z\in A$ since $e(b) \in C(A)$.

Then the triple $(\alpha L_b \alpha^{-1}, \varepsilon, L_{\alpha b})$ is an autotopy of $(A,+)$, i.e. $\alpha
L_b \alpha^{-1} y + z = L_{\alpha b} (y +z)$.     By $z=0$ we have $\alpha L_b \alpha^{-1} y = L_{\alpha b} y$.
Then the triple $(L_{\alpha b}, \varepsilon, L_{\alpha b})$ is a loop autotopy.

The  equality $\alpha L_b \alpha^{-1}= L_{\alpha b}$ means that  $\alpha b + y = \alpha(b + \alpha^{-1} y)$  for
all $b, y \in A$. If we change $y$ by $\alpha y$, then $\alpha b + \alpha y = \alpha(b +  y)$  for all $b, y \in
A$,  $\alpha \in Aut (A, +)$.

(iii) From equality $\alpha x + e(x) = x$ by $e(x) =0$ we have $\alpha x  = x$.
\end{proof}

\begin{remark}
Conditions   $[L_d , \alpha^{-1}] \in LM(A,+)$ for all $d\in A$ and  $\alpha \in Aut(A,+)$ are equivalent.
\end{remark}

\begin{corollary} \label{e_zero}
If $e(x) =0$  for all $x\in A$, then $\alpha \in Aut(A,+)$.
\end{corollary}
\begin{proof}
In this case equality (\ref{equat_LF_532}) takes the form  $(\alpha L_x \alpha^{-1}, \varepsilon, L_{\alpha
x})$. If autotopy of such form true in a loop, then $\alpha L_x \alpha^{-1} = L_{\alpha x}$, $\alpha L_x  =
L_{\alpha x}\alpha $.
\end{proof}

We recall   the following theorem (Theorem 17 \cite{SOH_99}).
\begin{theorem} \label{SOHA_F_QUAS_TH}
A group isotope $(Q,\cdot)$ with the form $x\cdot y = \alpha x + a + \beta y$ is a left F-quasigroup if and only
if $\beta$ is an automorphism of the group $(Q,+)$, $\beta$ commutes with $\alpha$ and $\alpha$ satisfies the
identity $\alpha (x+y) = x+ \alpha y - x + \alpha x$.
\end{theorem}

\begin{example} \label{DIOR_SCERB_1}
Dihedral group $(D_8,+)$ with the following Cayley table
\[ {\begin{array}{c|cccccccc}
+ & 0 & 1 & 2 & 3 & 4 & 5 & 6 & 7   \\
\hline
0 & 0 & 1 & 2  & 3 & 4 & 5 & 6 & 7     \\
1 & 1 & 0 & 3  & 2 & 6 & 7 & 4 & 5     \\
2 & 2 & 6 & 7  & 1 & 0 & 3 & 5 & 4     \\
3 & 3 & 4 & 5  & 0 & 1 & 2 & 7 & 6     \\
4 & 4 & 3 & 0  & 5 & 7 & 6 & 1 & 2     \\
5 & 5 & 7 & 6  & 4 & 3 & 0 & 2 & 1     \\
6 & 6 & 2 & 1  & 7 & 5 & 4 & 0 & 3     \\
7 & 7 & 5 & 4  & 6 & 2 & 1 & 3 & 0     \\
\end{array}}
\]
has endomorphism
\[ e= \left( {\begin{array}{cccccccc}
 0 & 1 & 2 & 3 & 4 & 5 & 6 & 7 \\
0 & 3 & 3 & 0  & 3 & 3 & 0 & 0 \\
\end{array}}\right), \qquad   e^{\, 2} =0,
\]
and permutation $\alpha =  (1\, 2) (4\, 5)$ such that $\alpha \notin Aut(D_8)$. Using this permutation and
taking into consideration Theorem \ref{SOHA_F_QUAS_TH} we may construct left F-quasigroups $(D_8, \cdot)$ and
$(D_8, \ast)$ with  the forms $x\cdot y = \alpha x + a + y$ and $x\ast y = \alpha x + a + \beta y$, where $\beta
=(15)(24)$. These quasigroups are right-linear group isotopes but they are not left linear quasigroups ($\alpha
\notin Aut(D_8, +)$). This example was constructed using Mace 4 \cite{MAC_CUNE_MACE}.
\end{example}

\begin{corollary} \label{LOOP_ISOTOPES__F_QUAS_LEFT_SPCIAL_LOOPS}
A left special loop $(Q, \oplus )$ is isotope of  a left F-quasigroup $(Q,\cdot)$ if and only if   $(Q, \oplus
)$ is isomorphic to the direct product of a group $(A, + )$ and a left S-loop $(B, \diamond)$.
\end{corollary}
\begin{proof}
By Theorem \ref{LOOP_ISOTOPES_LEFT_F_QUAS}  any left F-quasigroup  $(Q, \cdot )$  is LP-isotopic to a loop $(Q,
\oplus)$ which is the  direct product of a group $(A, + )$  and  a left  S-loop $(B, \diamond)$.

It is clear that any group is a left special loop. Any  left $S$-loop also is a special loop (\cite{ONOI_D}, p.
61). Therefore $(Q, \oplus )$ is a left special loop.

Converse. It is easy to see  that isotopic image of group $(A,+ )$ of the form $(\varepsilon, \psi,
\varepsilon)$, where $\psi\in Aut (A,+)$, is a left F-quasigroup.

From Theorem \ref{AT1} we have that that isotopic image of the loop $(B, \diamond)$ of the form $(\alpha,
\psi^{\diamond}, \varepsilon)$, where $\psi^{\diamond}$ is complete automorphism of $(B, \diamond)$, is a left
distributive quasigroup $(B, \circ)$. By Lemma  \ref{LEMMA_4_ED} (see also \cite{1a}) isotope of the form
$x\cdot y = x \cdot \psi ^{\circ} y$, where  $\psi ^{\circ}\in Aut (B, \circ)$, is a left F-quasigroup.

Therefore, among isotopic images of the left special loop $(Q, \oplus )$ there exists a left F-quasigroup.
\end{proof}
Corollary \ref{LOOP_ISOTOPES__F_QUAS_LEFT_SPCIAL_LOOPS} gives an answer to  Belousov 1a Problem  \cite{VD}.

\begin{corollary} \label{LOOP_ISOTOPES__F_QUAS_LEFT_M_LOOPS}
If $(Q, \ast )$ is a left M-loop which is isotopic to a left F-quasigroup $(Q,\cdot)$, then  $(Q, \ast )$ is
isotopic to the direct product of a group and LP-isotope of a left S-loop.
\end{corollary}
\begin{proof}
By Theorem \ref{LOOP_ISOTOPES_LEFT_F_QUAS}  any left F-quasigroup  $(Q, \cdot )$  is LP-isotopic to a loop $(Q,
\oplus)$ which is the  direct product of a group $(A, \oplus )$  and left  S-loop $(B, \diamond)$.

By Theorem \ref{BELOUSOV_LEFT M_LOOP} any loop which is isotopic to a left F-quasigroup is a left M-loop.

Up to isomorphism $(Q, \ast )$ is an LP-isotope of $(Q,\cdot)$. Then the loops $(Q, \ast )$ and $(Q, \oplus)$
are isotopic with an isotopy $(\alpha, \beta, \varepsilon)$. Moreover, they are LP-isotopic (\cite{1a}, Lemma
1.1).

From the proof of Lemma \ref{LEMMA_DIR_2} it follows that LP-isotopic image of a loop that is a direct product
of two subloops also is isomorphic to the direct product of some subloops.

By Albert Theorem (Theorem \ref{ALBERT_THEOREM}) LP-isotopic image of a group is a group.
\end{proof}

\subsection{F-quasigroups}

\begin{theorem} \label{LOOP_ISOTOPES__F_QUAS}
Any  F-quasigroup  $(Q, \cdot )$ is isotopic to the direct product of a group $(A, \oplus )\times (G, +) $ and a
commutative  Moufang loop $(K, \diamond)$, i.e.  $(Q, \cdot) \sim (A, \oplus) \times (G, +) \times (K,
\diamond). $
\end{theorem}
\begin{proof}
By Theorem  \ref{MAIN_LEFT_F}, Case 1  any  left F-quasigroup $(Q, \cdot)$ has the following structure $(Q,
\cdot) \cong (A, \circ) \times (B, \cdot),$ where $(A, \circ)$ is a quasigroup with a unique idempotent element;
$(B, \cdot)$ is isotope of a left distributive quasigroup $(B, \star)$, $x \cdot y = x \star \psi y$ for all $x,
y \in B$, $\psi \in Aut(B, \cdot)$, $\psi \in Aut(B, \star)$.

By Theorem  \ref{MAIN_LEFT_F}, Case 2 the quasigroup $(B, \cdot)$ has the following structure $(B, \cdot) \cong
(G, \circ) \times (K, \cdot),$ where $(G, \circ)$ is a quasigroup with a unique idempotent element; $(K, \cdot)$
is isotope of a right  distributive quasigroup $(K, \star)$, $x \cdot y = \varphi x \star  y$ for all $x, y \in
K$, $\varphi  \in Aut(K, \cdot)$, $\varphi \in Aut(K, \star)$.

By Theorem \ref{LOOP_ISOTOPES_LEFT_F_QUAS}, Case 1,  the quasigroup $(A, \circ)$ is a group isotope. By Theorem
\ref{LOOP_ISOTOPES_LEFT_F_QUAS}, Case 2,  the quasigroup $(G, \circ)$ is a group isotope.

In the quasigroup $(K, \cdot)$ the endomorphisms $e$ and $f$ are permutations of the set $K$ and by Theorem
\ref{LEFT_DISTRIB_Isotopic_to_F_QUAS} $(K, \cdot)$ is isotope of a distributive quasigroup. Then by  Belousov
Theorem (Theorem \ref{Belousov_Theorem}) quasigroup $(K, \cdot)$ is isotope of a CML $(K, \diamond)$.  Therefore
$(Q, \cdot) \sim (A, \oplus) \times (G, +) \times (K, \diamond). $
\end{proof}

\begin{theorem} \label{LOOP_ISOTOPES OF_F_QUAS}
Any loop $(Q, \ast )$  that is  isotopic to an F-quasigroup $(Q,\cdot)$ is isomorphic to the direct product of a
group and a Moufang loop \cite{kepka_05, kepka_07}.
\end{theorem}
\begin{proof}
By Theorem  \ref{LOOP_ISOTOPES__F_QUAS}  an F-quasigroup $(Q, \cdot )$ is isotopic to a loop $(Q,+) \cong (A,+)
\times (B,+)$ which  is the direct product of a group  and a commutative  Moufang loop. Then any left
translation $L$  of $(Q,+)$ it is possible to present as a pair $(L_1, L_2)$, where $L_1$ is a left translation
of the loop $(A,+)$, $L_2$ is a left translation of the loop $(B,+)$.

From Lemma \ref{NL1} it follows that any LP-isotope of the loop $(Q,+)$ is the direct product of its subloops.

By generalized Albert Theorem LP-isotope of a group is a group. Any LP-isotope of a commutative Moufang loop is
a Moufang loop \cite{VD}.
\end{proof}

\begin{corollary} \label{LOOP_ISOTOPES__F_QUAS_M_LOOPS}
If $(Q, \ast )$ is an  M-loop which is isotopic to an F-quasigroup, then  $(Q, \ast )$ is  a Moufang loop.
\end{corollary}
\begin{proof}
The proof follows from Theorem \ref{LOOP_ISOTOPES OF_F_QUAS}. It is well known that any group is a Moufang loop.
\end{proof}

\subsection{Left SM-quasigroups}

\begin{theorem} \label{LOOP_ISOTOPES_LEFT_SM_QUAS}
A  left SM-quasigroup $(Q, \cdot )$ is isotopic to the direct product of a group  $(A, \oplus )$   and a left
S-loop $(B, \diamond)$, i.e. $(Q, \cdot) \sim (A, \oplus) \times (B, \diamond). $
\end{theorem}
\begin{proof}
In many details the proof of this theorem repeats the proof of Theorem \ref{LOOP_ISOTOPES_LEFT_F_QUAS}.

By Theorem  \ref{MAIN_MIDDLE_F} any  left SM-quasigroup $(Q, \cdot)$ has the following structure $ (Q, \cdot)
\cong (A, \circ) \times (B, \cdot), $ where $(A, \circ)$ is a quasigroup with a unique idempotent element and
there exists a  number $m$ such that $|s^m(A, \circ)| =1$; $(B, \cdot)$ is an isotope of a left distributive
quasigroup $(B, \star)$, $x \cdot y =  s(x \star y)$ for all $x, y \in B$, $s \in Aut(B, \cdot)$, $s \in Aut(B,
\star)$.

By Corollary \ref{COROL_DIR_2}, if a quasigroup $Q$ is the direct product of quasigroups $A$ and $B$, then there
exists an isotopy $T=(T_1, T_2)$ of $Q$ such that  $Q\, T \cong AT_1 \times BT_2$ is a loop.

Therefore we have a possibility to divide our proof on two steps.

Step 1. Denote a unique idempotent element of $(A, \circ)$ by $0$. It is easy to check that $s^{\circ} \, 0 =
0$. Indeed, from $(s^{\circ})^{m} A = 0$ we have $(s^{\circ})^{m+1} A = s^{\circ} 0 = 0$.

(23)-parastrophe of $(A, \circ)$ is left F-quasigroup quasigroup $(A,\cdot)$ (Lemma  \ref{LEMMA_PARASTR}, (5))
such that $|e^m(A, \cdot)| =1$. Then $(A,\cdot)$ also has a unique idempotent element. By Theorem
\ref{LOOP_ISOTOPES_LEFT_F_QUAS} principal isotope of $(A,\cdot)$ is a group $(A,\oplus)$.

We shall use multiplication of isostrophies (Definition \ref{MULT_ISOS},  Corollary \ref{PERESTAN_ISOSTR} and
Lemma \ref{INVERSE_ISOSTROPHY}). (23)-parastrophe  image of group $(A,\oplus)$ coincides with its isotope of the
form $(I,\varepsilon, \varepsilon)$, where $x\oplus Ix = 0$ for all $x\in A$. Indeed, if $x\oplus y = z$, then
$x\oplus^{23} z = y$. But $y = Ix\oplus z$. Therefore $x\oplus^{23} z =  Ix\oplus z$, i.e. $(\oplus)((23),
\varepsilon) =  (\oplus) (\varepsilon, (I, \varepsilon, \varepsilon))$. Then $(\oplus) = (\oplus)
((23),(I,\varepsilon, \varepsilon))$, since $I^{2}= \varepsilon$.

 We
have
\[
\begin{array}{l}
(\oplus) = (\circ) ((23),\varepsilon)(\varepsilon, (\alpha,  \beta, \varepsilon)) = (\circ) ((23), (\alpha,
\beta, \varepsilon)),\\
 (\oplus) = (\oplus)((23),(I,\varepsilon, \varepsilon))  = (\circ) ((23), (\alpha,  \beta,
\varepsilon))((23),(I,\varepsilon, \varepsilon)) = \\ (\circ) (\varepsilon, (\alpha I, \varepsilon,  \beta)).
\end{array}
\]

Step 2. The proof of this step is similar to the proof of Step 2 from Theorem \ref{LOOP_ISOTOPES_LEFT_F_QUAS}
and we omit them.
\end{proof}

\subsection{Left E-quasigroups}

\begin{lemma} \label{LOOP_ISOTOPES_LEFT_E_QUAS}
A  left E-quasigroup $(Q, \cdot )$ is isotopic to the direct product of a left loop $(A, \oplus )$ with equality
$(\delta  x \oplus  x) \oplus (y \oplus z) =  (\delta x\oplus y) \oplus (x \oplus z)$, where $\delta$ is an
endomorphism of the loop  $(A, \oplus)$,  and a left S-loop $(B, \diamond)$, i.e. $(Q, \cdot) \sim (A, \oplus)
\times (B, \diamond). $
\end{lemma}

\begin{proof}
In some  details the  proof of Lemma \ref{LOOP_ISOTOPES_LEFT_E_QUAS} repeats the proof of Theorem
\ref{LOOP_ISOTOPES_LEFT_F_QUAS}. By Theorem \ref{MAIN_MIDDLE_F} any  left E-quasigroup $(Q, \cdot)$ has the
following structure $ (Q, \cdot) \cong (A, \circ) \times (B, \cdot), $ where $(A, \circ)$ is a quasigroup with a
unique idempotent element and there exists a  number $m$ such that $|f^m(A, \circ)| =1$; $(B, \cdot)$ is an
isotope of a left distributive quasigroup $(B, \star)$, $x \cdot y = f^{-1}(x) \star  y$ for all $x, y \in B$,
$f \in Aut(B, \cdot)$, $f \in Aut(B, \star)$.

By Corollary \ref{COROL_DIR_2}, if a quasigroup $Q$ is the direct product of quasigroups $A$ and $B$, then there
exists an isotopy $T=(T_1, T_2)$ of $Q$ such that  $Q\, T \cong AT_1 \times BT_2$ is a loop.

Therefore we have a possibility to divide our proof on two steps.

Step 1. We shall prove that $(A,\oplus)$ is a left loop. Denote a unique idempotent element of $(A, \circ)$ by
$0$. It is easy to check that $f^{\circ} \, 0 = 0$. Indeed, from $(f^{\circ})^{m} A = 0$ we have
$(f^{\circ})^{m+1} A = f^{\circ} 0 = 0$.

From left E-equality   $x \circ (y \circ z) = (f^{\circ}(x) \circ y) \circ (x \circ z)$ by $x=0$ we have $0
\circ (y \circ z) = (0\circ y) \circ (0 \circ z)$. Then $L_0 \in Aut(A, \circ)$.

Consider  isotope $(A,\oplus)$ of the quasigroup $(A, \circ)$: $x \oplus y = x \circ L^{-1}_0 y$. We notice that
 $(A,\oplus)$ is a left loop.   Indeed, $0 \oplus y = 0 \circ L^{-1}_0 y = y$.

Prove that $L_0 \in Aut (A, \oplus)$. From equality $L_0(x\circ y)= L_0 x \circ L_0 y $ we have $L_0(x\circ y)=
L_0(x\oplus L_0 y)$, $L_0 x \circ L_0 y = L_0 x \oplus L^2_0 y$, $L_0(x\oplus L_0 y)= L_0 x \oplus L^2_0 y$.

If we pass in  left E-equality to the operation $\oplus$, then we obtain $x \oplus (L_0 y \oplus L^2_0 z) =
(f^{\circ} x\oplus L_0 y) \oplus (L_0 x \oplus L^2_0 z)$. If we change $L_0 y$ by $y$, $L^2_0 z$ by $z$, then we
obtain
\begin{equation}
x \oplus (y \oplus z) =  (f^{\circ} x\oplus y) \oplus (L_0 x \oplus z). \label{equat_LE}
\end{equation}

We notice, $f^{\circ} x \circ x = x$. Then $f^{\circ} x \oplus L_0 x = x$.
 Moreover, from $f^{\circ}(x\circ y) = f^{\circ} x \circ f^{\circ} y$ we have $f^{\circ}(x\oplus L_0 y) = f^{\circ}(x) \oplus L_0 f^{\circ} (y)$. If $x=0$, then
 $f^{\circ} L_0(y) = L_0 f^{\circ}(y)$, $f^{\circ}(x\oplus L_0 y) = f^{\circ}(x) \oplus  f^{\circ} L_0 (y)$, $f^{\circ}$ is an endomorphism of the left loop
 $(A,\oplus)$.

We can rewrite equality (\ref{equat_LE}) in the following form
\begin{equation}
(f^{\circ} x \oplus L_0 x) \oplus (y \oplus z) =  (f^{\circ} x\oplus y) \oplus (L_0 x \oplus z).
\label{equat_LE_1}
\end{equation}

If we change in (\ref{equat_LE_1}) $x$ by $L^{-1}_x$, then we obtain
\begin{equation}
(f^{\circ} L^{-1}_0 x \oplus  x) \oplus (y \oplus z) =  (f^{\circ} L^{-1}_0 x\oplus y) \oplus (x \oplus z).
\label{equat_LE_21}
\end{equation}

If we denote the map $f^{\circ} L^{-1}_0 $ of the set $Q$ by $\delta$, then from (\ref{equat_LE_21})
 we have $(\delta x \oplus x)
\oplus (y \oplus z) =  (\delta x\oplus y) \oplus (x \oplus z)$. The map $\delta = f^{\circ} L^{-1}_0$ is an
endomorphism of the left loop $(A,\oplus)$ since $f^{\circ}$ is an endomorphism and $L^{-1}_0$  an automorphism
of $(A,\oplus)$.  We notice, $f^{\circ} L^{-1}_0 0 = 0.$

Step 2. From Theorems \ref{MAIN_MIDDLE_F} and  \ref{AT1} it follows that $$(B, \diamond) = (B, \cdot)(f,
\varepsilon, \varepsilon) ((R_a^{\star})^{ -1},  (L_a^{\star})^{-1}, \varepsilon)=  (B, \cdot)(f (R_a^{\star})^{
-1}, (L_a^{\star})^{-1}, \varepsilon)$$ is a left  S-loop.
\end{proof}

\begin{remark}
If we take $f^{\circ} a=0$, then from $f^{\circ} a \oplus L_0 a = a$ we have $L_0 a = a$. Thus from
(\ref{equat_LE_1}) we have $ a \oplus (y \oplus z) =  y \oplus (a \oplus z).$
\end{remark}

\begin{lemma} \label{LOOP_ISOTOPES_LEFT_E_QUAS_1}
A  left E-quasigroup $(Q, \cdot )$ is isotopic to the direct product of a  loop $(A, + )$ with equality $(\delta
x +  x) + (y + z) =  (\delta x + y) + (x + z)$, where $\delta$ is an endomorphism of the loop  $(A, +)$, and a
left S-loop $(B, \diamond)$, i.e. $(Q, \cdot) \sim (A, +) \times (B, \diamond). $
\end{lemma}
\begin{proof}
We pass from the operation $\oplus$ to operation $+$:  $x + y = R^{-1}_0 x \oplus y$, $x \oplus y = (x \oplus 0)
+ y $. Then $x + y = (x / 0) \oplus y$, where $x / y = z$ if and only if $z \oplus y = x$. We notice, $R^{-1}_0
\, 0 =0$, since  $R_0 0 =0$, $0 \oplus 0 =0$.

If we denote the map $R^{\oplus}_0$ by $\alpha$, then $x\oplus y = \alpha x + y$. We can rewrite
(\ref{equat_LE_21}) in terms of the loop operation + as follows
\begin{equation}
\alpha (\delta \alpha x  + x) + (\alpha y  + z) =  \alpha (\delta \alpha x + y) + (\alpha x  + z).
\label{equat_LF_11}
\end{equation}

Prove that $\alpha \delta = \delta \alpha$. Notice that $R^{\oplus}_y x = R^{\circ}_{L^{-1}_0 y} x $. Then
$R^{\oplus}_0  = R^{\circ}_0$. Thus \[
\begin{split}
& L_0 R^{\oplus}_0 x = L_0 R_0 x = 0\circ (x \circ 0) = (0\circ x) \circ 0 = R^{\oplus}_0 L_0 x, \\
& f^{\circ}R^{\oplus}_0 x = f^{\circ} (x\circ 0) = f^{\circ} x\circ 0 = R^{\oplus}_0 f^{\circ} x.
\end{split}
\]

Then $\delta$ is an endomorphism of the loop  $(A, +)$. Indeed, $\delta (x + y) = \delta (\alpha^{-1}x \oplus y)
= \delta \alpha^{-1} x \oplus \delta y = \alpha^{-1} \delta  x \oplus \delta y = \delta x + \delta y$.

Equality (\ref{equat_LF_11}) takes the form
\begin{equation}
\alpha (\delta x  + x) + (\alpha y  + z) =  \alpha (\delta  x + y) + (x  + z). \label{equat_LF_121}
\end{equation}
If we put  in  equality (\ref{equat_LF_121}) $x=y$, then $\alpha x = x$, $\alpha = \varepsilon$ and equality
(\ref{equat_LF_121}) takes the form
\begin{equation}
(\delta x  + x) + (y  + z) =  (\delta  x + y) + (x  + z). \label{equat_LF_122}
\end{equation}
\end{proof}

\begin{lemma}\label{LOOP_ISOTOPES_LEFT_E_QUAS_11}
If $\delta x  = 0$ for all $x\in A$, then $(A, + )$ is a commutative group.
\end{lemma}
\begin{proof}
If we put in equality (\ref{equat_LF_122}) $z=0$, then $x+y=y+x$. Therefore, from  $x + (y  + z) = y + (x + z)$
we have $(y  + z) + x   = y + (z + x)$.
\end{proof}

\begin{lemma} \label{LEFT_E_M_ZERO} There exists a number $m$ such that in the loop $(A, + )$ the chain
\begin{equation} \label{chain_1}
 (A, +) \supset \delta(A,+) \supset \delta^{\, 2}(A, +)
\supset \dots  \supset \delta^{\, m}(A,+) = (0,+)
\end{equation}
is stabilized on the element $0$.
\end{lemma}
\begin{proof}
 From  Theorem \ref{MAIN_MIDDLE_F} it follows that $(A, \circ)$ is a left E-quasigroup
 with a unique idempotent element $0$ such that the chain
\begin{equation} \label{chain_2}
 (A,\circ) \supset f^{\circ}(A, \circ) \supset (f^{\circ})^2(A, \circ)
\supset \dots  \supset (f^{\circ})^{m}(A, \circ) = (0, \circ)
\end{equation}
is stabilized on the element $0$.

From Lemmas \ref{LOOP_ISOTOPES_LEFT_E_QUAS} and \ref{LOOP_ISOTOPES_LEFT_E_QUAS_1} it follows that  $(A,+)  =
(A,\circ)T$, where  isotopy $T$ has the form $(R^{-1}_0, L^{-1}_0, \varepsilon)$. Since $0 \in (f^{\circ})^i(A,
\circ)$, then $((f^{\circ})^i(A, \circ))T = (f^{\circ})^i(A, +) $ is a subloop of the loop $(A,+)$  (Lemma
\ref{LP_AND_subquas}) for all suitable values of $i$.

Thus we obtain that the isotopic image of  chain (\ref{chain_2}) is the following chain
\begin{equation} \label{chain_3}
 (A,+) \supset f^{\circ}(A, +) \supset (f^{\circ})^2(A, +)
\supset \dots  \supset (f^{\circ})^{m}(A, +) = (0,+).
\end{equation}

We recall,  $\delta = f^{\circ} L^{-1}_0$ and $f^{\circ} L^{-1}_0 = L^{-1}_0 f^{\circ}$ (Lemma
\ref{LOOP_ISOTOPES_LEFT_E_QUAS}). Then $\delta^i = (f^{\circ})^i L^{-i}_0$ and $\delta^i (A,+) =  (f^{\circ})^i
L^{-i}_0 (A,+)$. It is clear that $L^{-i}_0$ is a bijection of the set $A$ for all suitable values of $i$.

Thus we can establish the following  bijection:  $(f^{\circ})^i  (A,+)  \leftrightarrow \delta^i (A,+)$. Then
$\delta^i (A,+) \supset \delta^{i+1} (A,+)$,  since  $(f^{\circ})^i  (A,+) \supset (f^{\circ})^{i+1} (A,+)$.
Therefore $(f^{\circ})^m (A,+) \leftrightarrow \delta^m (A,+)$, $\delta^m (A,+) = (0,+)$.
\end{proof}

\begin{lemma} \label{LOOP_ISOTOPES_LEFT_E_QUAS_i}
The loop $(A, + )$ is a commutative group.
\end{lemma}
\begin{proof}
From Lemma \ref{LEFT_E_M_ZERO} it follows that in $(A, + )$ there exists a number $m$ such that  $\delta^{\, m}
x = 0$ for all $x\in A$. We have used Prover's 9 help \cite{MAC_CUNE_PROV}.
 From (\ref{equat_LF_122}) by $y=0$ we obtain
\begin{equation}
(\delta x  + x) + y =  \delta  x  + (x  + y). \label{equat_LF_55}
\end{equation}

If we change in equality (\ref{equat_LF_55}) $y$ by $y+z$, then we obtain
\begin{equation}
(\delta x  + x) + (y+z) =  \delta  x  + (x  + (y+z)). \label{equat_LF_55_DOP}
\end{equation}

From (\ref{equat_LF_122}) by $z=0$ using (\ref{equat_LF_55}) we have
\begin{equation}
(\delta x  + y) + x =  \delta  x  + (x  + y). \label{equat_LF_56}
\end{equation}

If we change in (\ref{equat_LF_56}) $y$ by $\delta x \backslash y$, then
\begin{equation}
 (\delta  x + (\delta  x\backslash  y)) + x = \delta  x +(x  + (\delta  x \backslash y)). \label{equat_LF_280}
\end{equation}

But $(\delta  x + (\delta  x\backslash  y))  = y $ (Definition  \ref{quasigr_as_algebra}, equality (\ref{(1)})).
Therefore
\begin{equation}
\delta  x + (x  + (\delta  x \backslash y)) = y + x. \label{equat_LF_281}
\end{equation}

If we change in  (\ref{equat_LF_122}) $x$ by  $\delta^{\, m-1}  x$, then,    using condition $\delta^{\, m} x =
0$, we have
\begin{equation}
\delta^{\, m-1} x  + (y + z) =  y + (\delta^{\, m-1}  x  + z). \label{equat_LF_59_i}
\end{equation}

\textbf{Begin Cycle}

If we change in  equality (\ref{equat_LF_281}) the element $x$ by the element $\delta^{\, m-2} x$, then we have
\begin{equation}
\delta^{\, m-1}  x + (\delta^{\, m-2}x  + (\delta^{\, m-1} x \backslash y)) = y + \delta^{\, m-2} x.
\label{equat_LF_281_II}
\end{equation}

If we change in  (\ref{equat_LF_59_i}) $z$ by  $\delta^{\, m-1}  x \backslash z$, then, using Definition
\ref{quasigr_as_algebra}, equality (\ref{(1)}), we obtain
\begin{equation}
\delta^{\, m-1} x  + (y + (\delta^{\, m-1}  x \backslash z)) =  y +  z. \label{equat_LF_340_i}
\end{equation}

If we change in  (\ref{equat_LF_340_i}) $y $ by  $\delta^{\, m-2} x$,  $z$ by $y$ and compare
(\ref{equat_LF_340_i}) with (\ref{equat_LF_281_II}), then we obtain
\begin{equation}
\delta^{\, m-2} x  + y =  y + \delta^{\, m-2} x. \label{equat_LF_387}
\end{equation}

We have $\delta^{\, m-1} (A) \subseteq \delta^{\, m-2} (A)$ since $\delta$ is an endomorphism of the loop
$(A,+)$. Notice, from equalities (\ref{equat_LF_59_i}) and (\ref{equat_LF_387}) it follows that $\delta^{\, m-1}
(A) \subseteq N_l(A)$.

From equality  $(y / \delta^{\, m-2} x) + \delta^{\, m-2} x = y$ (Definition  \ref{quasigr_as_algebra}, equality
(\ref{(2)})) using commutativity (\ref{equat_LF_387}) we obtain
\begin{equation}
\delta^{\, m-2} x  + (y \slash \delta^{\, m-2} x) =  y. \label{equat_LF_588}
\end{equation}

From equality  (\ref{equat_LF_588}) and  definition of the operation $\backslash$ we have
\begin{equation}
\delta^{\, m-2} x  \backslash y =  y \slash \delta^{\, m-2} x. \label{equat_LF_622}
\end{equation}

If we change in  (\ref{equat_LF_59_i}) $y + z$ by  $y$, then $y$ pass in $y\slash z$ and we have
\begin{equation}
\delta^{\, m-1} x  + y =  (y \slash z) + (\delta^{\, m-1} x + z). \label{equat_LF_342}
\end{equation}

Applying to (\ref{equat_LF_342}) the operation $\slash$ we have
\begin{equation}
(\delta^{\, m-1} x  + y) \slash    (\delta^{\, m-1} x + z) =  (y \slash z). \label{equat_LF_1186}
\end{equation}

 Write equality (\ref{equat_LF_122}) in the form
\begin{equation}
 (\delta  x + y) \backslash ((\delta x  + x)  + (y  + z)) = x  + z. \label{equat_LF_41}
\end{equation}

From (\ref{equat_LF_41}) using (\ref{equat_LF_55_DOP}) we obtain
\begin{equation}
 (\delta  x + y) \backslash (\delta x  + (x  + (y  + z))) = x  + z. \label{equat_LF_73}
\end{equation}

From equality  (\ref{equat_LF_73}) using  (\ref{equat_LF_622})  we have
\begin{equation}
(\delta x  + (x  + (y  + z))) \slash  (\delta  x + y) = x  + z. \label{equat_LF_645}
\end{equation}

If we change in  equality  (\ref{equat_LF_645}) $x$ by $\delta^{\, m - 2}$, then we obtain
\begin{equation}
(\delta^{\, m-1} x  + (\delta^{\, m - 2} x  + (y  + z))) \slash  (\delta^{\, m - 1}  x + y) = \delta^{\, m - 2}
x + z. \label{equat_LF_645_i}
\end{equation}

Using equality (\ref{equat_LF_1186}) in  equality (\ref{equat_LF_645_i})  we have
\begin{equation}
(\delta^{\, m - 2} x  + (y  + z)) \slash   y = \delta^{\, m - 2} x  + z. \label{equat_LF_1226}
\end{equation}
Therefore
\begin{equation}
 \delta^{\, m - 2} x  + (y  + z) = (\delta^{\, m - 2} x  + z)+ y,
\end{equation}
and
\begin{equation} \label{EQULITY_74_LEFT_E}
 \delta^{\, m-2} x  + (y + z) =  y +
(\delta^{\, m-2}  x  + z).
\end{equation}
\textbf{End Cycle}

Therefore we can change equality (\ref{equat_LF_59_i}) by the equality (\ref{EQULITY_74_LEFT_E}) and start new
step of the cycle.

After $m$ steps we obtain that in the loop $(A,+)$ the equality  $ x  + (y + z) =  y + ( x  + z)$ is fulfilled,
i.e. $(A,+)$ is an abelian group. If $m=\infty$ then we can use arguments similar to the  arguments from the
proof of Theorem \ref{LOOP_ISOTOPES_LEFT_F_QUAS}.
\end{proof}

\begin{theorem} \label{LOOP_ISOTOPES_LEFT_E_QUAS_TH}
1. A left E-quasigroup $(Q, \cdot )$ is isotopic to the direct product of an abelian group $(A, + )$   and a
left S-loop $(B, \diamond)$, i.e.  $(Q, \cdot) \sim (A, +) \times (B, \diamond). $

2. A  right E-quasigroup $(Q, \cdot )$ is isotopic to the direct product of an abelian  group $(A, +)$ and a
right S-loop $(B, \diamond)$, i.e.  $(Q, \cdot) \sim (A,+) \times (B, \diamond). $
\end{theorem}
\begin{proof}
1. The proof follows from Lemmas \ref{LOOP_ISOTOPES_LEFT_E_QUAS_1}  and \ref{LOOP_ISOTOPES_LEFT_E_QUAS_11}.
\end{proof}
Theorem \ref{LOOP_ISOTOPES_LEFT_E_QUAS_TH} gives  an answer to Kinyon-Phillips problems (\cite{Kin_PHIL_04},
Problem 2.8, (1)).

\begin{corollary} \label{LOOP_ISOTOPES_LEFT_FESM_QUAS}
A left FESM-quasigroup $(Q, \cdot )$ is isotopic to the direct product of an abelian group  $(A, \oplus )$ and a
left S-loop $(B, \diamond)$.
\end{corollary}
\begin{proof}
 We can use Theorem \ref{LOOP_ISOTOPES_LEFT_E_QUAS_TH}.
\end{proof}
Corollary \ref{LOOP_ISOTOPES_LEFT_FESM_QUAS} gives  an  answer to Kinyon-Phillips problem (\cite{Kin_PHIL_04},
Problem 2.8, (2)).

We hope in a forthcoming paper we shall discuss a generalization of Murdoch theorems about the structure of
finite  binary and $n$-ary medial quasigroups \cite{4, SC05} on infinite case and  medial groupoids.

\footnotesize{
 }

\noindent \footnotesize{Institute of Mathematics
         and Computer Science\\
         Academy of Sciences of Republic Moldova\\
         str. Academiei, 5  \\
MD-2028, Chisinau\\
Moldova\\
e-mail: scerb@math.md}

\end{document}